\documentclass[12pt]{amsart}
\usepackage{amsmath,amssymb,mathrsfs,amsthm,amsfonts,stmaryrd}
\SetSymbolFont{stmry}{bold}{U}{stmry}{m}{n}
\usepackage[top=1in,bottom=1in, left=1 in, right= 1 in, includehead,includefoot]{geometry}

\usepackage{hyperref}

\usepackage{thmtools}

\usepackage{graphicx}
\usepackage{enumerate}
\usepackage[sort&compress,numbers]{natbib}
\hypersetup{
     colorlinks=true,
     linkcolor=blue,
     filecolor=black,
     citecolor = blue,      
     urlcolor=blue,
     }

\usepackage{tikz}
\usetikzlibrary{decorations.markings}
\usetikzlibrary{arrows}
\usetikzlibrary{calc}

\theoremstyle{plain}
\newtheorem{theorem}{Theorem}[section]
\newtheorem{lemma}[theorem]{Lemma}
\newtheorem{proposition}[theorem]{Proposition}
\newtheorem{assumption}[theorem]{Assumption}

\theoremstyle{definition}

\newtheorem{remark}[theorem]{Remark}

\graphicspath{{Images/}}

\newcommand{\EE}{\mathbb{E}}
\newcommand{\C}{\mathbb{C}}
\newcommand{\N}{\mathbb{N}}

\newcommand{\R}{\mathbb{R}}

\newcommand{\Z}{\mathbb{Z}}

\newcommand{\ind}{\mathbf{1}}

\newcommand{\equh}{\begin{equation}}
\newcommand{\eque}{\end{equation}}

\def\fddto{\xrightarrow{\textit{f.d.d.}}}

\def\inddd#1{{\ind}_{\left\{#1\right\}}} 
\newcommand{\proba}{\mathbb P}
\newcommand{\esp}{{\mathbb E}}

\newcommand{\inv}{^{-1}}

\newcommand{\calD}{{\mathcal D}}

\newcommand{\B}{{\mathbb B}}

\def\summ#1#2#3{\sum_{#1 = #2}^{#3}}
\def\prodd#1#2#3{\prod_{#1 = #2}^{#3}}
\def\sif#1#2{\sum_{#1=#2}^\infty}

\newcommand{\eqd}{\stackrel{d}{=}}

\def\topp#1{^{(#1)}}

\def\abs#1{\left|#1\right|}

\def\ccbb#1{\left\{#1\right\}} 

\def\pp#1{\left(#1\right)}
\newcommand{\bpp}[1]{\Big(#1\Big)}
\def\spp#1{(#1)}

\def\floor#1{\left\lfloor #1 \right\rfloor}

\def\ceil#1{\left\lceil #1 \right\rceil}

\def\vv#1{{\boldsymbol #1}}

\newcommand{\vvr}{{\vv r}}
\newcommand{\vvx}{{\vv x}}
\newcommand{\vvy}{{\vv y}}
\newcommand{\vvz}{{\vv z}}
\newcommand{\vvc}{{\vv c}}

\def\qmand{\quad\mbox{ and }\quad}

\def\qmwith{\quad\mbox{ with }\quad}
\def\mfa{\mbox{ for all }}

\def\wt#1{\widetilde{#1}}

\def\limn{\lim_{n\to\infty}}

\renewcommand{\i}{{\mathrm i}}
\renewcommand{\d}{{\mathrm d}}
\newcommand{\rmRe}{{\rm Re}}
\newcommand{\rmRes}{{\rm Res}}
\newcommand{\supp}{{\rm supp}}

\newcommand{\cequiv}{\leftrightsquigarrow}
\newcommand{\ddfrac}[2]{\frac{\displaystyle #1}{\displaystyle #2}}

\numberwithin{equation}{section}

\newcommand\thankssymb[1]{*}
\begin{document}\sloppy
\author[Himwich]{Zoe Himwich}
\author[Wang]{Yizao Wang}
\address{Zoe Himwich: Department of Statistics, University of California, Berkeley, Berkeley, CA 94705, USA}
\email{himwich@berkeley.edu}
\address{Yizao Wang: Department of Mathematical Sciences, University of Cincinnati, 2815 Commons Way, Cincinnati, OH, 45221-0025, USA}
\email{yizao.wang@uc.edu}
\date{\today}

\title[The stationary measure of the open KPZ equation in the shock region]{The Laplace transform of the stationary measure of the open KPZ equation in the shock region}

\begin{abstract}
The stationary increment measure of the open Kardar--Parisi--Zhang equation on a bounded interval can be represented as the independent sum of a Brownian motion and another process defined via an explicit Radon--Nikodym derivative with respect to the measure induced by a Brownian motion. In the fan region, its finite-dimensional Laplace transform was previously expressed in terms of the continuous dual Hahn process. The main contribution here consists of computing the finite-dimensional Laplace transform for suitable coefficients in the Laplace transform in the shock region by an analytic continuation argument. As an application, we provide a direct proof of the  convergence of the finite-dimensional distributions of the increment height function of open WASEP under the stationary measure in the shock region, by computing the convergence of finite-dimensional Laplace transform without relating to the open KPZ equation.
\end{abstract}\maketitle
  \tableofcontents

\section{Introduction}
The open Kardar--Parisi--Zhang (KPZ) equation is a fundamental model for non-equilibrium systems in mathematical physics \citep{corwin18open,corwin22some}. 
The open KPZ equation on $[0,1]$ for $t>0$ is given by
\[
\partial_t H_{u,v}(t,x)
=
\frac12\partial_x^2H_{u,v}(t,x)
+\frac12\bigl(\partial_xH_{u,v}(t,x)\bigr)^2
+\xi(t,x),
\]
where $\xi$ is space-time white noise. The Neumann boundary conditions are given by $\partial_xH_{u,v}(t,x)|_{x=0}=u$ and
$\partial_xH_{u,v}(t,x)|_{x=1}=-v$ for $u,v\in\mathbb R$. Solutions to the open KPZ equation are defined via the
Hopf--Cole transform (see \citep[Section~1]{corwin24stationary} for the
precise formulation).

A major milestone in the studies of the open KPZ equation is the characterization of its stationary measure. 
Strictly speaking, the open KPZ equation does not have a stationary measure, and by stationary measure we are referring to the stationary measure of the {\em increment process} (or the stationary increment measure), 
\[
\pp{H_{u,v}(t,x)-H_{u,v}(t,0)}_{x\in[0,1]}.
\]
Namely, the stationary increment measure is the law on $C[0,1]$ such that if the above increment process starts from this law, then its law remains the same at every later time  $t>0$. We will continue to use `stationary measure' (instead of `the stationary measure of the increment process') throughout this paper for the sake of simplicity, and this is also the convention in most of the literature.

The stationary measure of the open KPZ equation is a probability  measure on $C[0,1]$. It is convenient to consider the corresponding stochastic process, which we let $(\eta\topp{u,v}_t)_{t\in[0,1]}$ denote  throughout. It is known by now that we have
\equh\label{eq:stationary open KPZ}
\pp{\eta\topp{u,v}_t }_{t\in[0,1]}\eqd \pp{\frac1{\sqrt 2}\B_t+\zeta\topp{u,v}_t}_{t\in[0,1]},
\eque
where $\B$ is a standard Brownian motion, and $\zeta\topp{u,v}$ is another stochastic process  in $C[0,1]$ to be described in a moment, and the two processes are independent. We let $\mathsf P_{u,v}$ denote the law of $\zeta\topp{u,v}$, a probability measure on $C[0,1]$. It can be defined, for all $u,v\in\R$, via the Radon--Nikodym derivative as 
\begin{equation}\label{eq:zeta-RN}
\frac{\d \mathsf P_{u,v}}{\d \mathbb{P}_{B/\sqrt2}}(\omega)
=
\frac{1}{\mathfrak K_{u,v}}
e^{-2v\omega_1}
\left(\int_0^1e^{-2\omega_t}\d t\right)^{-(u+v)}, \omega = (\omega_t)_{t\in[0,1]}\in C[0,1],
\end{equation}
where $\proba_{B/\sqrt 2}$ is the probability measure on $C[0,1]$ induced by a standard Brownian motion scaled by $1/\sqrt 2$, and $\mathfrak K_{u,v}$ is a normalizing constant. Here $B$ denotes a generic standard Brownian motion, to distinguish it
from $\B$ in \eqref{eq:stationary open KPZ}.

Several different approaches have led to the characterization of $\eta\topp{u,v}$. 
The first breakthrough on the stationary measure of the open KPZ equation was made by exploiting its connection to the asymmetric simple exclusion process with open boundaries in the weakly asymmetric regime (open WASEP). \citet{corwin18open} first established the convergence of the open WASEP dynamics to the open KPZ equation. Building upon this connection, \citet{corwin24stationary} provided a first characterization of the stationary measure of the open KPZ equation via an explicit Laplace transform formula of the finite-dimensional distributions in the fan region $u+v>0$.  
Inspired by their work, \citet{barraquand22steady} derived formally \eqref{eq:zeta-RN} using Liouville quantum mechanics, and the proof for the formula \eqref{eq:zeta-RN} in the fan region following the work of Corwin and Knizel was later completed  by \citet{bryc23markov} and \citet{bryc22markov} for $u+v>0, \min\{u,v\}>-1$. More background is recalled in Section \ref{sec:open WASEP} below. Later, \citet{barraquand24stationary} proved the characterization \eqref{eq:zeta-RN} by exploiting the connection between the log-gamma polymer on a strip and the open KPZ equation, assuming the convergence from the former to the latter (which remains a conjecture). 
More recently, \citet{dunlap25invariant} proved that the law of $\eta\topp{u,v}$ is invariant for the increment process of the open KPZ equation for all $u,v\in\R$. Together with the uniqueness established by  \citet{parekh25ergodicity}, this identifies the law of $\eta\topp{u,v}$ as the unique stationary measure (see also a related result by \citet{knizel22strong}). At the same time as our paper was written, \citet{schmid26sharp,schmid26stationary} established independently another proof of convergence of the stationary measure of the open WASEP, by exploiting a newly developed two-layer representation due to \citet{bryc25stationary} (see Remark \ref{rem:two-layer} for more details). It is worth pointing out that the approach by \citet{dunlap25invariant} is of stochastic-analytic nature and does not exploit the convergence of any integrable model to the open KPZ equation as all the other known approaches do.

\subsection{Finite-dimensional Laplace transform}

The object of interest here is the Laplace transform of the finite-dimensional distributions of $\mathsf P_{u,v}$ (and we shall simply say finite-dimensional Laplace transform in the sequel) at points $0=x_0<x_1<\cdots<x_d=1$ with coefficients $c_1,\dots,c_d>0$:
\[
 \int_{C[0,1]} \exp\pp{-\summ k1d c_k\omega_{x_k}} \mathsf P_{u,v}(\d\omega)=\frac{\wt\Psi_{u,v}(\vvx,\vvc)}{\wt K_{u,v}},
\] 
where we define
\equh\label{eq:wt Psi def}
\wt\Psi_{u,v}(\vvx,\vvc) :=\esp\pp{ \exp\left(-\frac1{\sqrt2}\summ k1d c_k\B_{x_k}-\sqrt2v\B_1\right)\left(\int_0^1e^{-\sqrt2\B_t}\d t\right)^{-(u+v)}},
\eque
and $\wt K_{u,v}$ similarly (see \eqref{eq:wt K def}). Consider the simplest case $u+v>0$. Set $s_j:=\sum_{\ell=j}^d c_\ell$ and assume
\[
-2u<-s_1<\cdots<-s_d<2v.
\]
Then we have
\equh\label{eq:wt Psi}
\wt\Psi_{u,v}(\vvx,\vvc)
=
\int_{(0,\infty)^d}
\exp\pp{
-\frac14\sum_{k=1}^d
(x_k-x_{k-1})y_k
}
\mathfrak p_{-s_1}^{u,v,c}(y_1)
\prod_{k=1}^{d-1}
\mathfrak p_{-s_k,-s_{k+1}}^{u,v,c,c}
(y_k,y_{k+1})
\d\vvy.
\eque
The density functions $\mathfrak p_s^{u,v,c}$ and $\mathfrak p_{s,t}^{u,v,c,c}$ correspond to the continuous dual Hahn measures and their expressions can be found in Section \ref{sec:limiting measures}.

The integral on the right-hand side of \eqref{eq:wt Psi} was first obtained as the finite-dimensional Laplace transform of the stationary measure $\zeta\topp{u,v}$ after normalization (without having obtained its Radon--Nikodym representation in \eqref{eq:zeta-RN}) by \citet{corwin24stationary}. \citet{bryc23markov} recognized the right-hand side of \eqref{eq:wt Psi}, after normalization, as the Laplace transform of the increment process of a Markov process constructed from the Yakubovich heat kernel \citep{yakubovich11heat}. \citet{bryc22markov} then proved that the law of the increment process has  the Radon--Nikodym representation in \eqref{eq:zeta-RN}, making rigorous a physicist's derivation of \citet{barraquand22steady}, yielding the first rigorous proof of \eqref{eq:wt Psi}. We take \eqref{eq:wt Psi} as the starting point of our analysis, and we recall that this is in the fan region only.

Although $\wt\Psi_{u,v}$ defined in \eqref{eq:wt Psi def} is entire in $(u,v)$, we emphasize that we initially use \eqref{eq:wt Psi} under the conditions stated above including in particular $u+v>0$ (corresponding to the fan region). The expressions for $\mathfrak p_{-s_1}^{u,v,c}$ and $\mathfrak p_{-s_k,-s_{k+1}}^{u,v,c,c}$ involve Gamma functions and have poles. Hence the right-hand side of \eqref{eq:wt Psi} takes a different form when extended to the shock region. 

Our main contribution is to provide explicit expressions of $\wt\Psi_{u,v}(\vvx,\vvc)$ for parameter values  $(u,v)$ in the shock region, by an analytic continuation argument and considering $(u,v)\in\C^2$ with $\rmRe(u)+\rmRe(v)<0$. This argument is applied  in an iterative manner, making the formula much more involved. 
Under certain assumptions on $u,v,\vvc$ (see Proposition \ref{prop:path decomp} below), the expression takes the form
\equh\label{eq:limit path decomp}
\wt \Psi_{u,v}(\vvx,\vvc) = H_\emptyset(u)+\summ\ell1j\sum_{\vvr\in\mathbb S_\ell\topp i}H_\vvr(u) + \summ\ell{j+1}d\sum_{\vvr\in\mathbb S_\ell\topp{i-1}}H_\vvr(u),
\eque
where $H_\emptyset(u)$ is exactly the right-hand side of \eqref{eq:wt Psi}, $i\in\Z_{\ge0},j\in\{0,\dots,d\}$ are such that $s_{j+1,i}/2<\rmRe(u)<s_{j,i}/2$ (see Proposition \ref{prop:path decomp} for the definition of $s_{j,i}$), $\mathbb S_\ell\topp i = \{(r_1,\dots,r_\ell)\in\Z_{\ge0}^\ell: r_1\le \cdots\le r_\ell\le i\}$, and for each $\vvr =(r_1,\dots,r_\ell)$,  $H_\vvr(u)$ is a $(d-\ell)$-multiple integral over $(0,\infty)^{d-\ell}$ with respect to certain auxiliary measures. When $i=j=0$, the two double summations are understood as zero. For \eqref{eq:path decomp} to hold, the restriction relating $v$ and $\vvc$ is $-c_d<2\rmRe(v)$ for complex-valued $v$ (i.e., for fixed $v$ we first take $c_d$ large enough). Then, an analytic continuation argument is applied to $\wt\Psi_{u,v}(\vvx,\vvc)$ as a function of $u$. The extra $H_\vvr(u)$ terms compared to the right-hand side of \eqref{eq:wt Psi} are the residue terms collected during the analytic continuation.

We refer to \eqref{eq:limit path decomp} as a path decomposition of $\wt\Psi_{u,v}(\vvx,\vvc)$, interpreting the domain of each integral as a certain path space (corresponding to a trajectory with the first $\ell$ steps at certain deterministic atoms, and then the remaining $d-\ell$ steps in $(0,\infty)$). 
Moreover, we restrict ourselves to a range of parameters $u,v$ in the shock region that is sufficient for our application. Namely we shall assume $(u,v)\in\calD$ (see \eqref{eq:calD} below), which is dense in $\C^2$, and for those values in $\C^2\setminus\calD$ we can apply an approximation argument (see Remark \ref{rem:D approx}). Then, for each $(u,v)\in\calD\cap\R^2$ fixed, our formula requires further $\vvc$ to be in a subregion of $\R^d$. The values of the Laplace transform on this subregion determine the finite-dimensional distributions of $\mathsf P_{u,v}$ and hence suffices for our purposes. Extending the formula to other values of $\vvc$ would require keeping track of further pole crossings during the analytic continuation, which we choose not to pursue here.

Analytic continuation is not new in the study of the stationary measure of the open KPZ equation. When studying the fan region, such an argument has been applied to analyze the corresponding normalizing constant $\wt K_{u,v}$ there \citep[Section 3.3.2]{bryc23markov}; more precisely, the continuation of $\wt K_{u,v}$ as a function of $u,v$ concerns deforming a single contour. This argument serves as the starting point of our proof. Here, the argument is more involved because it involves multiple contour integrals. 
\begin{remark}
Related analytic-continuation arguments have recently appeared in the study of the stationary measure of the log-gamma polymers on a strip. \citet{barraquand26integral} showed that the finite-dimensional two-layer Whittaker formula admits analytic continuation in the boundary parameters, but made the residue corrections explicit only in the one-dimensional case in an application to the KPZ growth rate. Subsequently, \citet{bryc26stationary} gave an explicit residue expansion for the normalizing constant, including cases in which poles coalesce. The analysis of our $\wt K_{u,v}$ is a direct analogue of this one-contour problem. 
Neither work gives an explicit residue expansion for the finite-dimensional Laplace transform analogous to the one derived here.
\end{remark}

\subsection{Convergence of the stationary measure of open WASEP}
\label{sec:open WASEP}

Open ASEP has been an extensively investigated model in integrable probability. Its relatively explicit structure makes it possible to make many precise calculations, and hence provides access to the open KPZ equation.  
It has been known that the scaling limits of the height function under the stationary measures of the open ASEP, when the parameters are suitably chosen, may lead to various different limiting processes \citep{bryc19limit,bryc23asymmetric,corwin24stationary,wang25asymmetric}. 

The stationary measure of the open KPZ equation can be related to the stationary measure of the open ASEP in the so-called weakly asymmetric regime, also known as the open WASEP. \citet{corwin18open} established the convergence of the open WASEP dynamics to the open KPZ equation, and their setup was later extended by \citet{parekh19KPZ}. Building on this result, \citet{corwin24stationary} proved that the laws of the rescaled height functions of the open WASEP under its stationary measure are tight and that every subsequential limit is a stationary measure for the open KPZ equation. Subsequently, \citet{parekh25ergodicity} established uniqueness of this stationary measure for all $u,v\in\mathbb R$, so the entire sequence of rescaled height functions converges in distribution to it. This is in a sense a soft argument as identifying the limit through uniqueness in this way does not require an explicit formula for its law. More recently, \citet{yang25KPZ} extended the convergence of the dynamics \citep{corwin18open} and the stationary measures \citep{corwin24stationary} to models under significantly relaxed assumptions.

Next, in order to obtain a precise description of the law of $\eta\topp{u,v}$, one could exploit further the integrable structure of the open WASEP and analyze its asymptotics. 
In short, by a probabilistic version of Derrida's ansatz \citep{derrida93exact} due to \citet{bryc17asymmetric}, one can express the probability generating function of the stationary measure of the open ASEP in the fan region in terms of the so-called Askey--Wilson measures \citep{bryc10askey}. This method makes it possible to carry out explicit and refined calculations of the finite-dimensional Laplace transforms of the open ASEP height function under the stationary measure, as first demonstrated by \citet{bryc19limit} for the open ASEP.  For the open WASEP in the so-called {\em fan region}, \citet{corwin24stationary} applied this method and computed explicitly the finite-dimensional Laplace transform for some choices of $u,v,\vvx,\vvc$, and their expressions were later related to $\mathsf P_{u,v}$ by \citet{bryc23markov} and \citet{bryc22markov} in the fan region. A formal derivation also appeared in \citet{barraquand22steady}.

Thanks to our first result on the Laplace transform of $\mathsf P_{u,v}$ (and hence also of $\eta\topp{u,v}$) in the shock region, we can now compute the limiting Laplace transforms of the open WASEP directly, without relying on its connection to the open KPZ equation (the aforementioned soft argument). 
To complete this limiting Laplace transform program for the shock region, there are two steps. 
\begin{enumerate}[(i)]

\item First, for the pre-limit open ASEP with each $n\in\N$, one needs to have a corresponding Askey--Wilson representation of the finite-dimensional Laplace transform in the shock region. This turns out to be doable: the corresponding Askey--Wilson measures are no longer probability measures but now signed measures, but the Laplace transform can still be expressed as multiple integrals (w.r.t.~signed measures) \citep{wang24askey}. The complication here in the shock region is that the involved signed measures often have atoms and hence their analysis becomes much more complicated than in the fan region.

\item Second, one needs to compute the limiting finite-dimensional Laplace transforms of the open WASEP, and identify them with those of $\eta\topp{u,v}$. This computation is by now a standard, though lengthy, calculation, and the identification uses the finite-dimensional Laplace transform of $\mathsf P_{u,v}$  developed in the main contribution here. 
\end{enumerate}

The second part of the paper (Section \ref{sec:convergence}) is devoted to a direct proof of the convergence of the finite-dimensional distributions of the stationary measure of the open WASEP, by showing the convergence of the corresponding finite-dimensional Laplace transforms. This convergence alone is no longer new, as it now can be obtained indirectly by combining the convergence and uniqueness results \citep{corwin18open,corwin24stationary,parekh25ergodicity} and the explicit characterization by \citet{dunlap25invariant}. A remarkable discovery here is that for the pre-limit model, the finite-dimensional Laplace transform again can be expressed as the ratio of two terms, and the numerator $\Phi\topp n\equiv \Phi\topp n_{u,v}(\vvx,\vvc)$ has a path decomposition similar to the one of $\wt\Psi_{u,v}(\vvx,\vvc)$ in \eqref{eq:limit path decomp}. Indeed, we shall show that 
\equh\label{eq:prelimit decomp}
\Phi^{(n)}
=
\Phi^{(n),-}+\Phi_{\emptyset}^{(n)}
+\sum_{\ell=1}^{j}\ \sum_{\vvr\in\mathbb S_\ell\topp i}
 \Phi_{\vvr}^{(n)}
+\sum_{\ell=j+1}^{d}\ \sum_{\vvr\in\mathbb S_\ell\topp{i-1}}
 \Phi_{\vvr}^{(n)}, 
\eque
and further prove that
\begin{align*}
\limn n^{u+v} \Phi_\vvr\topp n &= \Gamma(u+v+2)e^{G(\vvx,\vvc)}H_\vvr(u), \mbox{ for legitimate } \vvr,\\
\limn n^{u+v}\Phi_\emptyset\topp n & = \Gamma(u+v+2)e^{G(\vvx,\vvc)}H_\emptyset(u),\\
\limn n^{u+v} \Phi^{(n),-} &= 0,
\end{align*}
 where $e^{G(\vvx,\vvc)}$ is the finite-dimensional Laplace transform of the corresponding Brownian component in \eqref{eq:stationary open KPZ} (for its formula, see \eqref{eq:G}) under suitable assumptions on $u,v,\vvc$. 
 Note that the assumptions for the limit theorems to hold are more restrictive than the ones for the decomposition to hold; for a comparison see Remark \ref{rem:uv}.
Our proof avoids completely passing through the open KPZ equation. At a few places we also simplify the treatments for the fan region \citep{corwin24stationary} (see Remarks \ref{rem:CK}, \ref{rem:DCT}, and \ref{rem:proof for fan}). 
\begin{remark}
When computing the convergence from the stationary measure of the open ASEP to the stationary measure of the open KPZ fixed point, \citet{wang25asymmetric} took the same approach as here and encountered a similar problem. Therein, the authors managed to establish the convergence of the finite-dimensional Laplace transform, where the pre-limit and the limiting formulas are similar to the ones in \eqref{eq:prelimit decomp} and \eqref{eq:limit path decomp} respectively. However, the authors were unable to relate the limiting Laplace transform directly to the stationary measure also defined via a Radon--Nikodym derivative. Instead, the authors used a soft argument that crucially relies on a recent result of \citet{bryc24two} on the two-layer representation of the open TASEP (totally asymmetric simple exclusion process). We expect that a similar analytic continuation argument can relate the limiting finite-dimensional Laplace transform to the Radon--Nikodym derivative representation (which was known for the fan region) directly.
\end{remark}
\begin{remark}\label{rem:two-layer}
An alternative route to the convergence of stationary open WASEP measures uses two-layer representations. \citet{enaud04large} introduced the bidiagonal matrix-product representation which leads to a representation of the stationary measure of the open ASEP in terms of weighted Motzkin-type paths (see also a related two-layer representation of the open TASEP in \citet{derrida04asymmetric}). Under Liggett's condition \citep{liggett75ergodic} and in the fan region, \citet{barraquand23stationaryJPA} used the Enaud--Derrida representation to express the stationary measure of the open ASEP through a reweighted two-dimensional simple random walk. They identified its formal limit in the weakly asymmetric regime with another Radon--Nikodym representation  of the stationary measure of the  open KPZ equation, and \citet{himwich24stationary} subsequently proved the convergence. This representation, however,  does not extend to the shock region. 

More recently, \citet{bryc25stationary} introduced a different two-layer representation for the open ASEP that is valid in both the fan and shock regions. \citet{schmid26sharp,schmid26stationary} provided a probabilistic interpretation for the associated Bryc polynomial in this representation in terms of a birth--death Markov chain, and then used these results to develop a new proof for convergence of the stationary measure of the open WASEP for both fan and shock regions. This approach is very probabilistic and very different in nature from the earlier developments based on Askey--Wilson representation. The proof of the convergence of the stationary measure of the open WASEP to the law of $\eta\topp{u,v}$ is considerably shorter, and moreover it shows the convergence in the $D$ space instead of in finite-dimensional distributions as our approach shows here. 
\end{remark}

\subsection*{Outline of the paper}
In Section \ref{sec:limiting measures}, we introduce the auxiliary
measures and provide a path decomposition of the finite-dimensional Laplace transform. The proof of this main result is provided in Section \ref{sec:proof Laplace}. In Section \ref{sec:convergence}, we prove the convergence of the finite-dimensional distributions of the rescaled stationary height function of the open WASEP.

\subsection*{Statement on the usage of AI}We started working on this problem around the end of 2024. By the spring of 2026, we had obtained essentially all the results presented here and written most of the first draft without using AI. ChatGPT was then used to proofread the mathematics and help improve the writing and the presentation. In particular, ChatGPT helped find an improvement in the dominated convergence argument used in Section \ref{sec:convergence}; see Remark \ref{rem:DCT}.

\subsection*{Acknowledgements}
The authors would like to thank Zongrui Yang for his involvement in the early stage of this project and for helpful discussions. YZ thanks Wlodek Bryc for several helpful discussions. ZH was supported by NSF MSPRF 2503374. YW was partially supported by 
the Simons Foundation (MP-TSM-00002359), and a Taft Center Fellowship (2025--2026) at the University of Cincinnati.

\section{Finite-dimensional Laplace transforms}\label{sec:limiting measures}

Our goal is to provide an explicit formula for the finite-dimensional Laplace transform of $\mathsf P_{u,v}$. For convenience we let $\zeta\topp{u,v}$ denote a process with law $\mathsf P_{u,v}$. 
Throughout the section, we assume $d\in\Z_{>0}$, $0=x_0<x_1<\cdots<x_d=1$, and $c_i>0$ for $i\in\llbracket d\rrbracket=\{1,\dots,d\}$. 
We write
\equh\label{eq:Laplace fdd}
   \esp\exp\pp{-\summ k1d c_k \zeta_{x_k}\topp{u,v}}
=
  \frac{\wt\Psi_{u,v}(\vvx,\vvc)}{\wt K_{u,v}},
\eque
with
\begin{align}
\wt\Psi_{u,v}(\vvx,\vvc) &:=\esp\pp{ \exp\left(-\frac1{\sqrt2}\summ k1d c_k\B_{x_k}-\sqrt2v\B_1\right)\left(\int_0^1e^{-\sqrt2\B_t}\d t\right)^{-(u+v)}},\nonumber\\
\wt K_{u,v} &:=\esp\pp{e^{-\sqrt2v\B_1}\left(\int_0^1e^{-\sqrt2\B_t}\d t\right)^{-(u+v)}}.\label{eq:wt K def}
\end{align}
We shall allow $u,v$ to take complex values. 

 Our expressions shall involve a family of complex-valued measures  $\mathfrak p_s^{u,v}(\d x)$ and kernels $\mathfrak p_{s,t}^{u,v}(x,\d y)$. Their full definition will be provided in Section \ref{sec:CDH}. 
For $t_1<\cdots<t_d$ and $t_1,\dots,t_d\in\mathbb T_{u,v}$ (to be defined below), we introduce the auxiliary measure
\equh\label{eq:mu}
\mu_{t_1,\dots,t_d}^{u,v}(\d\vvy):=\mathfrak p_{t_1}^{u,v}(\d y_1)\mathfrak p_{t_1,t_2}^{u,v}(y_1, \d y_2)\cdots\mathfrak p_{t_{d-1},t_d}^{u,v}(y_{d-1},\d y_d),
\eque
on $\C^d$.
That is, for every measurable function $f:\C^d\to\C$ for which the following integral is absolutely convergent, we write
\[
\int f\d \mu_{t_1,\dots,t_d}^{u,v} = \int\cdots\int f(y_1,\dots,y_d)\mathfrak p_{t_{d-1},t_d}^{u,v}(y_{d-1},\d y_d)\cdots \mathfrak p_{t_1,t_2}^{u,v}(y_1, \d y_2)\mathfrak p_{t_1}^{u,v}(\d y_1).\]
We first introduce
\equh\label{eq:calD}
\calD := \ccbb{(u,v)\in\C^2: \rmRe(v),\rmRe(u)\notin\Z_{\le 0},  \rmRe(u)-\rmRe(v)\notin\Z \mbox{ if } \rmRe(u),\rmRe(v)<0},
\eque
and 
\[
\mathbb T_{u,v}:=\ccbb{t\in\R:\rmRe(u)+\frac t2\notin\Z_{\le 0}, \rmRe(v)-\frac t2> 0}, \quad (u,v)\in\calD.
\]
Next, we introduce, for all $(u,v)\in\calD$, 
\equh\label{eq:K_u,v}
K_{u,v} := \int_\C e^{-y/4}\mathfrak p_0^{u,v}(\d y),
\eque
and, with 
 $s_j:=\sum_{k=j}^d c_k$ for $j\in\llbracket d\rrbracket$ and 
$-s_1,\ldots,-s_d\in\mathbb T_{u,v}$,
\equh\label{eq:Psi}
\Psi_{u,v}(\vvx,\vvc) := \int_{\C^d}\exp\pp{-\frac14\summ k1d y_k(x_k-x_{k-1})}\mu_{-s_1,\dots,-s_d}^{u,v}(\d\vvy).
\eque
We can now state the main result of this section.
\begin{theorem}\label{thm:AC}
For all $(u,v)\in\calD$, 
\equh\label{eq:denom equal}
K_{u,v} = \wt K_{u,v}.
\eque
In addition, for every $(u,v)\in\calD$ and $d\in\Z_{>0}$, $0 = x_0<x_1<\cdots<x_d=1$, assume further
\equh\label{eq:c close} c_1,\dots,c_d>0 \qmand
\sum_{k=1}^{d-1}c_k<1,
\eque
such that $-s_1,\dots,-s_d\in\mathbb T_{u,v}$, 
\equh\label{eq:num equal}
\Psi_{u,v}(\vvx,\vvc) = \wt\Psi_{u,v}(\vvx,\vvc).
\eque
In particular, for $(u,v)\in\calD\cap\R^2$, 
\[
\esp\exp\pp{-\summ k1d c_k\zeta_{x_k}\topp{u,v}}
=
\frac{\Psi_{u,v}(\vvx,\vvc)}{K_{u,v}}.
\]
\end{theorem}
\begin{remark}
It is well-known that the values of the finite-dimensional Laplace transform $\Psi_{u,v}(\vvx,\vvc)/K_{u,v}$ for $\vvc$ from an open set of $\R^d$ determine the corresponding finite-dimensional distributions of the process. We shall also use the fact that if for a pre-limit model the finite-dimensional Laplace transform with coefficients $\vvc$ from an open set of $\R^d$ converges to $\Psi_{u,v}(\vvx,\vvc)/K_{u,v}$, then the corresponding finite-dimensional distribution of the pre-limit model also converges to the ones of $\zeta\topp{u,v}$  \citep[Appendix A]{bryc19limit}.
\end{remark}
The proof is essentially an analytic continuation argument. It is clear that $\wt K_{u,v}$ and $\wt\Psi_{u,v}(\vvx,\vvc)$ are entire on $\C^2$ as functions of $(u,v)$. For completeness we prove this result in Proposition \ref{prop:analytic RN} later. 

We focus on $\Psi_{u,v}(\vvx,\vvc)$ through the rest of this section. 
While its expression in \eqref{eq:Psi} takes the identical form as an integral with respect to the measure $\mu^{u,v}_{-s_1,\dots,-s_d}$, this measure can take drastically different forms for different $u,v,\vvc$. The main contribution here is to have a more detailed description of this measure, which leads to a path decomposition of the expression of $\Psi_{u,v}(\vvx,\vvc)$ as in \eqref{eq:limit path decomp}. This result is Proposition \ref{prop:path decomp}. 

Below we provide an overview of the rest of this section. In the initial fan-region domain, the   auxiliary measure $\mu_{-s_1,\dots,-s_d}^{u,v}$ in \eqref{eq:Psi} under $-2u<-s_1<\cdots<-s_d<2v$ is supported on $(0,\infty)^d$ and can be thought of, heuristically, as the finite-dimensional distribution of a Markov process at times $-s_1,\dots,-s_d$ (strictly speaking this is not a finite-dimensional distribution as the starting measure is an infinite measure). 

We shall fix $v$ and then apply the analytic continuation argument to $u$. First we allow $u,v$ to be complex-valued and the identity  \eqref{eq:num equal} continues to hold with $-2\rmRe(u)<-s_1<\cdots<-s_d<2\rmRe(v)$. 
As $\mathrm{Re}(u)$ decreases, the support of the auxiliary measure is contained in
\equh\label{eq:support}
\pp{(0,\infty)\cup \mathbb A_{-s_1}^u}\times \cdots\times\pp{(0,\infty)\cup \mathbb A_{-s_d}^u},
\eque
where $\mathbb A_{-s_j}^u$ (defined below in \eqref{eq:A}), when nonempty, gives the possible atomic locations of the $j$th coordinate of the measure. 
When $u,v\in\R, u+v>0$, the measure $\mu_{-s_1,\dots,-s_d}^{u,v}$ is heuristically the finite-dimensional distribution of a Markov process starting from an infinite measure at time $-s_1$, and then an element from the set \eqref{eq:support} can be viewed as a candidate trajectory of the $d$-step Markov process. This terminology has been convenient and commonly used in related earlier results. Here for general $u,v$ even when the measure no longer corresponds to a Markov process,  we continue to refer to each element from the set \eqref{eq:support} as a trajectory (path) of $d$ steps. 

It turns out that for $d\ge 2$ the support of the auxiliary measure can be much smaller than \eqref{eq:support}: once the trajectory is within $(0,\infty)$ at a time point it will remain there at all later time points. In this way, the support can be decomposed further into disjoint components depending on the atomic part at the beginning of each trajectory. We refer to this result, Proposition \ref{prop:path decomp}, as the path decomposition of the support of the auxiliary measure. The condition $\sum_{k=1}^{d-1}c_k<1$ is introduced so that the classification of the admissible paths can be encoded in a simple way. This decomposition then induces  a path decomposition of the Laplace transform: we simply write the integral on the right-hand side of \eqref{eq:Psi} as a summation of the same integral but over each component of the path decomposition of the domain. 

In the rest of this section, we shall introduce the auxiliary measure and state the result on the path decomposition of $\Psi_{u,v}(\vvx,\vvc)$. Most of the proofs, and also the analysis of the denominator $K_{u,v}$, are provided  in the next section.

\begin{remark}\label{rem:D approx}
Our expressions for $K_{u,v}$ and $\Psi_{u,v}(\vvx,\vvc)$ extend continuously 
to all  $(u,v)\in\C^2$ with $\rmRe(v)\ge -c_d/2$ as follows. We assume always that $\vvc$ satisfies \eqref{eq:c close}. Write 
\[
\mathcal A_{\vvc}
:=\{(u,v)\in\calD:
    -s_j\in\mathbb T_{u,v}\text{ for every }j\in\llbracket d\rrbracket\}.
\]
By Proposition \ref{prop:analytic RN}, both $\wt K_{u,v}$ and
$\wt\Psi_{u,v}(\vvx,\vvc)$ are entire in $(u,v)$. Hence
\eqref{eq:denom equal} and \eqref{eq:num equal} imply
\[
\wt K_{u,v}
=\lim_{\substack{(u',v')\to(u,v)\\(u',v')\in\calD}}K_{u',v'}, \quad \mfa (u,v)\in\C^2,
\]
and
\[
\wt\Psi_{u,v}(\vvx,\vvc)
=\lim_{\substack{(u',v')\to(u,v)\\(u',v')\in\mathcal A_{\vvc}}}
  \Psi_{u',v'}(\vvx,\vvc)
\quad \mfa (u,v)\in\C^2, \rmRe(v)\ge -c_d/2.
\]
Some  parameters in $\C^2\setminus\calD$ admit direct, term-by-term evaluation
of the auxiliary-measure formulas. For example, take $u=-1/2+\i$ and $v=-1/2$. Although $(u,v)\notin\calD$, the formulas below show that the continuous density
and both atomic weights in $\mathfrak p_0^{u,v}$ are well-defined. 
Thus \eqref{eq:K_u,v} can be evaluated directly at this choice of $(u,v)$.
Such a plug-in evaluation is not always possible. For example,  let $u=-1/2+\epsilon$ and
$v=-1/2-\epsilon$, for sufficiently small real $\epsilon\ne0$. The continuous contribution to $K_{u,v}$ vanishes, and its two
atomic contributions are
\[
-\frac{e^{u^2}}{2\epsilon}
+\frac{e^{v^2}}{2\epsilon}
=
e^{1/4+\epsilon^2}
\frac{\sinh\epsilon}{\epsilon}
\to e^{1/4},
\]
while neither has a finite limit separately. 
Hence the extension at $u=v=-1/2$ requires taking the limit as mentioned above.
\end{remark}

\subsection{Auxiliary measures}\label{sec:CDH}
We introduce three auxiliary objects. We refer to $\mathfrak p_t^{u,v}$ as the marginal  measure, to $\mathfrak p_{s,t}^{u,v}$ as the kernel, and to $\mu_{t_1,\ldots,t_d}^{u,v}$ as the finite-dimensional auxiliary measure. The naming follows from the fact that they can be associated to a Markov process in the fan region. For each fixed $x$, the kernel $\mathfrak p_{s,t}^{u,v}(x,\d y)$ is a measure. For general complex parameters $u,v$ all these measures are complex measures, while for real parameters they are signed ones. A more detailed summary is provided in Remark \ref{rem:measures}. Recall the definition of $\mu_{t_1,\dots,t_d}^{u,v}$ in \eqref{eq:mu}.

We introduce $\mathfrak p_s^{u,v}$ and $\mathfrak p_{s,t}^{u,v}$ separately. For $u,v\in\R$ and $u+v>0$, $\mathfrak p_{s,t}^{u,v}$ can be interpreted as a probability transition kernel, and the corresponding Markov process is referred to as the continuous dual Hahn process \citep{bryc12quadratic,bryc22continuous,corwin24stationary}. Then, $\mathfrak p_s^{u,v}$ is an infinite $\sigma$-finite measure satisfying $\int \mathfrak p_s^{u,v}(\d x) \mathfrak p_{s,t}^{u,v}(x,\d y) = \mathfrak p_t^{u,v}(\d y)$. 
(The measure $\mathfrak p_s^{u,v}$ is not a continuous dual Hahn measure but can be related to the latter via limit theorems; see \citet[Proof of Theorem 4.1]{bryc22continuous}.) The boundary case $u+v=0$ is degenerate and the auxiliary measure is concentrated on a single deterministic path. The interpretation via a Markov process is lost when $u+v<0$. See Remark \ref{rem:measures} for a detailed description. For complex parameters $u,v$, the measure $\mathfrak p_s^{u,v}$ is a locally finite complex-valued Borel measure of which the support is within $\C$, and it does not necessarily have finite mass. Similarly, $\mathfrak p_{s,t}^{u,v}$ may be complex-valued.

\subsubsection{Marginal measures}
Fix $(u,v)\in\calD$. For $t\in\mathbb T_{u,v}\cup \{0\}$, the measure $\mathfrak p_t^{u,v}$ has an absolutely continuous part on $(0,\infty)$ and possibly a finite number of atoms at locations 
\begin{align}
u_{t,j} &:= -4\pp{u+\frac t2+j}^2,\quad j=0,\dots,j_{u,t}:=\ceil{-\rmRe(u)-t/2-1},\label{eq:j_u,t}\\
v_{t,k} &:= -4\pp{v-\frac t2+k}^2, \quad k=0,\dots,k_{v,t}:=\ceil{-\rmRe(v)+t/2-1}.\nonumber
\end{align} 
(Note that the atom locations may be in $\C\setminus\R$.) 
For $j\in \llbracket 0, j_{u,t}\rrbracket$ (we write $\llbracket k,\ell\rrbracket = \{k,\dots,\ell\}$ for $k,\ell\in\Z, k\le\ell$), we refer to $u_{t,j}$ as the location of the $j$-th {\em candidate} atom associated to parameter $u$ at time $t$ provided $j_{u,t}\ge0$, and similarly for the location of the $k$-th candidate atom $v_{t,k}$ when $k\in\llbracket 0, k_{v,t}\rrbracket$. We say `candidate' because the weight may be zero; see Remark \ref{rem:candidate atoms} below.
Then, we have
\equh\label{eq:CDH marginal}
\mathfrak p_t^{u,v}(\d y) := \mathfrak p_t^{u,v,c}(y)\d y
+ \sum_{j=0}^{j_{u,t}}\mathfrak p_t^{u,v,d}(u_{t,j})\delta_{u_{t,j}}(\d y)
 + \sum_{k=0}^{k_{v,t}}\mathfrak p_t^{u,v,d}(v_{t,k})\delta_{v_{t,k}}(\d y),
\eque
where $\delta_z$ denotes the unit point measure at $z$.
The density and atomic masses are given by
\begin{align}
\mathfrak p_t^{u,v,c}(y) & := \frac{1}{8\pi\Gamma(u+v)}\ddfrac{{\Gamma\pp{v-\frac t2\pm\i\frac{\sqrt{y}}2,u+\frac t2\pm\i\frac{\sqrt{y}}2}}}{\sqrt{y}\Gamma(\pm\i\sqrt{y})} \inddd{y>0},
\label{eq:pc}\\
    \mathfrak{p}_{t}^{u,v,d}(u_{t,j}) & := \frac{\Gamma\left(v-u-t\right)}{\Gamma\left(-2u-t\right)} \frac{[2u+t,v+u]_{j}(2u+t+2j)}{[1,1+u-v+t]_{j}(2u+t)},  \quad j\in\llbracket 0,j_{u,t}\rrbracket ,\label{eq:pd u}\\
    \mathfrak{p}_{t}^{u,v,d}(v_{t,k}) & := \frac{\Gamma\left(u-v+t\right)}{\Gamma\left(-2v+t\right)}\frac{[2v-t,v+u]_{k}(2v+2k-t)}{[1,1-u+v-t]_{k}(2v-t)} ,\quad k\in \llbracket 0,k_{v,t}\rrbracket.\label{eq:pd v}
\end{align}
Here, we write $[a]_0 = 1$, 
$[a]_k:=\prod_{j=0}^{k-1}(a+j)$ for
$k\in\mathbb Z_{\geq0}$, 
$[a_1,\ldots,a_r]_k:=\prod_{\ell=1}^r[a_\ell]_k$, 
$\Gamma(a_1,\ldots,a_r):=\prod_{\ell=1}^r\Gamma(a_\ell)$,
 $\Gamma(a\pm b) = \Gamma(a+b)\Gamma(a-b)$, and $\Gamma(a\pm b,c\pm d) = \Gamma(a\pm b)\Gamma(c\pm d)$. 
Reciprocal Gamma factors are interpreted using the entire
function $1/\Gamma$.
In the formula \eqref{eq:CDH marginal} above, when either $j_{u,t}<0$ or $k_{v,t}<0$, the corresponding sum is understood as zero. 

A slightly more compact way of writing  \eqref{eq:pd u} and \eqref{eq:pd v} is 
\begin{align}
\mathfrak p_t^{u,v,d}(u_{t,j})&= -\frac{\Gamma(v-u-t-j)}{\Gamma(-2u-t-j+1)}\frac{[u+v]_j(2u+t+2j)}{j!}, \quad j\in\llbracket 0,j_{u,t}\rrbracket,\label{eq:pd u'}\\
\mathfrak p_t^{u,v,d}(v_{t,k})&= -\frac{\Gamma(u-v+t-k)}{\Gamma(-2v+t-k+1)}\frac{[u+v]_k(2v-t+2k)}{k!}, \quad k\in\llbracket 0,k_{v,t}\rrbracket.\label{eq:pd v'}
\end{align}
\begin{remark}
\label{rem:candidate atoms}
The weights assigned to the candidate atoms in \eqref{eq:pd u'} and \eqref{eq:pd v'} may vanish when $u+v\in\Z_{\le0}$. If $u+v\notin\Z_{\leq0}$, all the listed candidate
atoms have nonzero masses. Indeed, for every index $j\in\llbracket 0,j_{u,t}\rrbracket$,
all factors on the right-hand side of \eqref{eq:pd u'} other than
$[u+v]_j$ are finite and nonzero. Consequently, for $j\geq1$,
$\mathfrak p_t^{u,v,d}(u_{t,j})=0$ if and only if 
\[
u+v=-m
\quad\text{for some }m\in\llbracket 0,j-1\rrbracket.
\]
The analogous statement holds for the candidate atoms $v_{t,k}$.
In particular, if $u+v=-m$ for some $m\in\Z_{\geq0}$, then every
listed candidate atom with index at least $m+1$ has zero mass. Thus,
if $j_{u,t}\geq m+1$, the terminal candidate
$u_{t,j_{u,t}}$ has zero mass, and similarly for
$v_{t,k_{v,t}}$. At the boundary $u+v=0$, only the
candidate with index zero can have nonzero mass; that is, for
$t\in\mathbb T_{u,-u}$,
$\mathfrak p_t^{u,-u}=\delta_{u_{t,0}}$.
\end{remark}

For $(u,v)\in\calD$, $\mathfrak p_0^{u,v}$ (defined in \eqref{eq:CDH marginal}) is well-defined. The fact that $\mathfrak p_0^{u,v,c}$ is well-defined is guaranteed by the assumption $\rmRe(u), \rmRe(v)\notin\Z_{\le0}$. For $\mathfrak p_0^{u,v,d}$, see the next remark. 
\begin{remark}
The Gamma factors in the numerators of \eqref{eq:pd u'} and \eqref{eq:pd v'} may have poles. Here we check that the functions are well-defined. Start with \eqref{eq:pd u'} (we first assume $u,v\in\R$; for $u,v\in\C$, the same argument is applied to the real parts of the Gamma-function arguments).
For every $j\in\llbracket 0,j_{u,t}\rrbracket$, the definition of $j_{u,t}$ gives $u+t/2+j<0$. Hence
\[
v-u-t-j=v-\frac t2-\left(u+\frac t2+j\right) > v-\frac t2.
\]
Thus $\Gamma(v-u-t-j)$ has no pole when $v-t/2\ge0$.
A similar argument applies to show that \eqref{eq:pd v'} is involved only when $v-t/2<0$, and in this case if $u+t/2\ge 0$ then \eqref{eq:pd v'} is well-defined. 

That is, the only case that has not been addressed above is when $u+t/2<0$ and $v-t/2<0$. Since $t\in\mathbb T_{u,v}$ implies $v-t/2>0$, this case can arise here only when $t=0$. In that case, the condition $u-v\notin\Z$ ensures that neither $v-u-j$ nor $u-v-k$ belongs to $\Z_{\leq0}$ for any admissible $j$ or $k$. For complex $u,v$, the same conclusion follows from $\rmRe(u)-\rmRe(v)\notin\Z$. This also explains the corresponding constraint in the definition of $\calD$. Consequently, $\mathfrak p_0^{u,v,d}$ is well-defined for every $(u,v)\in\calD$.
\end{remark}

\subsubsection{Transitional kernels}
Assume always $(u,v)\in\calD$. For each $t\in\mathbb T_{u,v}$, since $\rmRe(v)-t/2>0$, there are no atoms associated to the parameter $v$. 
This restriction of parameters is convenient and sufficient for our application to the open WASEP later. 
For $r\in\mathbb T_{u,v}$, we define the set of candidate atom locations as
\equh\label{eq:A}
\mathbb A_r^u
:=
\{u_{r,j}:j\in\llbracket 0 , j_{u,r}\rrbracket \},
\eque
where the set is understood as empty when $j_{u,r}<0$. For each $s,t\in\mathbb T_{u,v}$ with $s<t$, define
\[
\mathfrak p_{s,t}^{u,v}(x,\d y) := \mathfrak p_{s,t}^{u,v,c,c}(x,y)\d y + \mathfrak p_{s,t}^{u,v,d,c}(x,y)\d y + \mathfrak p_{s,t}^{u,v,d,d}(x,\d y),
\]
as a measure with $\mathfrak p_{s,t}^{u,v,c,c}, \mathfrak p_{s,t}^{u,v,d,c}$ and $\mathfrak p_{s,t}^{u,v,d,d}$ to be defined below.
For $s,t\in\mathbb T_{u,v}$ and $s<t$, the density functions are
\begin{align}
        \mathfrak{p}_{s,t}^{u,v,c,c}(x,y)         & := \ddfrac{{\Gamma\pp{v-\frac{t}{2}\pm\i\frac{\sqrt{y}}{2}}}}{{\Gamma\pp{v-\frac{s}{2}\pm\i\frac{\sqrt{x}}{2}}}}\ddfrac{\abs{\Gamma\pp{ \frac{t-s}{2}+\i\frac{\sqrt{x}}2\pm  \i\frac{\sqrt{y}}2}}^2}{\displaystyle 8\pi \sqrt{y}\Gamma(t-s)\abs{\Gamma(\i\sqrt{y})}^{2}}\inddd{y>0}, x\in(0,\infty),
        \label{eq:transitional cc}
\\
        \mathfrak{p}_{s,t}^{u,v,d,c}(u_{s,j},y) & := \frac{\displaystyle{\Gamma\left(\frac{t-s}{2}\pm\left(u+j+\frac{s}{2}\right)\pm\i\frac{\sqrt{y}}{2},v-\frac{t}2\pm\i\frac{\sqrt{y}}{2}\right)}}{\displaystyle 8\pi\sqrt{y}\Gamma\pp{v-\frac{s}{2}\pm\left(u+j+\frac{s}{2}\right),t-s}\abs{\Gamma(\i\sqrt{y})}^{2}}\inddd{y>0}, j\in\llbracket 0,j_{u,s}\rrbracket,\label{eq:transitional dc}
\end{align}
and $\mathfrak p_{s,t}^{u,v,c,c}(x,\cdot) = 0$ for all $x\notin(0,\infty)$ and $\mathfrak p_{s,t}^{u,v,d,c}(x,\cdot) = 0$ for all $x\notin\mathbb A_s^u$.
 Note that $s,t\in\mathbb T_{u,v}$ implies that $\rmRe(v)-t/2>0$. So $\mathfrak p_{s,t}^{u,v,c,c}$ is well-defined. For \eqref{eq:transitional dc} it suffices to examine the remaining numerator factors, which we can write as $\Gamma(t/2+u+j\pm \i\sqrt{y}/2)\Gamma(t/2-u-j-s\pm \i\sqrt{y}/2)$. For the first pair, the condition $\rmRe(u)+t/2\notin\Z_{\leq0}$ ensures that neither Gamma factor has a pole for any $j\geq0$ and $y>0$, since the real part of each argument is $\rmRe(u)+t/2+j\notin\Z_{\leq0}$. For the second pair, we recall that the expression \eqref{eq:transitional dc} is involved only when $j_{u,s} =\ceil{-\rmRe(u)-s/2-1} \ge 0$, and in this case $\rmRe(t/2-u-j-s)
>(t-s)/2>0$.

Next, for each $x\in\mathbb A_s^{u}$, $\mathfrak p^{u,v,d,d}_{s,t}(x,\cdot)$ is a discrete measure whose possible atomic locations lie in $\mathbb A_t^u$.
We write, with a slight abuse of notation,
 \begin{align*}
\mathfrak p^{u,v,d,d}_{s,t}(x,y) \equiv \mathfrak p^{u,v,d,d}_{s,t}(x,\{y\}), \quad x\in \mathbb A_s^u, y\in\mathbb A_t^u.
 \end{align*}
We have 
\begin{multline*}
\mathfrak p_{s,t}^{u,v,d,d}(u_{s,j},u_{t,j'}) = \frac{\Gamma(v-u-j-t,-2u-2j-s)}{\Gamma(v-u-j-s,-2u-2j-t)}\\
\times \frac{[t+2u+2j,v+u+j,t-s]_{j'-j}(t+2u+2j')}{[1,1+t+j+u-v,1+2u+2j+s]_{j'-j}(t+2u+2j)}(-1)^{j'-j}\inddd{j\le j'}, 
\end{multline*}
for all $j\in \llbracket 0, j_{u,s}\rrbracket$ and $j'\in \llbracket j, j_{u,t}\rrbracket$, and $\mathfrak p_{s,t}^{u,v,d,d}(x,\cdot) = 0$ for all $x\notin \mathbb A_s^u$. 

\begin{remark}
\label{rem:measures}
For general $(u,v)\in\calD$, the marginal measures
$\mathfrak p_t^{u,v}$, the kernel measures
$\mathfrak p_{s,t}^{u,v}(x,\cdot)$, and the resulting auxiliary
measures $\mu_{t_1,\ldots,t_d}^{u,v}$ are complex-valued. The first
two are measures on $\C$, whereas
$\mu_{t_1,\ldots,t_d}^{u,v}$ is a measure on $\C^d$. When
$u,v\in\R$, these measures are real-valued but need not be positive.
We distinguish three cases and provide a summary for each case below.

First, suppose that $u+v>0$. In this case, the marginal measures and
kernels are associated with the continuous dual Hahn process. More
precisely,
\begin{enumerate}[(i)]
\item $\mathfrak p_t^{u,v}$ is a positive, locally finite measure of
infinite total mass for every $t\in\mathbb T_{u,v}$;
\item for every $s<t$ in $\mathbb T_{u,v}$ and every
$x\in(0,\infty)\cup \mathbb A_s^u$,
$\mathfrak p_{s,t}^{u,v}(x,\cdot)$ is a probability measure. When
$x>0$, this probability measure is supported on $(0,\infty)$.
\end{enumerate}

Second, suppose that $u+v=0$, so that $v=-u$. In this case, the
auxiliary measures are concentrated on a deterministic trajectory.
More precisely,
\begin{enumerate}[(i)]
\item for $t\in\mathbb T_{u,-u}$,
$\mathfrak p_t^{u,-u}=\delta_{u_{t,0}}$,
\item for every $s<t$ in $\mathbb T_{u,-u}$,
$ \mathfrak p_{s,t}^{u,-u}(u_{s,0},\d y)
    =\delta_{u_{t,0}}(\d y)$.
\end{enumerate}
On the probabilistic side, the Girsanov formula gives
\[
    \pp{\zeta_x^{u,-u}}_{x\in[0,1]}
    \stackrel{\mathrm d}{=}
    \left(ux+\frac{1}{\sqrt2}\B_x\right)_{x\in[0,1]}.
\]
Now, for our formula, $\mathfrak p_0^{u,-u}=\delta_{-4u^2}$, so that
$K_{u,-u}=e^{u^2}$. Since the auxiliary measure is concentrated
on $(u_{-s_1,0},\ldots,u_{-s_d,0})$, we have
\[
\Psi_{u,-u}(\vvx,\vvc)
=\exp\pp{
\sum_{k=1}^d(x_k-x_{k-1})
\left(u-\frac{s_k}{2}\right)^2}.
\]
One readily checks that
\[
\frac{\Psi_{u,-u}(\vvx,\vvc)}{K_{u,-u}}
=
\exp\pp{
-u\sum_{k=1}^d c_kx_k
+\frac14\sum_{j,k=1}^d c_jc_k\min\{x_j,x_k\}} 
=\mathbb E
    \exp\pp{-\sum_{k=1}^d c_k\zeta_{x_k}^{u,-u}}.
\]

Third, suppose that $u+v<0$. In this case, the marginal measures are
genuinely signed, and the family as a whole does not define a Markov
process. Nevertheless, some of the individual kernel measures remain
probability measures. More precisely,
\begin{enumerate}[(i)]
\item $\mathfrak p_t^{u,v}$ is a nonzero finite signed measure of
total mass zero:
\[
    \mathfrak p_t^{u,v}(\R)=0.
\]
If $u+v\notin\Z_{<0}$, its continuous density has the constant sign
of $1/\Gamma(u+v)$, and the sum of its atomic masses has the opposite
sign. The individual atomic masses, however, need not have the same
sign. If $u+v$ is a negative integer, then the continuous part
vanishes and the atomic masses sum to zero;

\item for every $x>0$ and every $s<t$ in $\mathbb T_{u,v}$,
$\mathfrak p_{s,t}^{u,v}(x,\d y)$ is a probability measure supported
on $(0,\infty)$;

\item for every $s<t, s,t\in\mathbb T_{u,v}$ and for every candidate atom $u_{s,j}$, whether $\mathfrak p_{s,t}^{u,v}(u_{s,j},\cdot)$ is non-negative 
depends on $u+v+j$. If $u+v+j>0$, then
$\mathfrak p_{s,t}^{u,v}(u_{s,j},\d y)$ is a probability measure. If
$u+v+j=0$, then
\[
    \mathfrak p_{s,t}^{u,v}(u_{s,j},\d y)
    =\delta_{u_{t,j}}(\d y).
\]
If $u+v+j<0$, then
$\mathfrak p_{s,t}^{u,v}(u_{s,j},\d y)$ is a genuinely signed measure
of total mass one.
\end{enumerate}
\end{remark}

\subsection{Path decomposition of the support of the auxiliary measure}\label{sec:path}
For $a<b$, we write $\C_{(a,b)}:=\{z\in\C:\mathrm{Re}(z)\in(a,b)\}$ and refer to such a region as a strip.
Now that we have introduced the auxiliary measures, we have more explicit expressions of $\Psi_{u,v}(\vvx,\vvc)$ on each strip. It turns out that they can be further simplified. 

Recall that throughout  $c_i>0, i=1,\dots,d$ and $\summ i1{d-1} c_i<1$. Throughout this section, we shall fix $v$ but vary $u$. Thus
we assume $v\in\C$ is such that
\equh\label{eq:v}
-c_d<2\rmRe(v).
\eque
(Recall that $-s_1<\cdots<-s_d = -c_d$.)
Then, for every $u$ with $(u,v)\in\calD$, recall that we assume $-s_j\in\mathbb T_{u,v}$ for all $j\in\llbracket d\rrbracket$, which under the previous assumption is equivalent to
\equh\label{eq:u constraint}
\rmRe(u)-\frac{s_j}2\notin\Z_{\le0}.
\eque
We then set 
\[
s_{j,i} := s_j-2i, j=1,\dots, d,
\] 
and 
\[
s_{d+1,i} := s_{1,i+1}\qmand s_{0,i+1}:=s_{d,i} \quad\mbox{ for }\quad i\in\Z_{\ge 0},
\]
 with the convention $s_{0,0}:=+\infty$. Then, for $u,v$ satisfying the above, $\Psi_{u,v}(\vvx,\vvc)$ is defined for $u\in\C_{\pp{s_{j+1,i}/2,s_{j,i}/2}}$, for any $i\in\Z_{\ge0}$ and $j\in\llbracket 0,d\rrbracket$. On the strip $\C_{(s_{j+1,i}/2,s_{j,i}/2)}$, the candidate atomic indices associated with $u$ are $0,\ldots,i$ at the first $j$ auxiliary times $-s_1,\ldots,-s_j$. At the remaining times, the admissible indices are $0,\ldots,i-1$, with this list empty when $i=0$.

We also set $\Lambda_{\emptyset}^{(d)}:=(0,\infty)^d$, and introduce
\[
\Lambda\topp d_{(r_1,\dots,r_\ell)}\equiv \Lambda\topp d_{(r_1,\dots,r_\ell)}(u,\vv c) :=\{u_{-s_1,r_1}\}\times\cdots\times\{u_{-s_\ell,r_\ell}\}\times(0,\infty)^{d-\ell},
\]
with $r_k\in\llbracket 0, j_{u,-s_k}\rrbracket, k\in \llbracket \ell\rrbracket$. Here $\llbracket\ell\rrbracket:=\llbracket1,\ell\rrbracket$. We also define 
\equh\label{eq:S^(i)}
\mathbb S\topp i_{\ell} :=\ccbb{(r_1,\dots,r_\ell)\in\Z_{\ge 0}^{\ell}: r_1\le\cdots \le r_\ell \le i}.
\eque
For $i\in\Z_{\ge0}$, we also set $\mathbb S_0\topp i:=\{\emptyset\}$. For $\vvr=(r_1,\dots,r_\ell)$, write $\vvr\topp i:=(r_1,\dots,r_\ell,i)$, with the convention $\emptyset\topp i:=(i)$.

Accordingly, we set
\begin{align}
H_\emptyset(u)\equiv H_{\vvx,\vvc,\emptyset}(u,v) & :=\int_{(0,\infty)^d}e^{-\summ k1d(x_k-x_{k-1})y_k/4}\mu_{-s_1,\dots,-s_d}^{u,v}(\d \vvy),\nonumber\\
H_\vvr(u)\equiv H_{\vvx,\vvc,\vvr}(u,v) & :=\int_{\Lambda_\vvr\topp d}e^{-\summ k1d(x_k-x_{k-1})y_k/4}\mu^{u,v}_{-s_1,\dots,-s_d}(\d\vvy).\nonumber
\end{align}
In the second definition, $1\leq \ell\leq d$ and $\vvr=(r_1,\dots,r_\ell)$ satisfying $0\leq r_1\leq\cdots\leq r_\ell$
 and $r_k\leq j_{u,-s_k},\ k\in\llbracket \ell\rrbracket$.  Set $\mathbb S\topp{-1}_\ell := \emptyset$, $\ell\in \llbracket d\rrbracket$.

\begin{proposition}\label{prop:path decomp}
For $(u,v)\in\calD$, under \eqref{eq:c close}, \eqref{eq:v}, and \eqref{eq:u constraint}, the measure $\mu_{-s_1,\dots,-s_d}^{u,v}$ is concentrated, for $u\in\C_{\pp{s_{j+1,i}/2,s_{j,i}/2}}$, on the disjoint union
\[
\Lambda_{\emptyset}^{(d)}\cup\pp{\bigcup_{\ell=1}^j\bigcup_{\vvr\in\mathbb S\topp {i}_\ell}\Lambda_\vvr\topp d}\cup\pp{\bigcup_{\ell={j+1}}^d\bigcup_{\vvr\in\mathbb S\topp {i-1}_\ell}\Lambda_\vvr\topp d},
\]
for all $i\ge 0, j=0,\dots,d$. 
As a consequence, for $u\in\C_{\pp{s_{j+1,i}/2,s_{j,i}/2}}$,
\equh\label{eq:path decomp}
\Psi_{u,v}(\vvx,\vvc) = H_\emptyset(u)+\summ\ell1j\sum_{\vvr\in\mathbb S_\ell\topp i}H_\vvr(u) + \summ\ell{j+1}d\sum_{\vvr\in\mathbb S_\ell\topp{i-1}}H_\vvr(u).
\eque
\end{proposition}
If $j=0$, then the first double summation on the right-hand side is understood as zero. If $j=0,i=0$, then both double summations on the right-hand side are understood as zero.

Note that for all $(u,v)\in\calD$ and $c_1,\dots,c_d$ under \eqref{eq:c close}, \eqref{eq:v}, and \eqref{eq:u constraint}, it follows that there exist $i \in\Z_{\ge 0}$ and $j=0,\dots,d$  such that $u\in \C_{(s_{j+1,i}/2,s_{j,i}/2)}$. Therefore, for all $(u,v)\in\calD$, the above proposition provides an expression of $\Psi_{u,v}(\vvx,\vvc)$, although the values $\vvc$ depend on the values of $(u,v)$. 
\begin{proof}The proof follows by examining the formulas in Section \ref{sec:CDH}. For $u\in\C_{\pp{s_{j+1,i}/2,s_{j,i}/2}}$, by the definition of $j_{u,-s_k}$ and the assumption $s_1-s_d=c_1+\cdots+c_{d-1}<1$,
 $\llbracket 0,j_{u,-s_k}\rrbracket = \llbracket 0,i\rrbracket$ for $k\in \llbracket  j\rrbracket $ and $\llbracket 0,j_{u,-s_k}\rrbracket =  \llbracket 0,i-1\rrbracket$ for $k\in\llbracket j+1,d\rrbracket$, where we apply the convention that $\llbracket 0,m\rrbracket$ is empty when $m<0$. The assumption \eqref{eq:v} implies that $\rmRe(v)+s_k/2>0$, for $k\in \llbracket d\rrbracket $, and consequently there are no atoms associated with $v$ at any of the $-s_1,\dots,-s_d$. 

We decompose the multiple integral into disjoint $d$-step paths each represented by $\Lambda_\vvr^{(d)}$ or $\Lambda_\emptyset^{(d)}$. The set $\Lambda_\emptyset^{(d)}$ consists of all the paths that remain in the continuous domain $(0,\infty)$ at each of the $d$ steps. If the $k$-th step of a path is at the $r$-th atom associated to $u$ at time $-s_k$ (the location of this atom is $u_{-s_k,r}$), then at the $(k+1)$-st step the path either goes to the $r'$-th atom associated to $u$ at time $-s_{k+1}$ with $r'\ge r$, or to the continuous domain $(0,\infty)$. 

Since there is no continuous-to-discrete component in the transition kernel, once a path enters the continuous domain $(0,\infty)$, it remains there at every subsequent step. Furthermore, since the atomic indices along such a path are nondecreasing, an atomic prefix of length $\ell$ is indexed by $\vvr\in\mathbb S_\ell^{(i)}$ when $\ell\leq j$, and by $\vvr\in\mathbb S_\ell^{(i-1)}$ when $\ell>j$. Finally, if $u_{-s_k,r}$ is a candidate atom, then $\rmRe(u-s_k/2+r)<0$, and hence $u_{-s_k,r}=-4(u-s_k/2+r)^2\notin(0,\infty)$. The map $r\mapsto u_{-s_k,r}$ is injective on the set of candidate indices. Consequently, sets corresponding to different atomic prefixes or to different prefix lengths are pairwise disjoint.
\end{proof}

Proposition \ref{prop:path decomp} provides a formula of $\Psi_{u,v}$ for each strip. We shall show that these are the formulas of $\wt\Psi_{u,v}$ on the corresponding strips. The argument is essentially based on an analytic continuation and is dealt with in Section \ref{sec:proof Laplace}.

\begin{remark}\label{rem:CK}
Our formula for the auxiliary measure is related to the continuous-dual-Hahn formula of \citet{corwin24stationary}, but differs from it in normalization, time convention, and parameter range. Let $\widetilde{\mathfrak p}_t^{u,v}$ denote their marginal measure. On the common real fan-region domain,
\equh\label{eq:CK p}
\widetilde{\mathfrak p}_t^{u,v}
=
\Gamma(u+v+2)\mathfrak p_{-t}^{u,v}.
\eque
Our normalization is also important for analytic continuation. If the Corwin--Knizel normalization were retained, the corresponding unnormalized quantities would carry the meromorphic factor $\Gamma(u+v+2)$, requiring one either to exclude
$u+v+2\in\mathbb Z_{\leq0}$ or to perform the cancellation only at the level of their ratio.

There is also a difference in time convention. The joint-measure formula in \citet{corwin24stationary} is written using the auxiliary process at the times
\[
0=s_{d+1}<s_d<\cdots<s_1.
\]
We instead use the increasing times
$-s_1<\cdots<-s_d$ (note  $t$ and $-t$ in each expression in \eqref{eq:CK p}). This time reversal is convenient but not essential.

Finally, \citet{corwin24stationary} formulated their duality formula 
with the continuous dual Hahn process starting from time $0$, which is unnecessary and can cause their formula to involve atoms. Equivalently, in our convention, they required their sequence of time points to end at 0, so $-s_1<\cdots<-s_d<-s_{d+1}=0$, while ours may all remain negative. When $v<0$, the auxiliary measure has $v$-associated atoms at time $-s_{d+1}=0$
even if the marginal measure at each of the times $-s_1,\ldots,-s_d$ has none. See \citet[Remark~7.2]{corwin24stationary} for more on their choice of the parameters. 
\end{remark}

\section{Proofs for the finite-dimensional Laplace transforms}\label{sec:proof Laplace}
In this section, we prove Theorem \ref{thm:AC}.

In the fan region $u+v>0$, 
$\wt K_{u,v} = K_{u,v}$ follows from combining \cite[Proposition~5.1]{bryc23markov} with \cite[Eq. (1.14), corrected arXiv version]{bryc22markov}. 
For the numerators, for suitably chosen $\vvc$, $\wt\Psi_{u,v}(\vvx,\vvc) = \Psi_{u,v}(\vvx,\vvc)$ essentially follows by computing directly the Laplace transform. This step involves the Yakubovich heat kernel and the Hartman--Watson density, and for this calculation one needs to further impose \equh\label{eq fan}
-2\rmRe(u)<-s_1<\cdots<-s_d<2\rmRe(v).
\eque
The proofs of the initial fan-region identities are scattered in the literature, and we provide a summary in Section \ref{sec:identity fan}.

Starting from the aforementioned range of parameters \eqref{eq fan} within the fan region, we use Lemma \ref{lem:Alexey} to construct holomorphic extensions of $\Psi_{u,v}$ and also $K_{u,v}$ from one strip to the next. 
Lemma \ref{lem:Alexey} is at the core of our proof, and it is inspired by an argument from \citep{bryc23markov}.
At each crossing, we identify the residue terms produced by Lemma \ref{lem:Alexey}. Since $\widetilde{K}_{u,v}$ and $\widetilde{\Psi}_{u,v}(\vvx,\vvc)$ are entire and coincide with the auxiliary expressions on the initial strip, the identity theorem for holomorphic functions yields the desired identity on each successive strip. Iterating this procedure proves the claimed formula for the denominator for all $(u,v)\in\calD$, and for the numerator for parameters satisfying the assumptions in Theorem \ref{thm:AC}.

We prove the key result Lemma \ref{lem:Alexey} in Section \ref{sec:techlem}, and apply it to the denominator in Proposition \ref{prop:denom} in Section \ref{sec:denom}, and to the numerator in Proposition \ref{prop:num} in Section \ref{sec:num}. Theorem \ref{thm:AC} follows from the two propositions immediately. 

We start by proving that both $\wt K_{u,v}$ and $\wt \Psi_{u,v}(\vvx,\vvc)$ are entire functions of $(u,v)\in\C^2$.

\begin{proposition}\label{prop:analytic RN}
Fix $d\in\mathbb Z_{>0}$, $0=x_0<x_1<\cdots<x_d=1$,
and $\vvc\in(0,\infty)^d$. Then the maps $(u,v)\mapsto\wt K_{u,v}$ and $(u,v)\mapsto\wt\Psi_{u,v}(\vvx,\vvc)$ are entire on $\C^2$.
\end{proposition}
\begin{proof}
Let $\B$ be a standard Brownian motion. Consider the random variables
\begin{align*}
X:=\pp{\int_0^1e^{-\sqrt2\B_t}\d t}^{-1},
 & &
Y:=e^{-\sqrt2\B_1}
\pp{\int_0^1e^{-\sqrt2\B_t}\d t}^{-1}.
\end{align*}
Since $X,Y>0$ almost surely, for $u,v\in\C$ we define $X^u:=e^{u\log X}$, and $Y^v:=e^{v\log Y}$,  where $\log X$ and $\log Y$ denote the real logarithms. Then
\[
X^uY^v
=
e^{-\sqrt2v\B_1}
\pp{\int_0^1e^{-\sqrt2\B_t}\d t}^{-(u+v)}.
\]
We will show below that $\esp|X^uY^v|<\infty$ for every
$(u,v)\in\C^2$, so that $\wt K_{u,v}=\esp(X^uY^v)$ is well-defined on $\C^2$. 

Let $S:=\sup_{t\in[0,1]}|\B_t|$. Since
\[
e^{-\sqrt2 S} \leq \int_0^1e^{-\sqrt2\B_t}\d t \leq e^{\sqrt2 S},
\]
we have $|\log X|\leq\sqrt2 S$ and $|\log Y|\leq\sqrt2|\B_1|+|\log X|\leq2\sqrt2 S$.
Consequently, for every compact set $Q\subset\C^2$, there is a constant $C_Q<\infty$ such that $\sup_{(u,v)\in Q}|X^uY^v| \leq e^{C_QS}$. Since $S$ has sub-Gaussian tails, $\esp e^{CS}<\infty$ for every $C>0$. It follows that  $\esp|X^uY^v|<\infty$ for every $(u,v)\in\C^2$, and the preceding estimate gives an integrable bound uniform
on each compact subset of $\C^2$. Hence $\wt K_{u,v}$ is jointly continuous on $\C^2$ by the dominated convergence theorem.
 For $\wt\Psi_{u,v}(\vvx,\vvc)$, the treatment is similar as we can write $\wt\Psi_{u,v}(\vvx,\vvc)=\EE(X^uY^vZ)$ with 
\[
Z:=\exp\pp{-\frac1{\sqrt 2}\summ k1d c_k\B_{x_k}}.
\]

Write $F(u,v) = X^uY^v$ or $X^uY^vZ$. Fix $v\in\C$ first. Then, for every closed triangular contour
$\gamma\subset \C$, Fubini's theorem gives
\[
\int_\gamma \esp F(u,v)\d u = \esp \pp{\int_\gamma F(u,v)\d u} = 0.
\]
Hence, by Morera's theorem, $\esp F(u,v)$ is holomorphic on $\C$ as a function of $u$ for fixed $v\in\C$. The same argument applies to $\esp F(u,v)$ with $u$ fixed as a function of $v\in\C$. Hartogs' theorem then entails that $\esp F(u,v)$ is jointly holomorphic on $\C^2$ and entire.
\end{proof}


\subsection{A technical lemma}\label{sec:techlem}
In this subsection, we state and prove the key lemma of this paper.
This result is inspired by the argument used in the proof of \cite[Theorem~3.4]{bryc23markov}. 
Recall that for $a,b$ satisfying $-\infty\leq a<b\leq \infty$,  we write $\C_{(a,b)}:=\ccbb{z\in\C:\rmRe(z)\in(a,b)}$.
\begin{lemma}\label{lem:Alexey}
Fix $k\in\Z_{\le 0}$ and $\epsilon\in(0,1/2)$. 
Let 
\[
h(u,z) = \frac{\Gamma(u\pm z/2)}{\Gamma(\pm z)}\wt h(u,z)
\] be a meromorphic function on $\C_{(k-\epsilon,k+\epsilon)}\times \C_{(-2\epsilon,2\epsilon)}$ satisfying the following: 
\begin{enumerate}[(i)]
\item $\wt h(u,z)$ is holomorphic on  $\C_{(k-\epsilon,k+\epsilon)}\times \C_{(-2\epsilon,2\epsilon)}$ and $\wt h(u,z) = \wt h(u,-z)$,
\item for every compact set $L$ contained in
$\C_{(k,k+\epsilon)}$ or in $\C_{(k-\epsilon,k)}$, \[\int_{z\in\i\R}
\sup_{u\in L}|h(u,z)|\,|\d z|<\infty\]
and, for every compact set
$L\subset\C_{(k-\epsilon/3,k+\epsilon/3)}$,
\[
\int_{z\in2\epsilon/3+\i\R}
\sup_{u\in L}|h(u,z)|\,|\d z|<\infty,
\]
\item for every compact set
$L\subset\C_{(k-\epsilon/3,k+\epsilon/3)}$,
\begin{align*}
\lim_{K\to\infty} \sup_{u\in L}\int_0^{2\epsilon/3}\left( |h(u,x+\i K)|+|h(u,x-\i K)|\right)\d x =0.
\end{align*}
\end{enumerate} 
Then, 
\[
H(u):=\int_{z\in\i\R}h(u,z)\d z, \quad u\in\C_{(k,k+\epsilon)},
\] 
is a holomorphic function of $u\in\C_{(k,k+\epsilon)}$ that can be extended by analytic continuation to $u\in\C_{(k-\epsilon,k+\epsilon)}$, with 
\[
H(u) = \int_{z\in\i\R}h(u,z)\d z + 8\pi\i\cdot\frac{\Gamma(2u-k) \rmRes_{z=k}{(\Gamma(z))}}{\Gamma(\pm 2(u-k))} \wt h(u,2(u-k)), \quad u\in\C_{(k-\epsilon,k)}.
\]
For $u\in\C_{(k-\epsilon/3,k+\epsilon/3)}$, this continuation is given by
\[
H(u) = \int_{z\in2\epsilon/3+\i\R}h(u,z)\d z - \frac{8\pi\i(u-k)}{(-k)!\,\Gamma(1-2u+k)} \wt h(u,2(u-k)).
\]
\end{lemma}
When the conclusion of the lemma holds, we write, for simplicity,
\begin{multline}\label{eq:Alexey}
\pp{\int_{z\in\i\R} h(u,z)\d z}_{u\in\C_{(k,k+\epsilon)}}
\\\cequiv
\pp{\int_{z\in\i\R}h(u,z)\d z + 8\pi \i\cdot \frac{\Gamma(2u-k) \rmRes_{z=k}{(\Gamma(z))}}{\Gamma(\pm2(u-k))} \wt h(u,2(u-k))}_{u\in\C_{(k-\epsilon,k)}}.
\end{multline}
Here $\cequiv$ means that the function on the right is
obtained from the function on the left by analytic continuation,
with both functions holomorphic on their indicated domains.
\begin{proof}

Condition~$(ii)$ implies that the integral expression for $H(u)$
converges locally uniformly for
$u\in\C_{(k,k+\epsilon)}$. Therefore, $H$ is continuous on $\C_{(k,k+\epsilon)}$. For every closed triangular contour
$\gamma\subset\C_{(k,k+\epsilon)}$, Fubini's theorem gives
\[
\int_\gamma H(u)\,\d u
=
\int_{z\in\i\R}
\left(\int_\gamma h(u,z)\,\d u\right)\d z
=0.
\]
The inner integral vanishes by Cauchy's theorem. Hence, by Morera's
theorem, $H$ is holomorphic on $\C_{(k,k+\epsilon)}$.

Now we explain the analytic continuation argument, which consists of three steps. 
\medskip

\noindent (i) First,  consider and fix $u\in\C_{(k,k+\epsilon/3)}$, and change the contour in the integral expression. Let $\gamma_K$ be
the positively oriented boundary of the rectangle with vertices
\[
-\i K,\quad
2\epsilon/3-\i K,\quad
2\epsilon/3+\i K,\quad
\i K,
\]
where $K>3|\operatorname{Im}(u)|$. The only possible pole of
$h(u,\cdot)$ inside $\gamma_K$ is at $z=2(u-k)$. By the residue
theorem,
\[
\int_{\gamma_K}h(u,z)\,\d z
=
2\pi\i\operatorname{Res}_{z=2(u-k)}h(u,z).
\]
Conditions~$(ii)$ and~$(iii)$ allow us to let $K\to\infty$, giving
\[
\int_{2\epsilon/3+\i\R}h(u,z)\,\d z
-
\int_{\i\R}h(u,z)\,\d z
=
2\pi\i\operatorname{Res}_{z=2(u-k)}h(u,z).
\]
Throughout, all vertical contours are oriented in the positive imaginary direction.
Near $z=2(u-k)$, only the factor $\Gamma(u-z/2)$ has a pole:
\[
\Gamma(u-z/2)
=
-\frac{2\operatorname{Res}_{w=k}\Gamma(w)}
       {z-2(u-k)}+O(1).
\]
Evaluating the other factors at \(z=2(u-k)\), we obtain
\[
\operatorname{Res}_{z=2(u-k)}h(u,z)
=
-2\frac{\Gamma(2u-k)\operatorname{Res}_{w=k}\Gamma(w)}
        {\Gamma(\pm2(u-k))}
\wt h(u,2(u-k))
=:-2R_k(u).
\]
Since
$\operatorname{Res}_{w=k}\Gamma(w)={(-1)^{-k}}/{(-k)!}$, 
the Gamma recurrence gives
\begin{equation}\label{eq:1/2}
R_k(u)
=
-\frac{2(u-k)}
{(-k)!\Gamma(1-2u+k)}
\wt h(u,2(u-k)).
\end{equation}
The reciprocal Gamma function is entire, and
$u\mapsto\wt h(u,2(u-k))$ is holomorphic on the strip
$\C_{(k-\epsilon,k+\epsilon)}$. Thus \eqref{eq:1/2}
extends $R_k$ holomorphically to $\C_{(k-\epsilon,k+\epsilon)}$.
Consequently,
\begin{equation}\label{eq:h(u) 0}
H(u)
=
\int_{2\epsilon/3+\i\R}h(u,z)\,\d z
+4\pi\i R_k(u),
\quad
u\in\C_{(k,k+\epsilon/3)}.
\end{equation}

\noindent (ii) Second, we fix the contour at $2\epsilon/3+\i\R$, but now change the value of $u$. We show that the expression on the right-hand side of
\eqref{eq:h(u) 0} extends holomorphically across
$\operatorname{Re}(u)=k$. Set
\begin{equation}\label{eq:h(u) 1}
F(u):=
\int_{2\epsilon/3+\i\R}h(u,z)\,\d z
+4\pi\i R_k(u).
\end{equation}
For $u\in\C_{(k-\epsilon/3,k+\epsilon/3)}$ and
$z\in2\epsilon/3+\i\R$, neither $\Gamma(u+z/2)$ nor
$\Gamma(u-z/2)$ has a pole. Condition~$(ii)$ and the same Morera
argument as above therefore show that the shifted-contour integral
is holomorphic on $\C_{(k-\epsilon/3,k+\epsilon/3)}$. Together with \eqref{eq:1/2}, this
shows that $F$ is holomorphic on
$\C_{(k-\epsilon/3,k+\epsilon/3)}$. By \eqref{eq:h(u) 0}, it agrees
with $H$ on $\C_{(k,k+\epsilon/3)}$.
\medskip

\noindent (iii) Finally, consider and fix $u\in\C_{(k-\epsilon/3,k)}$, and we now move the contour in the integral formula back from $2\epsilon/3+\i\R$ to $\i\R$. The same
rectangular-contour argument now shows that the only possible pole
between $\i\R$ and $2\epsilon/3+\i\R$ is at $z=2(k-u)$. 
This time near $z=2(k-u)$,
\[
\Gamma(u+z/2)
=
\frac{2\operatorname{Res}_{w=k}\Gamma(w)}
       {z+2(u-k)}+O(1).
\]
Hence, 
\[
\operatorname{Res}_{z=2(k-u)}h(u,z)
=
2
\frac{
\Gamma(2u-k)\operatorname{Res}_{w=k}\Gamma(w)}{\Gamma(2(k-u))\Gamma(-2(k-u))}
\wt h(u,2(k-u))
=2R_k(u),
\]
where in the last step we used $\wt h(u,z)=\wt h(u,-z)$.
It follows that
\[
\int_{2\epsilon/3+\i\R}h(u,z)\,\d z
=
\int_{\i\R}h(u,z)\,\d z
+4\pi\i R_k(u).
\]
Substituting this identity into \eqref{eq:h(u) 1}, we obtain
\[
F(u) = \int_{\i\R}h(u,z)\,\d z +8\pi\i R_k(u), \quad
u\in\C_{(k-\epsilon/3,k)}.
\]

By condition~$(ii)$, the same Morera argument as above, and
\eqref{eq:1/2}, the expression on the right-hand side of
\eqref{eq:Alexey} is holomorphic on $\C_{(k-\epsilon,k)}$.
Thus $H$ on $\C_{(k,k+\epsilon)}$, $F$ on
$\C_{(k-\epsilon/3,k+\epsilon/3)}$, and the right-hand side of
\eqref{eq:Alexey} on $\C_{(k-\epsilon,k)}$ agree on their respective
overlaps. They therefore define the claimed analytic continuation.
\end{proof}

\subsection{The denominator}\label{sec:denom}
Using Lemma \ref{lem:Alexey}, we examine 
\[
K_{u,v} = \int_\C e^{-y/4} \mathfrak p_0^{u,v}(\d y), (u,v)\in\calD.
\]
Recall that $\mathfrak p_0^{u,v}$ is a complex measure. Note that although this expression involves only a single integral, for certain values of $u,v$, $\mathfrak p_0^{u,v}$ has atoms associated to both $u$ and $v$. This feature makes the analysis here more delicate than the one for the numerator. 
\begin{proposition}\label{prop:denom}
For every $(u,v)\in\calD$, we have $K_{u,v}=\wt K_{u,v}$. In particular, $(u,v)\mapsto K_{u,v}$ is holomorphic on $\calD$.
\end{proposition}
\begin{proof}
We first establish the identity for $u,v\in\C_{(0,\infty)}$. For a general target $(u,v)\in\calD$, we then continue along $\lambda\mapsto(u+\lambda,v+\lambda)$, identifying each residue term with the corresponding atomic contribution as $\rmRe(u+\lambda)$ or $\rmRe(v+\lambda)$ crosses a nonpositive integer.

For $u,v\in\C_{(0,\infty)}$, we note that $\mathfrak p_0^{u,v}(\d y) = \mathfrak p_0^{u,v,c}(y)\d y$, and hence 
\begin{align}
K_{u,v} &  = \frac1{8\pi}\int_0^\infty  e^{-y/4}\frac{\Gamma(u\pm\i\sqrt y/2,v\pm\i \sqrt y/2)}{\Gamma(u+v)\sqrt y|\Gamma(\i \sqrt y)|^2}\d y  = \frac1{8\pi \i} \int_{z\in\i \R}e^{z^2/4} \frac{\Gamma(u\pm z/2,v\pm z/2)}{\Gamma(u+v,\pm z)}\d z\nonumber\\
 & = \frac1{8\pi\i}\int_{z\in\i\R}h(u,v,z)\d z,\label{eq:f-1,-1}
\end{align}
where we define 
\[
h(u,v,z) := \frac{\Gamma(u\pm z/2)}{\Gamma(\pm z)}\times \frac{e^{z^2/4}\Gamma(v\pm z/2)}{\Gamma(u+v)}.
\]
By Stirling's formula, the integral in \eqref{eq:f-1,-1} converges locally uniformly for $u,v\in\C_{(0,\infty)}$, so $K_{u,v}$ is holomorphic there. The function $\wt K_{u,v}$ is entire by Proposition \ref{prop:analytic RN}. The fan-region identity reviewed in Section \ref{sec:identity fan} gives $K_{u,v}=\wt K_{u,v}$ for real $u,v>0$. Applying the identity theorem first in $u$ and then in $v$ extends $K_{u,v} = \wt K_{u,v}$ to $u,v\in\C_{(0,\infty)}$. 

For the rest of the proof, we fix $(u,v)\in\calD$ and choose $\lambda_0>\max\{0,-\rmRe(u),-\rmRe(v)\}$. We only consider the case $-\rmRe(u)>0$ or $-\rmRe(v)>0$ (otherwise there is nothing to prove). We analytically continue the contour integral representation of $K_{u+\lambda,v+\lambda}$  in
\eqref{eq:f-1,-1} from $\lambda = \lambda_0$ to $\lambda=0$, and show that the ultimate formula with $\lambda = 0$ is exactly $K_{u,v}$. The equality between $K$ and $\wt K$ established above on $\C_{(0,\infty)}\times\C_{(0,\infty)}$ gives $K_{u+\lambda,v+\lambda}=\wt K_{u+\lambda,v+\lambda}$ for $\lambda$ in a sufficiently small neighborhood of $\lambda_0$. 

As $\lambda$ decreases from $\lambda_0$ to zero, $\rmRe(u+\lambda)$ or $\rmRe(v+\lambda)$ crosses a nonpositive integer at values $\lambda=-\rmRe(u)+k$ or $\lambda=-\rmRe(v)+k$ for $k\in\Z_{\le0}$, lying in $(0,\lambda_0)$. There are finitely many such crossings, and the definition of $\calD$ (in particular, $\rmRe(u)-\rmRe(v)\notin\Z$ when $\rmRe(u),\rmRe(v)<0$) implies that they must be distinct. The definition of $\calD$ also implies there is no crossing at $\lambda=0$ (since $\rmRe(u),\rmRe(v)\notin\Z_{\le 0}$).

We first examine the case that a crossing occurs at $\lambda=-\rmRe(u)+k>0$. We introduce the notation $w=u+\lambda$. In this notation, the second boundary parameter is $v+\lambda=w+v-u$. Choose $0<\epsilon< \min\left\{ 1/2, (1/2)\operatorname{dist} (k+\rmRe(v-u),\Z_{\le0}) \right\}$.
The distance is strictly positive as explained earlier. We apply Lemma \ref{lem:Alexey} to 
\[
h^*(w,z):= h(w,w+v-u,z) = \frac{\Gamma(w\pm z/2)}{\Gamma(\pm z)} \wt h^*(w,z)
\]
with
\[
\wt h^*(w,z) := \frac{e^{z^2/4}\Gamma(w+v-u\pm z/2)}{\Gamma(2w+v-u)}.
\]
Our choice of $\epsilon$ ensures that $\wt h^*(w,z)$ is jointly holomorphic for $w\in\C_{(k-\epsilon,k+\epsilon)}$ and $z\in\C_{(-2\epsilon,2\epsilon)}$. Moreover, $\wt h^*(w,z)=\wt h^*(w,-z)$, so condition~(i) holds. Stirling's formula and the Gaussian factor give the locally uniform integrability and horizontal-contour decay required in conditions~(ii) and~(iii).

Now, by Lemma \ref{lem:Alexey}, we have
\[
\pp{\int_{\i\R} h^*(w,z)\d z}_{w\in\C_{(k,k+\epsilon)}}\cequiv
\pp{\int_{\i\R} h^*(w,z)\d z+ 8\pi\i R_{k}^*(w)}_{w\in\C_{(k-\epsilon,k)}}
\]
with 
\begin{align*}
R_{k}^*(w)  
&= 
-\frac{(2w-2k)e^{(w-k)^2}\Gamma(v-u+k)[2w+v-u]_{-k}}{(-k)!\Gamma(1-2w+k) }\\
& = e^{(w-k)^2} \mathfrak p_0^{w,w+v-u,d} \pp{-4(w-k)^2},
\end{align*}
for all $(w,w+v-u) = (u+\lambda,v+\lambda)\in\calD$ and $\rmRe(w)<k$, 
where we used the residue simplification \eqref{eq:1/2}, the identity $\Gamma(a-k)/\Gamma(a)=[a]_{-k}$, and the formula~\eqref{eq:pd u'} of $\mathfrak p_0^{u,v,d}$.

The factor $\Gamma(v-u+k)$ is independent of $\lambda$ ($w=u+\lambda$ and $u$ is fixed); it is finite because $\rmRe(v-u+k)\notin\Z_{\le0}$. All the remaining terms are entire in $\lambda$, and hence $R_k^*(w)$ is entire in $\lambda$. So once a residue term $R_k^*(w)$ is created when passing $\rmRe(w) =k$, the term remains in the subsequent crossings and at the end, when $\lambda = 0$, this term becomes (recall that $u_{0,j} = -4(u+j)^2, j=0,\dots,j_{u,0}$), 
\equh\label{eq:R u}
R_k^*(w)\bigg\vert_{w = u} = e^{(u-k)^2} \mathfrak p_0^{u,v,d} \pp{-4(u-k)^2} = \int_{\{u_{0,-k}\}}e^{-y/4}\mathfrak p_0^{u,v}(\d y).
\eque

The case that a crossing occurs at $\lambda=-\rmRe(v)+k>0$ can be treated similarly: setting $w' = v+\lambda$ one can repeat the same calculation above. This time, the residue term when crossing $\rmRe(v)+\lambda = k\in\Z_{\le 0}$ is
\[
R^{**}_k(w') = -\frac{(2w'-2k)e^{(w'-k)^2}\Gamma(u-v+k)[2w'+u-v]_{-k}}{(-k)!\Gamma(1-2w'+k)},
\]
The same argument as before shows that $R^{**}_k(w')$ is an entire function over the strip that contains the subsequent crossings. Therefore, at the end it contributes to a term
\equh\label{eq:R v}
R^{**}_k(w')\bigg\vert_{w' = v} = \int_{\{v_{0,-k}\}}e^{-y/4}\mathfrak p_0^{u,v}(\d y).
\eque

We have examined all the crossings of $(u+\lambda,v+\lambda)$ from $\lambda = \lambda_0$ decreasing to $0$. Summing all the terms \eqref{eq:R u} and \eqref{eq:R v} we have shown that 
the analytic extension of $\int_0^\infty e^{-y/4}\mathfrak p_0^{u,v}(\d y)$ from $(u+\lambda_0,v+\lambda_0)$ to $(u,v)$ along the path $(u+\lambda,v+\lambda)$ has the expression
\[
\frac1{8\pi\i}\int_{\i\R}h(u,v,z)\d z + \sum_{k\in\Z_{\le 0},\rmRe(u)<k}R_k^*(u) + \sum_{k\in\Z_{\le 0},\rmRe(v)<k}R_k^{**}(v) = \int_\C e^{-y/4}\mathfrak p_0^{u,v}(\d y),
\]
which is exactly $K_{u,v}$ defined earlier. 
Since the function $\lambda\mapsto\wt K_{u+\lambda,v+\lambda}$ is entire by Proposition \ref{prop:analytic RN} and coincides with $K_{u+\lambda,v+\lambda}$ in a neighborhood of $\lambda_0$, uniqueness of analytic continuation therefore gives $K_{u,v}=\wt K_{u,v}$.  \end{proof}

\subsection{The numerator}\label{sec:num}
For convenience, we recall some notation introduced in Section \ref{sec:path}.
We write $\Lambda_{\emptyset}^{(d)}:=(0,\infty)^d$, and 
\[
\Lambda\topp d_{(r_1,\dots,r_\ell)}\equiv \Lambda\topp d_{(r_1,\dots,r_\ell)}(u,\vv c) :=\{u_{-s_1,r_1}\}\times\cdots\times\{u_{-s_\ell,r_\ell}\}\times(0,\infty)^{d-\ell},
\]
with $r_k\in\llbracket 0, j_{u,-s_k}\rrbracket, k\in \llbracket \ell\rrbracket$. We also write
\[
\mathbb S\topp i_{\ell} :=\ccbb{(r_1,\dots,r_\ell)\in\Z_{\ge 0}^{\ell}: r_1\le\cdots \le r_\ell \le i}.
\]
We set $\mathbb S_0\topp i:=\{\emptyset\}$ for $i\in\Z_{\ge0}$. 
Accordingly, we set
\begin{align*}
H_\emptyset(u)\equiv H_{\vvx,\vvc,\emptyset}(u,v) & :=\int_{(0,\infty)^d}e^{-\summ k1d(x_k-x_{k-1})y_k/4}\mu_{-s_1,\dots,-s_d}^{u,v}(\d \vvy),\nonumber\\
H_\vvr(u)\equiv H_{\vvx,\vvc,\vvr}(u,v) & :=\int_{\Lambda_\vvr\topp d}e^{-\summ k1d(x_k-x_{k-1})y_k/4}\mu^{u,v}_{-s_1,\dots,-s_d}(\d\vvy).
\end{align*}
The second definition applies when  $1\leq \ell\leq d$ and $\vvr=(r_1,\dots,r_\ell)$ satisfying $0\leq r_1\leq\cdots\leq r_\ell$
 and $r_k\leq j_{u,-s_k}$ for all $k\in\llbracket \ell\rrbracket$.  Set $\mathbb S\topp{-1}_\ell := \emptyset$, $\ell\in \llbracket d\rrbracket$. 
Recall also $s_{j,i} := s_j-2i, j=1,\dots, d$,
and 
$s_{d+1,i} := s_{1,i+1}$ and $s_{0,i+1}:=s_{d,i}$ for $i\in\Z_{\ge 0}$  with the convention $s_{0,0}:=+\infty$.
Recall that we are under the assumptions \eqref{eq:c close}, \eqref{eq:v}, and \eqref{eq:u constraint}: that is, $c_1,\dots,c_d>0, \summ j1{d-1}c_j<1, -s_d<2\rmRe(v), \rmRe(u)-s_j/2\notin\Z_{\le0}$. 
We establish the following result on the numerator.
 \begin{proposition}
\label{prop:num}
For $(u,v)\in\calD$, 
\[
\pp{H_\emptyset(u)}_{u\in\C_{(s_1/2,\infty)}}\cequiv\pp{H_\emptyset(u)+\summ\ell1j\sum_{\vvr\in\mathbb S_\ell\topp i}H_\vvr(u) + \summ\ell{j+1}d\sum_{\vvr\in\mathbb S_\ell\topp{i-1}}H_\vvr(u)}_{u\in\C_{\pp{s_{j+1,i}/2,s_{j,i}/2}}},
\]
for all $i\in\Z_{\ge0}, j\in \llbracket 0,d\rrbracket $.  
As a consequence, $\Psi_{u,v}(\vvx,\vvc) = \wt\Psi_{u,v}(\vvx,\vvc)$ for $u\in \C_{(s_{j+1,i}/2,s_{j,i}/2)}$. 
\end{proposition}

When $d=1$, we have $x_1=1$ and $s_1=c_1$. The marginal measure formulas and the definitions of $\wt\Psi$ and $\wt K$ give
\begin{align*}
\Psi_{u,v}(1,c_1) =K_{u-c_1/2,v+c_1/2},
 & &  \wt\Psi_{u,v}(1,c_1) =\wt K_{u-c_1/2,v+c_1/2}.
\end{align*} Since $-c_1\in\mathbb T_{u,v}$, we have $\rmRe(u-c_1/2)\notin\Z_{\le0}$ and $\rmRe(v+c_1/2)>0$, so $(u-c_1/2,v+c_1/2)\in\calD$. Thus Proposition \ref{prop:denom} proves the claimed result. From now on, we assume $d\ge2$. 

\subsubsection{The beginning regime}
We first work under the parameter restriction
\equh\label{eq:uv good}
-2\rmRe(u)<-s_1<\cdots<-s_d<2\rmRe(v).
\eque
For real $u>s_1/2$ and $v>0$, it is known that 
\[
H_\emptyset(u)\equiv H_{\vvx,\vvc,\emptyset}(u,v)
=\Psi_{u,v}(\vvx,\vvc)
=\wt\Psi_{u,v}(\vvx,\vvc).
\]
This step follows from the so-called duality formula \citep{bryc23markov} and results on the Yakubovich heat kernel \citep{bryc22markov}. For convenience we present a proof in Appendix \ref{sec:identity fan}.

Moreover, the contour integral in \eqref{eq:Psi_u,v 1} and uniform
Stirling estimates give a jointly holomorphic extension of
$H_{\vvx,\vvc,\emptyset}(u,v)$ to
\[
(u,v)\in \C_{(s_1/2,\infty)}\times\C_{(-s_d/2,\infty)}.
\]
The function $\wt\Psi_{u,v}(\vvx,\vvc)$ is also holomorphic there by
Proposition \ref{prop:analytic RN}. Since these functions agree for real
$u>s_1/2$ and $v>0$, applying the identity theorem first in $u$
and then in $v$ shows that they agree throughout $\C_{(s_1/2,\infty)}\times\C_{(-s_d/2,\infty)}$.
Consequently,
$\Psi_{u,v}(\vvx,\vvc)=\wt\Psi_{u,v}(\vvx,\vvc)$
for $(u,v)$ in the intersection of this domain with $\calD$. (The extension above does not require the restriction $\min(u,v)>-1$ for real parameters in \citet[Proposition~1.4]{bryc23markov}, which was needed there when matching their duality formula with the range of Laplace parameters covered by \citet{corwin24stationary}.)

Fix $v$ with $\rmRe(v)>-s_d/2$ and $\rmRe(v)\notin\Z_{\le0}$. Since $\wt\Psi_{u,v}(\vvx,\vvc)$ is entire, by uniqueness of analytic continuation, it remains only to prove that the expression defining $\Psi_{u,v}(\vvx,\vvc)$, as a function of $u$, on each admissible strip is obtained by analytic continuation from the initial strip $\C_{(s_1/2,\infty)}$. Each time a boundary is crossed, new multiple-integral terms appear in the expression in the next strip, essentially due to the residues collected when a contour crosses certain poles. The proof is by an induction argument. 

\subsubsection{An overview of the argument}
We illustrate the induction by showing how the atomic contributions accumulate as $\rmRe(u)$ decreases. 

\begin{enumerate}[(i)]
\item First, we shall show that 
\equh\label{eq:deform H_emptyset}
\pp{H_\emptyset(u)}_{u\in\C_{(s_1/2,\infty)}}
\cequiv
\pp{H_\emptyset(u)+H_{(0)}(u)}
_{u\in\C_{(s_2/2,s_1/2)}}.
\eque
\item Next, we have (assuming $d\ge 3$)
\begin{align*}
\pp{H_\emptyset(u)}_{u\in\C_{(s_2/2,s_1/2)}}&\cequiv\pp{H_\emptyset(u)}_{u\in\C_{(s_3/2,s_2/2)}},
\\
\pp{H_{(0)}(u)}_{u\in\C_{(s_2/2,s_1/2)}}& \cequiv\pp{H_{(0)}(u)+H_{(0,0)}(u)}_{u\in\C_{(s_3/2,s_2/2)}}.
\end{align*}
In fact, $H_\emptyset$ is holomorphic over $\C_{(s_3/2,s_1/2)}$.
\item Iterating the procedure we obtain, with $\vv0_\ell:=(0,\ldots,0)\in\mathbb S_\ell^{(0)}$,
\[
\pp{H_\emptyset(u)}_{u\in\C_{(s_1/2,\infty)}}\cequiv\pp{H_\emptyset(u)+\summ \ell1d H_{\vv0_\ell}(u)}_{u\in\C_{(s_1/2-1,s_d/2)}},
\]
Recall that we assume $\summ i1{d-1}c_i<1$. By the conventions $s_{j,i}=s_j-2i$ and $s_{0,i+1}=s_{d,i}$, we have $(s_1/2-1,s_d/2) = (s_{1,1}/2,s_{0,1}/2)$.

\item Applying the same argument as in
\eqref{eq:deform H_emptyset} at the next boundary gives
\[
\pp{H_{\emptyset}(u)}_{u\in\C_{(s_{1,1}/2,s_{0,1}/2)}}\cequiv\pp{H_{\emptyset}(u)+H_{(1)}(u)}_{u\in\C_{(s_{2,1}/2,s_{1,1}/2)}}.
\]
The terms $H_{\vv0_\ell}(u)$ also extend unchanged across this boundary. Thus, 
\[
\pp{H_{\emptyset}(u)}_{u\in\C_{(s_{1}/2,\infty)}}\cequiv\pp{H_{\emptyset}(u)+\summ \ell1 dH_{\vv0_\ell}(u)+H_{(1)}(u)}_{u\in\C_{(s_{2,1}/2,s_{1,1}/2)}}.
\]
\item Next, when extending from $\C_{(s_{2,1}/2,s_{1,1}/2)}$ to $\C_{(s_{3,1}/2,s_{2,1}/2)}$, the  function $H_\vvr(u)$ remains unchanged for all $\vvr$ involved with {\em two} exceptions: $\vvr = (0)$ and $\vvr = (1)$. We have
\[
\pp{H_{(r)}(u)}_{u\in \C_{(s_{2,1}/2,s_{1,1}/2)}}\cequiv\pp{H_{(r)}(u)+H_{(r,1)}(u)}_{u\in \C_{(s_{3,1}/2,s_{2,1}/2)}}, \quad r=0,1.
\]
Note that this time, when a pole is crossed, multiple (two here) terms show up in the extension. 
\end{enumerate}
In each of the steps above, we apply Lemma \ref{lem:Alexey}. We repeat this procedure to obtain continuation to domains in which $\rmRe(u)$ is arbitrarily small.

\subsubsection{Proof by induction}

In the rest of this section, we first provide a detailed proof of \eqref{eq:deform H_emptyset} to illustrate the method, and then prove a general statement for the induction in Lemma \ref{lem:induction}.
\begin{proof}[Proof of \eqref{eq:deform H_emptyset}]
We first apply Lemma \ref{lem:Alexey} to continue $H_\emptyset(u)$ locally across the line $\rmRe(u)=s_1/2$. We then identify the resulting residue term with $H_{(0)}(u)$ and extend the continuation. Through this approach, we will establish that
\[
\pp{H_\emptyset(u)}
_{u\in\C_{(s_1/2,\infty)}}
\cequiv
\pp{H_\emptyset(u)+H_{(0)}(u)}
_{u\in\C_{(s_2/2,s_1/2)}}.
\]
After substituting \eqref{eq:pc} and \eqref{eq:transitional cc} into the definition of $H_\emptyset(u)$, and setting $z_k=\i\sqrt{y_k}$, we obtain
\begin{multline}\label{eq:Psi_u,v 1}
H_\emptyset(u) = \frac1{8\pi\i\Gamma(u+v)}\int_{(\i\R)^d}e^{\summ k1d (x_k-x_{k-1})z^2_k/4}
\\ \times \ddfrac{\Gamma\pp{u-\frac{s_1}2\pm \frac{z_1}2, v+\frac{s_d}2 \pm \frac{z_d}2}}{\Gamma(\pm z_1)} \prodd k1{d-1}\ddfrac{\Gamma\pp{\frac{c_k+z_k}2\pm \frac{z_{k+1}}2, \frac{c_k-z_k}2\pm \frac{z_{k+1}}2}}{ 8\pi\i\Gamma(c_k)\Gamma\left(\pm z_{k+1}\right)}\d \vv z. \end{multline} 
Here $\mathrm d\vvz:=\mathrm dz_1\cdots\mathrm dz_d$. Under \eqref{eq:uv good}, the Gamma factors in the numerator have no poles on the contours. One can check that $H_\emptyset(u)$ is holomorphic over $\C_{(s_1/2,\infty)}$.

We rewrite \eqref{eq:Psi_u,v 1} as
\[
H_\emptyset(u) = \frac1{\Gamma(u+v)}\frac1{8\pi\i} \int_{\i\R}\ddfrac{\Gamma\pp{u-\frac{s_1}2\pm\frac {z_1}2}}{\Gamma(\pm z_1)}I_1(z_1)\d z_1
\]
with
\begin{multline*}
I_1(z_1) :=\frac1{(8\pi\i)^{d-1}}\int_{(\i\R)^{d-1}}e^{\summ k1d(x_k-x_{k-1})z_k^2/4}\\
\times
\prodd k1{d-1}\frac{\displaystyle\Gamma\left( \frac{c_k+z_k}{2}\pm\frac{z_{k+1}}2, \frac{c_k-z_k}{2}\pm\frac{z_{k+1}}2\right)}{\displaystyle 
\Gamma(c_k,\pm z_{k+1})}\cdot\Gamma\pp{v+\frac{s_d\pm z_d}2}\d \vv z_{2:d}.
\end{multline*} 
Here  $\mathrm d\vvz_{2:d}:=\mathrm dz_2\cdots\mathrm dz_d$. Choose $0<\epsilon<
\min\left\{ 1/2,c_1/2, \rmRe(v)+s_1/2 \right\}$. For this choice of $\epsilon$, the function $I_1$ is holomorphic and even on $\C_{(-2\epsilon,2\epsilon)}$. 

Applying Lemma \ref{lem:Alexey}, we obtain 
\begin{multline}\label{eq:1st step}
\pp{
\frac1{8\pi\i}
\int_{\i\R}
\ddfrac{
\Gamma\pp{u-\frac{s_1}{2}\pm\frac{z_1}{2}}
}{
\Gamma(\pm z_1)
}
I_1(z_1)\,\d z_1
}_{u\in\C_{(s_1/2,s_1/2+\epsilon)}}
\\
\cequiv
\pp{
\frac1{8\pi\i}
\int_{\i\R}
\ddfrac{
\Gamma\pp{u-\frac{s_1}{2}\pm\frac{z_1}{2}}
}{
\Gamma(\pm z_1)
}
I_1(z_1)\,\d z_1
+
\frac{I_1(2u-s_1)}
{\Gamma(-2u+s_1)}
}_{u\in\C_{(s_1/2-\epsilon,s_1/2)}}.
\end{multline}
For $d\geq2$, we note that $s_1-s_2=c_1$, and recall that we assume $0<c_1<1$. Therefore, for $u\in\C_{(s_2/2,s_1/2)}$ and $z_1\in\i\R$, \[-\frac{c_1}{2}<\rmRe\left(u-\frac{s_1}{2}\pm\frac{z_1}{2}\right)<0.\]
The integral representation of $I_1$ shows that it is
holomorphic on $\C_{(-c_1,c_1)}$. Since $u\in\C_{(s_2/2,s_1/2)}$ implies $2u-s_1\in\C_{(-c_1,0)}$, the residue term is holomorphic on $\C_{(s_2/2,s_1/2)}$. Therefore, the expressions on the left-hand side and the right-hand side of \eqref{eq:1st step}  extend holomorphically to $\C_{(s_1/2,\infty)}$ and $\C_{(s_2/2,s_1/2)}$, respectively.

Furthermore, we can write
\[
I_1\pp{2u-s_1}  = \frac{e^{x_1(u-s_1/2)^2}}{8\pi\i}\int _{\i\R}\ddfrac{\Gamma\pp{u-\frac{s_2}2\pm\frac{z_2}2}}{\Gamma(\pm z_2)}
\ddfrac{\Gamma\pp{\frac{c_1}2-u+\frac{s_1}2\pm \frac{z_2}2}}{\Gamma(c_1)}I_2(z_2)\d z_2
\]
with
\begin{multline*}
I_2(z_2) := 
\frac1{(8\pi\i)^{d-2}}
\int_{(\i\R)^{d-2}}e^{\summ k2d(x_k-x_{k-1})z_k^2/4}\\
\times
\prodd k2{d-1}\frac{\displaystyle\Gamma\left( \frac{c_k+z_k}{2}\pm\frac{z_{k+1}}2, \frac{c_k-z_k}{2}\pm\frac{z_{k+1}}2\right)}{\displaystyle 
\Gamma(c_k,\pm z_{k+1})} \Gamma\pp{v+\frac{s_d\pm z_d}2}\d \vv z_{3:d}.
\end{multline*}
Then we can write \begin{align*}
H_{(0)}(u) & = \frac1{8\pi\i}\int_{\i\R}\ddfrac{\Gamma\pp{u-\frac{s_2}2\pm\frac{z_2}2}}{\Gamma(\pm z_2)} e^{x_1(u-s_1/2)^2}\ddfrac{\Gamma\pp{\frac{c_1}2-u+\frac{s_1}2\pm \frac{z_2}2}}{\Gamma\pp{c_1,u+v,-2u+s_1}}I_2(z_2)\d z_2.
\end{align*}
(See also the more general formula for $H_\vvr(u)$ in \eqref{eq:H_vvr factorization} below.)

We have thus shown
\begin{align*}
\pp{H_\emptyset(u)}_{u\in\C_{(s_1/2,\infty)}} & \cequiv\pp{H_\emptyset(u)+\frac{I_1(2u-s_1)}{\Gamma(u+v,-2u+s_1)}}_{u\in\C_{(s_2/2,s_1/2)}}
\\
&
 \equiv\pp{H_\emptyset(u)+H_{(0)}(u)}_{u\in\C_{(s_2/2,s_1/2)}},
\end{align*}
as claimed. 
\end{proof}
An almost identical calculation yields that for every $i\in\Z_{\geq0}$,
\begin{equation}\label{eq:deform H_emptyset general}
\pp{H_{\emptyset}(u)}_
{u\in\C_{\pp{s_{1,i}/2,s_{0,i}/2}}}
\cequiv
\pp{H_{\emptyset}(u)+H_{(i)}(u)}_
{u\in\C_{\pp{s_{2,i}/2,s_{1,i}/2}}}.
\end{equation}
The following lemma, together with \eqref{eq:deform H_emptyset general}, yields Proposition~\ref{prop:num} by induction.
\begin{lemma}\label{lem:induction}
Fix $i\in\Z_{\geq0}$. The following statements hold: 
\begin{enumerate}[(a)]
\item For each $\ell\in\llbracket d-1\rrbracket $ and $\vvr=(r_1,\dots,r_\ell)\in \mathbb S_\ell\topp i$, setting $\vvr\topp i:=(r_1,\dots,r_\ell,i)$ we have
\equh\label{eq:H_vvr ell}
\pp{H_\vvr(u)}_{u\in\C_{\pp{s_{\ell+1,i}/2,s_{\ell,i}/2}}}\cequiv
\pp{H_{\vvr}(u) + H_{\vvr\topp i}(u)}_{u\in\C_{\pp{s_{\ell+2,i}/2,s_{\ell+1,i}/2}}}.
\eque
\item For each $\ell\in\llbracket d-1\rrbracket $ and $\vvr\in\spp{\bigcup_{\ell'=1}^{\ell-1}\mathbb S_{\ell'}\topp i} \cup \spp{\bigcup_{\ell'=\ell+1}^d\mathbb S_{\ell'}\topp{i-1}}$, we have
\[
\pp{H_\vvr(u)}_{u\in\C_{\pp{s_{\ell+1,i}/2,s_{\ell,i}/2}}}\cequiv\pp{H_{\vvr}(u)}_{u\in\C_{\pp{s_{\ell+2,i}/2,s_{\ell,i}/2}}}.
\]
\item For all $\vvr\in \mathbb S_d\topp i$, $H_\vvr(u)$ is holomorphic for $u\in\C_{(-\infty,s_{d,i}/2)}$.
\end{enumerate}
\end{lemma}
To prove Lemma \ref{lem:induction}, we first rewrite  $H_\vvr(u)$ for convenience. Fix $i\in\Z_{\geq0}$, $\ell\in\llbracket d\rrbracket$, and
$\vvr=(r_1,\dots,r_\ell)\in\mathbb S_\ell\topp i$. 
Suppose $(u,v)\in\calD$ and $-s_1,\dots,-s_d\in\mathbb T_{u,v}$.
If $\rmRe(u)<s_{\ell,r_\ell}/2$, then
$r_k\le j_{u,-s_k}$ for every $k\le\ell$. Indeed,
$s_k\ge s_\ell$ and $r_k\le r_\ell$, so
\[
\rmRe(u)-s_k/2+r_k
\le \rmRe(u)-s_\ell/2+r_\ell<0.
\]
For such $u$, we define
\begin{align}
M_\vvr(u)& \equiv M_{\vvx,\vvc,\vvr}(u,v)\nonumber\\
& :=  e^{-\summ k1{\ell}(x_k-x_{k-1})u_{-s_k,r_k}/4}\mathfrak p_{-s_1}^{u,v,d}(u_{-s_1,r_1})\prodd k1{\ell-1}\mathfrak p_{-s_k,-s_{k+1}}^{u,v,d,d}(u_{-s_k,r_k},u_{-s_{k+1},r_{k+1}}).\nonumber
\end{align}
For $\ell\in \llbracket d-1\rrbracket$, define the function $I_\ell$, which encodes
the remaining contour integrations, by
\begin{multline*}
I_\ell(z_\ell) :=
\frac{e^{(x_\ell-x_{\ell-1})z_\ell^2/4} }{(8\pi\i)^{d-\ell}}
\int_{(\i\R)^{d-\ell}}e^{\summ k{\ell+1} d(x_k-x_{k-1})z_k^2/4}\\
\times
\prodd k\ell{d-1}\ddfrac{\Gamma\left( \frac{c_k+z_k}{2}\pm\frac{z_{k+1}}2, \frac{c_k-z_k}{2}\pm\frac{z_{k+1}}2\right)}{\Gamma(c_k,\pm z_{k+1})}\times \Gamma\pp{v+\frac{s_d\pm z_d}2}\d \vv z_{\ell+1:d}.
\end{multline*}
This formula is consistent with $I_1$ and $I_2$ introduced earlier. We further define 
\equh\label{eq:I_d}
I_d(z_d) := e^{(x_d-x_{d-1})z_d^2/4}\Gamma\pp{v+\frac{s_d\pm z_d}2}.
\eque
In the next lemma we express $H_\vvr$ in terms of $M_\vvr$ and $I_{\ell+1}$ and record properties needed for the
induction steps.
\begin{lemma}\label{lem:H-factorization}
Let $\ell\in\llbracket d\rrbracket $, and suppose $\vvr=(r_1,\dots,r_\ell)\in\mathbb S_\ell\topp i$ for some $i\in\Z_{\geq0}$.
\begin{enumerate}[(i)]
\item The function $M_\vvr(u)$ extends holomorphically over $\C_{(-\infty, s_{\ell,r_\ell}/2)}$. 
\item If $\ell\in\llbracket d-1\rrbracket$, then, writing
$j:=r_\ell$, we have
\equh\label{eq:H_vvr factorization}
H_\vvr(u) = \frac{M_\vvr(u)}{8\pi\i}\int_{\i\R}\ddfrac{\Gamma\pp{u-\frac{s_{\ell+1,j}}2\pm\frac{z_{\ell+1}}2}}{\Gamma(\pm z_{\ell+1})} \ddfrac{\Gamma\bpp{\frac{c_{\ell}}2-\pp{u-\frac{s_{\ell,j}}2}\pm\frac{z_{\ell+1}}2}}{\Gamma(c_{\ell},v+u+j,v-u+s_{\ell}-j)}I_{\ell+1}(z_{\ell+1})\d z_{\ell+1},
\eque
and it is holomorphic over $\C_{(s_{\ell+1,r_\ell}/2,s_{\ell,r_\ell}/2)}$ and $\C_{(s_{\ell+1,k+1}/2, s_{\ell+1,k}/2)}$ for $k\ge r_\ell$. 
\item If $\ell=d$, then $H_\vvr(u) = M_\vvr(u)$. 
\end{enumerate}
\end{lemma}
\begin{proof}
We only check the rewriting in \eqref{eq:H_vvr factorization}. The rest is routine and omitted. 
By definition, 
\begin{multline*}
H_\vvr(u) = e^{-\summ k1{\ell}(x_k-x_{k-1})u_{-s_k,r_k}/4} \int_{(0,\infty)^{d-\ell}}e^{-\summ k{\ell+1}d(x_k-x_{k-1})y_k/4}\mathfrak p_{-s_1}^{u,v,d}(u_{-s_1,r_1})
\\
\times\prodd k1{\ell-1} \mathfrak p_{-s_k,-s_{k+1}}^{u,v,d,d}(u_{-s_k,r_k},u_{-s_{k+1},r_{k+1}})   \mathfrak p_{-s_{\ell},-s_{\ell+1}}^{u,v,d,c}\pp{u_{-s_{\ell},r_{\ell}},y_{\ell+1}}\prodd k{\ell+1}{d-1}\mathfrak p_{-s_k,-s_{k+1}}^{u,v,c,c}(y_k,y_{k+1})\d\vvy_{\ell+1:d}.
\end{multline*}
For all $j\in \llbracket 0, j_{u,s}\rrbracket , j'\in\llbracket j, j_{u,t}\rrbracket $, we re-write $\mathfrak p_{s,t}^{u,v,d,d}$ in the following form
\begin{multline*}
\mathfrak p_{s,t}^{u,v,d,d}(u_{s,j},u_{t,j'}) = \frac{\Gamma(v-u-t-j',-2u-s-j-j')}{\Gamma(v-u-s-j,-2u-t-2j)}\\
\times\frac{[2u+t+2j,v+u+j,t-s]_{j'-j}(2u+t+2j')}{2u+t+2j}\frac{(-1)^{j'-j}}{(j'-j)!}, 
\end{multline*}
Recalling the expressions for $\mathfrak p^{u,v,d,c}_{-s_\ell,-s_{\ell+1}},\mathfrak p^{u,v,c,c}_{-s_k,-s_{k+1}}$ in \eqref{eq:transitional dc} and \eqref{eq:transitional cc}, canceling Gamma factors involving $v$, and setting $z_k=\i\sqrt{y_k}$ for $k\in\llbracket \ell+1,d\rrbracket $, we obtain \eqref{eq:H_vvr factorization}.
The case $\ell=d$ is immediate.
\end{proof}
\begin{proof}[Proof of Lemma \ref{lem:induction}]
To prove part~\textup{(a)}, fix $\ell\in\llbracket d-1\rrbracket$ and $\vvr=(r_1,\dots,r_\ell)\in\mathbb S_\ell\topp i$, and set $j:=r_\ell$ and $k:=i-j$. 
The contour integral in \eqref{eq:H_vvr factorization} has the form required in Lemma \ref{lem:Alexey}. This time, we apply Lemma \ref{lem:Alexey} to the variable $u-s_{\ell+1,j}/2$, with the integer parameter in that lemma equal to $-k$, to obtain (for $\epsilon>0$ small enough)
\equh\label{eq:H*}
\pp{H_\vvr(u)}_{u\in\C_{\pp{s_{\ell+1,i}/2,s_{\ell+1,i}/2+\epsilon}}}\cequiv
\pp{H_\vvr(u)+H_\vvr^*(u)}_{u\in\C_{\pp{s_{\ell+1,i}/2-\epsilon,s_{\ell+1,i}/2}}},
\eque
where $H_\vvr^*(u)$ denotes the resulting residue term. 

We will show
that $H_\vvr^*(u)=H_{\vvr\topp i}(u)$. By Lemma \ref{lem:Alexey}, \eqref{eq:H*} holds with 
\begin{multline*}
H_\vvr^*(u)= \frac{M_\vvr(u)}{\Gamma(c_\ell,v+u+j,v-u+s_\ell-j)}\Gamma\pp{\frac{c_\ell}2-\pp{u-\frac{s_{\ell,j}}2\pm \frac{z_{\ell+1}}2}}\Bigg\vert_{z_{\ell+1} = 2u-s_{\ell+1,i}}\\
 \times  I_{\ell+1}\pp{2u-s_{\ell+1,i}} \ddfrac{\Gamma(2u-s_{\ell+1}+j+i)}{\Gamma\pp{\pm(2u-s_{\ell+1}+2i)}} \frac{(-1)^k}{k!}.
\end{multline*}
Substituting $z_{\ell+1}=2u-s_{\ell+1,i}$ yields
\[
\Gamma\pp{\frac{c_\ell}2-\pp{u-\frac{s_{\ell,j}}2\pm \frac{z_{\ell+1}}2}}\Bigg\vert_{z_{\ell+1} = 2u-s_{\ell+1,i}} = \Gamma\pp{c_\ell+k,s_\ell-2u-j-i}.
\]
Moreover, for $\ell\in\llbracket d-1\rrbracket $,
\[
I_\ell(z_\ell) = \frac{e^{(x_\ell-x_{\ell-1})z_\ell^2/4}}{8\pi\i}\int_{\i\R}\ddfrac{\Gamma\pp{\frac{c_\ell+z_\ell}2\pm\frac{z_{\ell+1}}2,\frac{c_\ell-z_\ell}2\pm \frac{z_{\ell+1}}2}}{\Gamma(c_\ell,\pm z_{\ell+1})}I_{\ell+1}(z_{\ell+1})\d z_{\ell+1}.
\] 
Assume first that $\ell\in\llbracket d-2\rrbracket $. Applying the preceding identity with $\ell$ replaced by $\ell+1$, we obtain
\begin{multline*}
I_{\ell+1}\pp{2u-s_{\ell+1,i}}  = \frac{e^{(x_{\ell+1}-x_\ell)(u-s_{\ell+1,i}/2)^2}}{8\pi\i}\\
\times \int_{\i\R}\ddfrac{\Gamma\bpp{u-\frac{s_{\ell+2,i}}2 \pm \frac{z_{\ell+2}}2}}{\Gamma(c_{\ell+1},\pm z_{\ell+2})}\Gamma\pp{\frac{c_{\ell+1}}2-\pp{u-\frac{s_{\ell+1,i}}2}\pm \frac{z_{\ell+2}}2}I_{\ell+2}(z_{\ell+2})\d z_{\ell+2}.
\end{multline*}
By the definition of $M_{\vvr\topp{i}}$, we have
\[
M_{\vvr\topp{i}}(u)
=
M_\vvr(u)
\exp{\left(
-\frac{x_{\ell+1}-x_\ell}{4}
u_{-s_{\ell+1},i}
\right)}
\mathfrak p_{-s_\ell,-s_{\ell+1}}^{u,v,d,d}
\left(
u_{-s_\ell,j},
u_{-s_{\ell+1},i}
\right).
\]

In summary, we have computed
\begin{align*}
H_\vvr^*(u) &= M_\vvr(u) \exp{\left( -\frac{x_{\ell+1}-x_\ell}{4} u_{-s_{\ell+1},i}\right)} \\
& \quad\times \frac{\Gamma(c_\ell+k,-2u+s_\ell-j-i)}{\Gamma(c_\ell,v+u+j,v-u+s_\ell-j)}\frac{\Gamma(2u-s_{\ell+1}+j+i)}{\Gamma\pp{\pm(2u-s_{\ell+1}+2i)}}\frac{(-1)^k}{k!}
\\
& \quad\times  \frac1{8\pi\i}\int_{\i\R}\ddfrac{\Gamma\pp{u-\frac{s_{\ell+2,i}}2
\pm\frac{z_{\ell+2}}2,\frac{c_{\ell+1}}2-\pp{u-\frac{s_{\ell+1,i}}2}\pm\frac{z_{\ell+2}}2}}{\Gamma(c_{\ell+1},\pm z_{\ell+2})} I_{\ell+2}(z_{\ell+2})\,\d z_{\ell+2}.
\end{align*}
Therefore, for $\ell\in\llbracket d-2\rrbracket $, the identity $H_\vvr^*(u)=H_{\vvr\topp i}(u)$ reduces to
\begin{multline}\label{eq:ell=d-1}
\mathfrak p_{-s_\ell,-s_{\ell+1}}^{u,v,d,d}(u_{-s_\ell,j},u_{-s_{\ell+1},i})\\
 = \ddfrac{\Gamma(c_\ell+k,-2u+s_\ell-j-i,u+v+i,v-u+s_{\ell+1}-i,2u-s_{\ell+1}+j+i)}{\Gamma(c_\ell,v+u+j,v-u+s_\ell-j,\pm(2u-s_{\ell+1,j}+2k))}\frac{(-1)^k}{k!}.
\end{multline}
This follows from a straightforward calculation. Indeed, the right-hand side is the same as
\begin{multline*}
 \frac{\Gamma(c_\ell+k,u+v+i,v-u+s_{\ell+1}-i)}{\Gamma(c_\ell,u+v+j,v-u+s_\ell-j)}\frac{(-1)^k}{k!}
\times \frac{\Gamma(-2u+s_\ell-j-i,2u-s_{\ell+1}+j+i)}{\Gamma(\pm(2u-s_{\ell+1}+2i))}\\
 = \frac{[c_\ell,v+u+j]_k\Gamma(v-u+s_{\ell+1}-i)}{\Gamma(v-u+s_\ell-j)}\frac{(-1)^k}{k!}\\\times \frac{\Gamma(-2u+s_\ell-j-i)}{\Gamma(-2u+s_{\ell+1}-2j)}[2u-s_{\ell+1}+2j]_k\frac{2u-s_{\ell+1}+2i}{2u-s_{\ell+1}+2j},
\end{multline*}
which is the same as the left-hand side above.  We have thus proved $H_\vvr^*(u) = H_{\vvr\topp i}(u)$ for $\ell\in\llbracket d-2\rrbracket $.

It remains to deal with the case that $\ell=d-1$. In this case, we can use \eqref{eq:I_d} directly. Since $s_{d,i}=s_d-2i$, it follows that $I_d\pp{2u-s_{d,i}}=e^{(x_d-x_{d-1})(u-s_d/2+i)^2}\Gamma\pp{v+u+i,v-u+s_d-i}$. Substituting this expression into the formula for $H_\vvr^*(u)$ and noticing that \eqref{eq:ell=d-1} also holds, we have shown that  $H_\vvr^*(u)=M_{\vvr\topp i}(u)=H_{\vvr\topp i}(u)$.

It remains to extend the continuation to the full strips which appear in \eqref{eq:H_vvr ell}. This part follows from applying Lemma \ref{lem:H-factorization}. This proves part~\textup{(a)}. By Lemma \ref{lem:H-factorization} we can prove that $H_\vvr(u)$ is holomorphic on
\[
\begin{cases}
\C_{(s_{\ell'+1,i+1}/2,s_{\ell'+1,i}/2)}, & \mbox{ if } \ell'<\ell, \vvr\in\mathbb S_{\ell'}\topp i, i\in\Z_{\ge0},\\
\C_{(s_{\ell'+1,i}/2,s_{\ell'+1,i-1}/2)}, & \mbox{ if }\ell<\ell'<d,\vvr\in\mathbb S_{\ell'}\topp{i-1},i\in\Z_{\ge1}\\
\C_{(-\infty, s_{d,r_d}/2)}, & \mbox{ if } \ell'=d, \vvr\in\mathbb S_d\topp{i}, i\in\Z_{\ge0}.
\end{cases}
\]
In the first two cases, the stated strips contain
$\C_{(s_{\ell+2,i}/2,s_{\ell,i}/2)}$.
The same holds in the third case for the paths appearing in
part~\textup{(b)}: then $\vvr\in\mathbb S_d\topp{i-1}$, so
$r_d\le i-1$ and $s_{d,r_d}>s_{\ell,i}$ by $s_\ell-s_d<1$.
This proves part~\textup{(b)}.
For part~\textup{(c)}, $\vvr\in\mathbb S_d\topp i$ gives
$r_d\le i$ and hence $s_{d,r_d}\ge s_{d,i}$.
Thus the strip in the third case also contains
$\C_{(-\infty,s_{d,i}/2)}$.
\end{proof}

\section{Convergence of the stationary measure of the open WASEP}\label{sec:convergence}

In this section, we explain how the developed formula for finite-dimensional Laplace transform of the process $\zeta\topp{u,v}$ (and hence of $\eta\topp{u,v}_t = \B_t/\sqrt 2+\zeta\topp{u,v}_t$) can be used to provide a proof of the convergence of finite-dimensional distributions of the open WASEP. This program was first carried out by \citet{corwin24stationary} and \citet{bryc23markov} for the fan region. To extend the approach to the shock region, one needs to first extend the Askey--Wilson representation of the stationary measure of the open ASEP to the shock region, and this step was due to \citet{wang24askey}. Then one needs to be able to recognize the limiting finite-dimensional Laplace transform as the one for $\eta\topp{u,v}$, and this step was the major contribution of this paper presented in Section \ref{sec:limiting measures}. The proof here consists of first writing down the Askey--Wilson representation of the finite-dimensional Laplace transform of the stationary measure of the open WASEP. Remarkably, here the representation has again a path decomposition. The proof is then completed by showing that (except for a few integrals that vanish in the limit) each integral in the decomposition converges to an integral in the decomposition for the finite-dimensional Laplace transform of $\eta\topp{u,v}$. The proof does not involve any arguments related to the stationary measure of the open KPZ equation. At the same time, the approach does not prove the tightness of the convergence. 

For each $n\in\N$, the open ASEP is an interacting particle system with particles jumping over a lattice indexed by $1,\dots,n$. This is a fundamental model in nonequilibrium systems \citep{derrida07nonequilibrium}.  Each particle jumps to the right adjacent position at rate 1 and to the left adjacent one at rate $q\in[0,1)$, if the position is empty. Particles also enter the system from the left boundary at the position 1 when empty with rate $\alpha>0$, and from the right boundary at the position $n$ when empty with rate $\delta\ge 0$. A particle at position 1 exits the system from the left with rate $\gamma\ge0$, and a particle at position $n$ exits the system from the right with rate $\beta>0$.  See an illustration in Figure \ref{Fig1}.
The open ASEP is an irreducible continuous-time Markov process taking values in $\{0,1\}^n$ and therefore it has a unique stationary distribution, which we refer to as the stationary measure of the open ASEP. 
\begin{figure}[!ht]
  \begin{tikzpicture}[scale=.9]
\draw [fill=black, ultra thick] (.5,1) circle [radius=0.2];
  \draw [ultra thick] (1.5,1) circle [radius=0.2];
\draw [fill=black, ultra thick] (2.5,1) circle [radius=0.2];
  \draw [ultra thick] (5,1) circle [radius=0.2];
   \draw [fill=black, ultra thick] (6,1) circle [radius=0.2];

    \draw [ultra thick] (7,1) circle [radius=0.2];
      \draw [fill=black, ultra thick] (9.5,1) circle [radius=0.2];
   \draw [ultra thick] (10.5,1) circle [radius=0.2];
     \draw[->] (-1,2.3) to [out=-20,in=135] (.5,1.5);
   \node [above right] at (-.2,2) {$\alpha$};
     \draw[->] (10.5,1.5) to [out=45,in=200] (12,2.3);
     \node [above left] at (11.2,2) {$\beta$};
            \node  at (8.25,1) {$\cdots$};  \node  at (3.75,1) {$\cdots$};
      \node [above] at (6.5,1.8) {$1$};
      \draw[->,thick] (6.1,1.5) to [out=45,in=135] (7,1.5);
        \node [above] at (5.5,1.8) {$q$};
            \draw[<-] (5,1.5) to [out=45,in=135] (5.9,1.5);
                 \node [above] at (10,1.8) {$1$};
                \draw[->,thick] (9.6,1.5) to [out=45,in=135] (10.4,1.5);
               \node [above] at (9,1.8) {$q$};
                 \draw[<-] (8.4,1.5) to [out=45,in=135] (9.4,1.5);
       \draw[<-] (-1,-.3) to [out=0,in=-135] (.5,0.6);
   \node [below right] at (-.2,0) {$\gamma$};
    \node [above] at (0.5,0) {$1$};
    \node [above] at (1.5,0) {$2$};
   \node [above] at (2.5,0) {$3$};
     \node [above] at (3.75,0) {$\cdots$};
   \node [above] at (5,0) {$k-1$};
     \node [above] at (6,0) {$k\vphantom{-1}$};
       \node [above] at (7,0) {$k+1$};
        \node [above] at (8.25,0) {$\cdots$};
         \node [above] at (10.5,0) {$n\vphantom{-1}$};
          \node [above] at (9.5,0) {$n-1$};
        \draw[<-] (10.5,.6) to [out=-45,in=180] (12,-.3);
   \node [below left] at (11.2,0) {$\delta$};
\end{tikzpicture}
\caption{\label{Fig1}Transition rates of the asymmetric simple exclusion process with open boundaries, with parameters $\alpha,\beta,\gamma,\delta, q$. Black disks represent occupied sites.}
\end{figure}
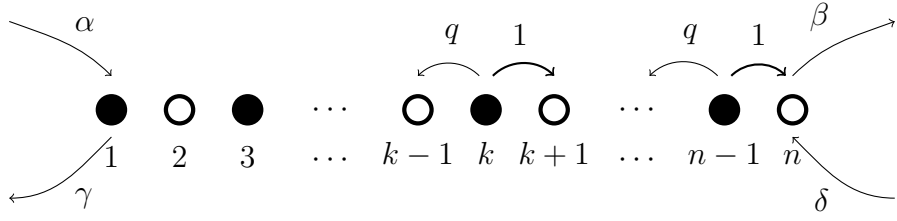

We reparameterize the open ASEP via
\begin{align*}
     A := \kappa^{+}(q,\beta,\delta), & & B := \kappa^{-}(q,\beta,\delta), & & C := \kappa^{+}(q,\alpha,\gamma), & & D := \kappa^{-}(q,\alpha,\gamma),
 \end{align*}
 with
 \begin{align*}
     \kappa^{\pm}(q,x,y) = \frac{1}{2x}\left(1-q-x+y\pm \sqrt{(1-q-x+y)^{2}+4xy}\right).
 \end{align*}
 This reparameterization is often used in the literature, and is intrinsically related to the Askey--Wilson representation we shall apply below. 
The assumptions on the boundary rates imply
$A,C\geq0$ and $-1<B,D\leq0$. 
Then, the so-called {\em fan region} of open ASEP corresponds to the parameters satisfying $AC<1$, and the {\em shock region} corresponds to the parameters satisfying $AC>1$. 

Throughout, we consider a sequence of stationary measures of the open ASEP of size $n\in\N$ each, with parameters chosen as follows:
\equh\label{eq:parameters}
q_n = e^{-2/\sqrt n}, \quad            A_n := q_n^{v}, \quad B_n := -q_n, \quad C_n :=q_n^{u},\quad D_n := -q_n.
\eque
We are interested in the asymptotic behavior of the sequence of stationary measures indexed by $n$, and the sequence of the models with the choice of parameters \eqref{eq:parameters} is often referred to as the {\em weakly asymmetric simple exclusion process with open boundaries} (open WASEP). 
Let $(\tau_{n,1},\dots,\tau_{n,n})$ denote a random vector with the stationary measure of the open ASEP with parameters above.
More specifically, we are interested in the process
\[
 h\topp{u,v}_n(x):=\frac1{\sqrt n}\summ j1{\floor{nx}}\pp{2\tau_{n,j}-1}, x\in[0,1].
\]
Since $A_nC_n=q_n^{u+v}$ and $q_n\in(0,1)$, the fan and shock regions correspond to $u+v>0$ and $u+v<0$, respectively. At the boundary case $u+v=0$, the stationary measure of the open ASEP is a Bernoulli product measure \citep[Remark 2.4]{bryc19limit}, \citep[(2.14), (2.15)]{bryc17asymmetric}, and from this fact the scaling limit of the normalized height function is straightforward \citep[Theorem~1.2(3)]{corwin24stationary}. 

The main result on the open WASEP is the following.
\begin{theorem}
\label{thm:main}
Consider $u,v\in\R$. Assume $d\in\N$, $0<x_1<\cdots<x_d=1$ and $c_j>0$ with $j=1,\dots,d$ and $\summ j1{d-1}c_j<1$. Assume $-c_d<2v$. Then, 
\equh\label{eq:fdd conv}
\limn \esp\exp\pp{-\summ j1dc_jh_n\topp{u,v}(x_j)} = \esp\exp\pp{-\summ j1d c_j \pp{\frac1{\sqrt 2}\B_{x_j}+\zeta\topp{u,v}(x_j)}}.
\eque
 where on the right-hand side, $\B$ is a standard Brownian motion, $\zeta\topp{u,v}$ is defined as in \eqref{eq:zeta-RN}, and the two processes are independent. 
In particular, 
\[
\pp{h_n\topp{u,v}(x)}_{x\in[0,1]}\fddto \pp{\frac1{\sqrt 2}\B_x+\zeta\topp{u,v}(x)}_{x\in[0,1]}.
\]
\end{theorem}
For the result above, the fan region $u+v>0$ and the boundary case $u+v=0$ have been established by \citet{corwin24stationary}. Our main contribution is the computation of the finite-dimensional Laplace transform for the shock region.

\begin{remark}
    For $q\in(0,1)$, the constraint $B=D=-q$ is equivalent to Liggett's condition \citep{liggett75ergodic}, which imposes $\alpha+\gamma/q = 1$ and $\beta + \delta/q = 1$. We expect that the convergence of the finite-dimensional Laplace transforms continues to hold under weaker assumptions; see related works by \citet{himwich25convergence} and \citet{yang25KPZ} that do not require Liggett's condition. We do not pursue this generalization. Moreover, convergence of the Laplace transforms does not directly imply the tightness. It is worth noticing that the method by \citet{schmid26stationary} proves the convergence in the $D$ space directly. 
\end{remark}

The rest of this section is devoted to the proof of Theorem \ref{thm:main}. For this purpose, we first recall the connection between Askey--Wilson signed measures and the open ASEP in Section \ref{s:prelim}. In particular, a new path decomposition of the left-hand side in \eqref{eq:fdd conv} is established in Proposition \ref{prop:openWASEP path} in Section \ref{sec:path prelimit}. The asymptotics of the involved Askey--Wilson signed measures are established and related to those of the auxiliary measures in Section \ref{sec:limiting measures} in Section \ref{sec:AWasym}, and limiting finite-dimensional Laplace transform is computed in Sections \ref{sec:denomconv} and \ref{sec:numconv}. In fact, we compute directly the limiting finite-dimensional Laplace transform for a restrictive set of $u,v$, but then apply a standard sandwich argument to all $u,v\in\R$ (Proposition \ref{prop:sandwich}).

\subsection{Askey--Wilson signed measures and the open ASEP}\label{s:prelim}
We review the Askey--Wilson representation introduced by \citet{bryc17asymmetric}. The original representation was for the fan region, and the extension to the shock region was due to \citet{wang24askey}. 

The Askey--Wilson representation of the stationary measure of the open ASEP originates from the matrix ansatz of \citet{derrida93exact}. The matrix ansatz provides an explicit formula of the probability generating function of the stationary measure involving infinite-dimensional matrices and vectors. 
By choosing a specific family of matrices and vectors, \citet{sasamoto99one} related the stationary measure of the open ASEP with $\gamma=\delta=0$ to the Al-Salam--Chihara orthogonal polynomials, a specialization of the Askey--Wilson orthogonal polynomials. Later on, \citet{uchiyama04asymmetric} extended the connection by allowing $\gamma>0$ and $\delta>0$ to the entire Askey--Wilson orthogonal polynomials. More recently, \citet{bryc17asymmetric} provided a representation for the stationary measure in the fan region using the Askey--Wilson process introduced by the same authors \citep{bryc10askey}. This time, the probability generating function is expressed as the ratio of two expectations of certain functionals of the Askey--Wilson process, and this representation makes it possible to compute the second-order fluctuations through finite-dimensional Laplace transforms.  The extension of the representation to the shock region and hence the entire phase diagram was due most recently to  \citet{wang24askey}.

Throughout this section, we will rely on the following assumptions about the parameters of the open ASEP model.
\begin{assumption}\label{assump:Omega}
For some $q\in(0,1)$, suppose that $a,b,c,d\in\C$ satisfy the following conditions:
    \begin{enumerate}[(i)]
        \item $abcd \notin \{q^{-\ell}:\ell\in\Z_{\ge0}\}$,
        \item For any two distinct $\mathfrak{e},\mathfrak{f}\in\{a,b,c,d\}$ such that $|\mathfrak{e}|,|\mathfrak{f}|\geq 1$, we have $\mathfrak{e}/\mathfrak{f}\notin\{q^{\ell}:\ell \in\Z\}$,
        \item $a,b$ are real, and $c,d$ are either real or form a complex conjugate pair, with $ab<1$ and $cd<1$. 
    \end{enumerate}
    In this case, we write $(a,b,c,d)\in\Omega_q$.
\end{assumption}
If $(a,b,c,d)\in\Omega_q$, it is known that the Askey--Wilson signed measure with parameters $(a,b,c,d,q)$ exists \citep[Theorem 2.7]{wang24askey} and we let $\nu(\d y;a,b,c,d,q)$ denote this measure. 
This measure has an absolutely continuous part supported on $[-1,1]$ and possibly finitely many atoms outside $[-1,1]$. More precisely, we have
    \begin{align*}
        \nu(\d y;a,b,c,d,q) := f(y;a,b,c,d,q)\d y + \sum_{z\in F(a,b,c,d,q)}p(z;a,b,c,d,q)\delta_{z}(\d y),
    \end{align*}
    where 
    \[
        f(y;a,b,c,d,q) := \frac{(q,ab,ac,ad,bc,bd,cd;q)_{\infty}}{2\pi (abcd;q)_{\infty}\sqrt{1-y^{2}}}\bigg|\frac{\left(e^{2\i\theta_{y}};q\right)_{\infty}}{\left(ae^{\i\theta_{y}},be^{-\i\theta_{y}},ce^{\i\theta_{y}},de^{\i\theta_{y}};q\right)_{\infty}} \bigg|^{2}\mathbf{1}_{\{|y|<1\}},
\]
 is the density over $[-1,1]$, where $\theta_y:=\arccos(y)\in[0,\pi]$. Here and below, we  use the $q$-Pochhammer symbols $(a;q)_k:=\prod_{j=0}^{k-1}(1-aq^j)$ for $k\in\mathbb Z_{\geq0}$ and  $(a;q)_\infty:=\prod_{j=0}^{\infty}(1-aq^j)$. For $k\in\mathbb Z_{\geq0}\cup\{\infty\}$, we write $(a_1,\ldots,a_r;q)_k:=\prod_{\ell=1}^r(a_\ell;q)_k$. The collection of candidate atomic locations is
 \[
 F(a,b,c,d,q):=\bigcup_{\mathfrak e\in\{a,b,c,d\}}\ccbb{\rho(\mathfrak eq^j)}_{j\in\Z_{\ge0}, |\mathfrak e|q^j>1} \qmwith \rho(x) := \frac12\pp{x+\frac1x}.
 \]
 Note that the atoms are classified by $\mathfrak e\in\{a,b,c,d\}$, and the convention is, for each $\mathfrak e\in\{a,b,c,d\}$, to refer to $\{\rho(\mathfrak eq^j)\}_{j\in\Z_{\ge 0}, |\mathfrak e|q^j>1}$ as the locations of candidate atoms associated to the parameter $\mathfrak e$. The formulas concerning us the most are those for the atoms associated to the parameter $c$. 
For $abcd\neq0$ and $|c|>1$, the mass function $p$ is
\begin{align}
    p\left(\rho(c);a,b,c,d,q\right) & = \frac{\left(c^{-2},ab,bd,ad;q\right)_{\infty}}{\left(b/c,a/c,d/c,abcd;q\right)_{\infty}}, \label{eq:AW atoms 0}\\
    p\left(\rho(cq^j);a,b,c,d,q\right) & = p(\rho(c);a,b,c,d,q)\frac{\left(c^{2},ac,bc,cd;q\right)_{j}\left(1-c^{2}q^{2j}\right)}{\left(q,qc/a,qc/b,qc/d;q\right)_{j}(1-c^{2})}\left(\frac{q}{abcd}\right)^{j},\label{eq:AW atoms j}
\end{align} for $j\in\Z_{\ge 1}, |c|q^j>1$. 
Some of the masses in \eqref{eq:AW atoms 0} and
\eqref{eq:AW atoms j} may vanish (and hence the notion of `candidate'). Therefore,
\[
\operatorname{supp}\nu(\cdot;a,b,c,d,q) \subseteq [-1,1]\cup F(a,b,c,d,q).
\] 
The mass formulas for the candidate atomic locations associated with $\mathfrak e\in\{a,b,d\}$ are obtained by swapping $c$ and $\mathfrak e$ in the formulas above. 

Next, we explain how the stationary measure of the open ASEP is related to Askey--Wilson signed measures.  For all $t>0$ such that 
        \equh\label{eq:Omega t}
        \left(A\sqrt{t},B\sqrt{t},Ct^{-1/2},Dt^{-1/2}\right)\in\Omega_{q},
        \eque
        we set \[
        \pi_{t}(\d y)\equiv \pi_{t}^{(A,B,C,D,q)}(\d y) :=\nu\left(\d y; A\sqrt{t},B\sqrt{t},Ct^{-1/2},Dt^{-1/2},q\right).
        \]
        We note that $\pi_t$ is supported on
        \[
        \mathbb{U}_{t}\equiv \mathbb{U}_{t}(A,B,C,D,q) := [-1,1]\cup F\left(A\sqrt{t},B\sqrt{t}, Ct^{-1/2},Dt^{-1/2},q\right).
        \]
 For all $s<t$ and $x\in\mathbb{U}_{s}$ such that 
\equh\label{eq:Omega st}
        \left(A\sqrt{t}, B\sqrt{t}, \sqrt{\frac{s}{t}}\left(x+\sqrt{x^{2}-1}\right), \sqrt{\frac{s}{t}}\left(x-\sqrt{x^{2}-1}\right)\right)\in \Omega_{q},
\eque
    we set 
\begin{align*}
    P_{s,t}(x,\d y) & \equiv P_{s,t}^{(A,B,C,D,q)}(x,\d y) 
    \\ & := \nu\left(\d y; A\sqrt{t}, B\sqrt{t}, \sqrt{\frac{s}{t}}\left(x+\sqrt{x^{2}-1}\right), \sqrt{\frac{s}{t}}\left(x-\sqrt{x^{2}-1}\right),q\right),
\end{align*} and $P_{s,t}(x,\d y):=0$ when $x\notin \mathbb{U}_{s}$.  When $|x|<1$ we understand $x\pm\sqrt{x^{2}-1} = e^{\pm i\theta_{x}}$ for $\theta_x$ satisfying $\cos{(\theta_{x})}=x$. For $s=t, x\in\mathbb U_s$ we simply set $P_{s,s}(x,\d y) = \delta_x(\d y)$.  For $0<t_1\le\cdots\le t_d$ such that all the measures below are well-defined, we introduce
\begin{align*}
    \pi_{t_1,\dots,t_{d}}\topp{A,B,C,D,q}\left(\d y_{1},...,\d y_{d}\right) & := \pi_{t_{1}}\topp{A,B,C,D,q}(\d y_{1})P\topp{A,B,C,D,q}_{t_{1},t_{2}}(y_{1},\d y_{2})\cdots P\topp{A,B,C,D,q}_{t_{d-1},t_{d}}(y_{d-1},\d y_{d}).
\end{align*} 
With this notation, we have the following Askey--Wilson representation of the probability generating function of the stationary measure of the open ASEP. 
\begin{theorem}\label{thm:AW rep}
Consider an open ASEP with parameters $A,B,C,D,q$ and let $(\tau_{n,1},\dots,\tau_{n,n})$ be a random vector following its stationary measure. 
Suppose that $(A,B,C,D)\in\Omega_q$. Then, for $0<t_1\le\cdots\le t_n$ such that 
\equh\label{eq:t}
|A\sqrt {t_n}|<1, |B\sqrt{t_n}|<1,
\eque
 we have
\equh\label{eq:generating function}
        \mathbb{E}\pp{\prod_{i=1}^nt_{i}^{\tau_{n,i}}}  = \frac{\int_{\mathbb{R}^{n}}\prod_{i=1}^{n}(1+t_{i}+2\sqrt{t_{i}}y_{i})\pi\topp{A,B,C,D,q}_{t_{1},...,t_{n}}(\d y_{1},...,\d y_{n})}{2^n\int_\R (1+y)^n\pi_1\topp{A,B,C,D,q}(\d y)}.
\eque
\end{theorem}
\begin{proof}
The representation was established in \citet[Theorem~1.1]{wang24askey}
for times in a suitable interval near $1$. Subsequently,
\citet[Lemma~2.2 and Theorem~2.3]{wang25asymmetric} applied the same
argument on a more general admissible set of times. Here, the constraint is further relaxed, although the proof remains the same and hence we only provide a sketch. 
We verify admissibility under \eqref{eq:t}. For every $0<r\le t_n$, the parameters $A\sqrt r$ and $B\sqrt r$ have modulus less than $1$, while
$C/\sqrt r\ge0$ and $D/\sqrt r\le0$. Hence \eqref{eq:Omega t} holds. For $0<s<t\le t_n$, the last two parameters in \eqref{eq:Omega st} have product $s/t<1$; since the first two have modulus less than $1$, \eqref{eq:Omega st} holds for every $x\in\mathbb U_s$. The orthogonality and projection argument of \citet[Section~3]{wang24askey} identifies the numerator in \eqref{eq:generating function} with the unnormalized stationary generating function. The denominator can be dealt with similarly (by taking $t_1=\cdots=t_n = 1$). Their ratio gives \eqref{eq:generating function}.
\end{proof}

\subsection{A path decomposition for the integral representation of open WASEP}\label{sec:path prelimit}
Recall that we are interested in the sequence of stationary measures with parameters satisfying \eqref{eq:parameters}. We shall write explicitly the integral representation in Theorem \ref{thm:AW rep}. In particular, the numerator this time has again a path decomposition that facilitates our analysis later.  We write
\[
\pi_t^{(n)} :=\pi_t^{(A_n,B_n,C_n,D_n,q_n)} \qmand  P_{s,t}^{(n)} :=P_{s,t}^{(A_n,B_n,C_n,D_n,q_n)},
\]
and use the same superscript convention for the multipoint measures $\pi_{t_1,\ldots,t_d}^{(n)}$ whenever they are well-defined.

We begin by re-writing the finite-dimensional Laplace transform. Fix $d\in\Z_{>0}$, $0=x_0<x_1<\cdots<x_d=1$, and $c_1,\dots,c_d>0$. For $k\in\llbracket d\rrbracket$ set $s_k:=\sum_{j=k}^d c_j$ and
for $k\in\{0,\dots,d\}$ set $n_k:=\lfloor nx_k\rfloor$. We consider 
\[
\varphi_{n,\vv x}\pp{\frac{\vv c}{\sqrt n}} = \esp\exp\pp{-\summ k1d c_kh_n\topp{u,v}(x_k)} 
= \esp\exp\pp{-\frac{1}{\sqrt{n}}\summ k1d c_k\summ j1{\floor{nx_k}}(2\tau_{n,j}-1)}.
\]
Under suitable assumptions on $u,v,\vv c$ to be specified later,  Theorem \ref{thm:AW rep} gives
\equh\label{eq:varphi}
\varphi_{n,\vv x}\pp{\frac{\vv c}{\sqrt n}} = \frac{\Phi\topp n(\vv x, \vv c)}{Z_n}
\eque
with
\equh\label{eq:0}
\Phi\topp n(\vv x,\vv c) := 2^{-n}\int_{\R^d} \prodd k1d\pp{\cosh\pp{\frac{s_k}{\sqrt n}}+y_k}^{n_k-n_{k-1}}\pi\topp n_{e^{-2s_1/\sqrt n},\dots,e^{-2s_d/\sqrt n}}(\d\vvy),
\eque
where $\d\vvy = \d y_1\dots\d y_d$, and 
\equh\label{eq:1}
Z_n := 2^{-n}\int_\R(1+y)^n \pi_1\topp n(\d y). 
\eque
We do not write the dependence of $q$ on $n$ for notational convenience from now on. We first check that $(A_n,B_n,C_n,D_n)\in \Omega_{q}$ under the assumption
\[u+v+2\notin\Z_{\le 0}\qmand u-v\notin\Z \mbox{ if } u,v\le 0.
\]
 The condition (i) of Assumption \ref{assump:Omega} is equivalent to $u+v+2\notin\Z_{\le0}. $ Since $B_n=D_n=-q$ have modulus less than one, only $A_n=q^v$ or $C_n=q^u$ concern the condition (ii). When $u,v\leq0$, this condition requires $A_n/C_n=q^{v-u}$ not to be an integer power of $q$, or equivalently  $u-v\notin\Z$. Condition (iii) of Assumption \ref{assump:Omega} holds automatically.  In particular, $\pi\topp n_1$ is well-defined.

Next, for the numerator, we deal with 
\[
\pi\topp n_{e^{-2s_1/\sqrt n},\dots,e^{-2s_d/\sqrt n}}(\d\vvy) = \pi\topp{A_n,B_n,C_n,D_n,q}_{e^{-2s_1/\sqrt n},\dots,e^{-2s_d/\sqrt n}}(\d\vvy).
\]
We shall not only exclude the values $u,v,s_j$ for which the relevant measures are not well-defined, but further restrict the $s_j$'s  so that all candidate atomic locations near 1, if any, are associated with the parameter $C_n$ but {\em not} $A_n$, and this shall simplify significantly our calculations in the sequel. For this purpose we impose
\[
-c_d<\min\{0,2v\}.
\]
 For every $j$, we then have
\begin{align*}
A_nq^{s_j/2} =q^{v+s_j/2}<1, & &  |B_nq^{s_j/2}|=q^{1+s_j/2}<1, & &  C_nq^{-s_j/2}=q^{u-s_j/2}.
\end{align*}
Note that in the shock region, we necessarily have $C_nq^{-s_j/2}>1$. 
 For any $j$ such that $s_j>2$, there are also candidate atomic locations associated with $D_nq^{-s_j/2}=-q^{1-s_j/2}$, and their locations approach $-1$ as $n\to\infty$. 

It turns out that we have a path decomposition of the integral of the numerator on the right-hand side of \eqref{eq:generating function}. To simplify the notation, write
\[
 \wt s_{n,i} := e^{-2s_i/\sqrt n}, i=1,\dots,d,
\]
and we examine support of $P\topp n_{\wt s_{n,i},\wt s_{n,i+1}}(x,\cdot)$ for $i=1,\dots,d-1$. For each $i=1,\dots, d$, consider the candidate positive and negative atoms 
\begin{align}
\mathsf u\topp n_{\wt s_{n,i},j} &:=\rho\pp{q^{u+j}\wt s_{n,i}^{-1/2}} = \rho\pp{q^{u+j-s_i/2}}, j=0,\dots,j_{u,-s_i}:=\ceil{-u+\frac {s_i}2-1},
\nonumber\\
\mathsf n\topp n_{\wt s_{n,i},j} & :=\rho\pp{-q^{j+1}\wt s_{n,i}^{-1/2}} = \rho\pp{-q^{j+1-s_i/2}}, j=0,\dots,\ell_{-s_i}:=\ceil{\frac{s_i}2-2}.
\label{eq:mathfrak d}
\end{align}
The positive atoms $(\mathsf u\topp n_{\wt s_{n,i},j})_{j=0,\dots,j_{u,-s_i}}$ are in correspondence to  the atoms $(u_{-s_i,j})_{j=0,\dots,j_{u,-s_i}}$ in the path decomposition of the finite-dimensional Laplace transform studied in Section \ref{sec:limiting measures}. In particular, the definition of $j_{u,t}$ here is the same as the one introduced earlier in \eqref{eq:j_u,t}. It is worth noting that  $\mathsf u\topp n_{\wt s_{n,i},j}$ does not converge to $u_{-s_i,j}$ in \eqref{eq:j_u,t}, but its transformation $u\topp n_{-s_i,j}$ introduced later does. 
Set
\begin{align*}
F_{\wt s_{n,i}}^{(n),+}&=\ccbb{\mathsf u\topp n_{\wt s_{n,i},j}}_{j\in\Z_{\ge 0}, q^{u+j}\wt s_{n,i}^{-1/2}>1} = \ccbb{\mathsf u\topp n_{\wt s_{n,i},j}}_{j=0,\dots,j_{u,-s_i}},\\
F_{\wt s_{n,i}}^{(n),-}&=\ccbb{\mathsf n\topp n_{\wt s_{n,i},j}}_{j\in\Z_{\ge 0}, q^{1+j}\wt s_{n,i}^{-1/2}>1} = 
\ccbb{\mathsf n\topp n_{\wt s_{n,i},j}}_{j=0,\dots,\ell_{-s_i}}.
\end{align*}
Thus, 
  $\supp(\pi\topp n_{\wt s_{n,i}})\subset\mathbb U^{(n)}_{\wt s_{n,i}} :=[-1,1]\cup F^{(n),+}_{\wt s_{n,i}}\cup F^{(n),-}_{\wt s_{n,i}}$.
  (Recall that the weight of a candidate atom here might be zero and hence the inclusion here might be strict.)
   For $x\in \mathbb U\topp n_{\wt s_{n,i}}$, $i=1,\dots,d-1$, one readily verifies that
\begin{align*}
(\mathsf a,\mathsf b,\mathsf c,\mathsf d) & \equiv \pp{A_n q^{s_{i+1}/2}, B_n q^{s_{i+1}/2}, q^{c_i/2}\pp{x+\sqrt{x^2-1}},q^{c_i/2}\pp{x-\sqrt{x^2-1}}}
\\ & = \pp{q^{v+s_{i+1}/2}, -q^{1+s_{i+1}/2}, q^{c_i/2}\pp{x+\sqrt{x^2-1}},q^{c_i/2}\pp{x-\sqrt{x^2-1}}}\in\Omega_q.
\end{align*}
By a careful examination of the formulas of Askey--Wilson signed measures we have the following observations. 
\begin{enumerate}[(i)]
\item Assume $x\in[-1,1]$. Then $|\mathsf a|=q^{v+s_{i+1}/2}<1$, $|\mathsf b|=q^{1+s_{i+1}/2}<1$, and $|\mathsf c|=|\mathsf d|=q^{c_i/2}<1$. In this case, none of the four parameters is associated to any atoms. That  is, 
\[
\mathrm{supp}\pp{P^{(n)}_{\wt s_{n,i},\wt s_{n,i+1}}(x,\cdot)} \subseteq[-1,1], \mfa x\in[-1,1].
\]

That is, once the path is in the interval $[-1,1]$, it remains in that interval for the rest of the time (until time $\wt s_{n,d}$). 

\item  Assume that $x = \mathsf u\topp n_{\wt s_{n,i},j}$. 
The third and fourth parameters of the transition measure are $q^{j+u-s_{i+1}/2}$ and $q^{s_i-s_{i+1}/2-j-u}$. It follows that 
the only parameter that may have modulus greater than or equal to one is the third parameter $q^{j+u-s_{i+1}/2}$. That is, 
\[
\mathrm{supp}\pp{P^{(n)}_{\wt s_{n,i},\wt s_{n,i+1}}\pp{\mathsf u\topp n_{\wt s_{n,i},j},\cdot}}  \subset [-1,1]\cup \ccbb{\mathsf u\topp n_{\wt s_{n,i+1},j'}}_{j'=j,\dots,j_{u,-s_{i+1}}}. 
\] 
That is, the $j$-th positive atom at time $e^{-2s_i/\sqrt n}$ jumps at time $e^{-2s_{i+1}/\sqrt n}$ only either to $[-1,1]$ or to the $j'$-th positive atom for certain admissible $j'\ge j$.

\item Assume that $x=\mathsf n\topp n_{\wt s_{n,i},j}$ (so necessarily $s_i>2$ and $j+1-s_i/2<0$). 

It is straightforward to check that the third and fourth parameters appearing in \eqref{eq:Omega st}
are, up to interchange, $-q^{j+1-s_{i+1}/2}$ and $-q^{s_i-s_{i+1}/2-j-1}$. The only possible parameter that has modulus greater than or equal to one is $-q^{j+1-s_{i+1}/2}$. It follows that 
\[
\mathrm{supp}\pp{P^{(n)}_{\wt s_{n,i},\wt s_{n,i+1}}(\mathsf n\topp n_{\wt s_{n,i},j},\cdot)}\subset [-1,1]\cup \ccbb{\mathsf n\topp n_{\wt s_{n,i+1},j'}}_{j'=j,\dots,\ell_{-s_{i+1}}}.
\]

That is, the $j$-th negative atom at time $e^{-2s_i/\sqrt n}$ jumps at time $e^{-2s_{i+1}/\sqrt n}$ only either to $[-1,1]$ or to the $j'$-th negative atom for certain admissible $j'\ge j$.

\end{enumerate}

Introduce
\[
\Omega^{(d,n)}_\emptyset:= (-1,1)^d, 
\]
and for $\ell=1,\dots,d$, 
\begin{align}
\Omega_{(r_1,\dots,r_\ell)}^{(d,n),+} &:= \ccbb{\mathsf u_{\wt s_{n,1},r_1}\topp n}\times\cdots\times \ccbb{\mathsf u_{\wt s_{n,\ell},r_\ell}\topp n}\times(-1,1)^{d-\ell},\quad 0\le r_1\le\cdots \le r_\ell\le j_{u,-s_\ell}.\nonumber\\
\Omega_{(r_1,\dots,r_\ell)}^{(d,n),-} &:= \ccbb{\mathsf n_{\wt s_{n,1},r_1}\topp n}\times\cdots\times \ccbb{\mathsf n_{\wt s_{n,\ell},r_\ell}\topp n}\times(-1,1)^{d-\ell},\quad 0\le r_1\le\cdots \le r_\ell\le \ell_{-s_\ell}.\nonumber
\end{align}

In summary, we have shown the following.

\begin{proposition}\label{prop:good t}
Assume $u,v\in\R$, $u+v+2\notin\Z_{\le0}$, and in addition \(u-v\notin\Z\) if \(u,v\leq0\). Assume further that $-s_1<\cdots<-s_d<\min\{0,2v\}$. Then, for every $n\in\mathbb Z_{>0}$, the representations
\eqref{eq:varphi}, \eqref{eq:0}, and \eqref{eq:1} are valid; that is,  all the Askey--Wilson signed measures involved in the definition of $\pi\topp n_1$ and $\pi\topp n_{e^{-2s_1/\sqrt n},\dots,e^{-2s_d/\sqrt n}}$ are well-defined. 

Furthermore, for every such $n$ and every bounded Borel-measurable function $f:\mathbb R^d\to\mathbb C$, we have
\begin{align*}
\int_{\R^d}f(\vvy)&\pi_{e^{-2s_1/\sqrt n},\dots,e^{-2s_d/\sqrt n}}\topp n(\d\vvy)\\
&  = \int_{\Omega^{(d,n)}_\emptyset}f(\vvy)
\pi_{e^{-2s_1/\sqrt n},\dots,e^{-2s_d/\sqrt n}}\topp n(\d\vvy) \\
&  \quad+ \summ \ell1d\sum_{0\le r_1\le\cdots\le r_\ell\le j_{u,-s_\ell}}\int_{\Omega^{(d,n),+}_{(r_1,\dots,r_\ell)}}f(\vvy)
\pi_{e^{-2s_1/\sqrt n},\dots,e^{-2s_d/\sqrt n}}\topp n(\d\vvy) \\
& \quad +
\summ\ell1d
\sum_{0\le r_1\le\cdots\le r_\ell\le \ell_{-s_\ell}}
\int_{\Omega^{(d,n),-}_{(r_1,\dots,r_\ell)}}f(\vvy)
\pi_{e^{-2s_1/\sqrt n},\dots,e^{-2s_d/\sqrt n}}\topp n(\d\vvy).
\end{align*}
\end{proposition}

This leads to a path decomposition of $\Phi\topp n$ in correspondence to the one of $\Psi_{u,v}(\vvx,\vvc)$ in Proposition \ref{prop:path decomp}. To see this better, and also to facilitate the computation of the limiting Laplace transform later, we apply  a change of variables $y_k\mapsto2n(1-y_k)$ and a time-reparametrization and work  with
\begin{align}
\wt \pi_t\topp n(\d y) & := \pi_{e^{2t/\sqrt n}}\topp n\pp{\d\pp{1-\frac y{2n}}},\label{eq:wt pi}\\
\wt P_{s,t}\topp n(y,\d z) & := P\topp n_{e^{2s/\sqrt n},e^{2t/\sqrt n}}\pp{1-\frac y{2n},\d\pp{1-\frac z{2n}}},\nonumber 
\end{align}
and we will use the notation $\wt\pi_t^{(n),c}$ and $\wt\pi_t^{(n),d}$ for the density and mass functions of $\wt\pi_t\topp n$ (see \eqref{eq:wt pi^c} below), respectively. For the transition kernels, we use the notation $\wt P_{s,t}^{(n),c,c}$, $\wt P_{s,t}^{(n),d,c}$, and $\wt P_{s,t}^{(n),d,d}$ (see \eqref{eq:wt P^c,c}, \eqref{eq:wt P^d,c}, and \eqref{eq:wt P^d,d} below). 
In this way,  we write the denominator of interest as
 \[
Z_n := 2^{-n}\int_\R(1+y)^n \pi_1\topp n(\d y) =\int_\R\pp{1-\frac y{4n}}^n\wt\pi_0\topp n(\d y).
\]
For the numerator, we shall rewrite the above path decomposition in terms of the measures $\wt \pi\topp n$ and $\wt P\topp n$. 
Let $i\geq0$ and $j\in\llbracket0,d\rrbracket$ be such that
\equh\label{eq:i,j}
u\in\C_{(s_{j+1,i}/2,s_{j,i}/2)}.
\eque 
This is the same condition imposed in the path decomposition in 
Section~\ref{sec:path}. 
Apply the change of variables $y_k=2n(1-z_k)$ to \eqref{eq:0} and introduce
\[
W_n(\vvy)=\prod_{k=1}^d\pp{\frac{\cosh(s_k/\sqrt n)+1-y_k/(2n)}{2}}^{n_k-n_{k-1}}.
\]
Then,
\begin{align*}
\Phi\topp n&\equiv \Phi\topp n(\vvx,\vvc)  = 2^{-n}\int_{\R^d} \prodd k1d\pp{\cosh\pp{\frac{s_k}{\sqrt n}}+y_k}^{n_k-n_{k-1}}\pi\topp n_{e^{-2s_1/\sqrt n},\dots,e^{-2s_d/\sqrt n}}(\d y_1,
\dots,\d y_d),\\
&=
\int_{\R^d}W_n(\vvy)\wt\pi\topp n_{-s_1,\dots,-s_d}(\d \vvy). 
\end{align*}
For $\ell\in\llbracket d\rrbracket$ and an admissible
$\vvr=(r_1,\ldots,r_\ell)$, set
\[
\Phi_{\vvr}^{(n)}\equiv \Phi_\vvr\topp n(\vvx,\vvc) :=
\int_{\ccbb{u_{-s_1,r_1}^{(n)}}\times\cdots
\times\ccbb{u_{-s_\ell,r_\ell}^{(n)}}
\times(0,4n)^{d-\ell}}
W_n(\vvy)\,
\wt\pi_{-s_1,\ldots,-s_d}^{(n)}(\d\vvy),
\]
and similarly
\begin{align*}
\Phi_\emptyset \topp n\equiv \Phi_\emptyset \topp n(\vvx,\vvc)&:=\int_{(0,4n)^d}W_n(\vvy) 
\wt\pi_{-s_1,\ldots,-s_d}^{(n)}(\d\vvy),\\
\Phi^{(n),-} \equiv\Phi^{(n),-}(\vvx,\vvc)&:= \int_{\R^d\setminus(-\infty,4n)^d}W_n(\vvy)\wt\pi_{-s_1,\dots,-s_d}\topp n(\d\vvy).
\end{align*}
Recall $\mathbb S\topp i_\ell = \{(r_1,\dots,r_\ell)\in\Z_{\ge0}^\ell: r_1\le\cdots \le r_\ell\le i\}$ in \eqref{eq:S^(i)}.
Then, Proposition \ref{prop:good t} yields the following.

\begin{proposition}\label{prop:openWASEP path}
Assume $u,v\in\R$, $u+v+2\notin\Z_{\le0}$, and in addition $u-v\notin\Z$ if $u,v\leq0$. 
Assume $d\in\N$, $\vvx,\vvc$ as in Proposition \ref{prop:good t}, $\summ j1{d-1}c_j<1$, and assume $i\in\Z_{\ge 0},j\in\{0,\dots,d\}$ satisfy \eqref{eq:i,j}. Then, we have
\begin{equation}\label{eq:prelimit path decomp}
\Phi^{(n)}
=
\Phi^{(n),-}+\Phi_{\emptyset}^{(n)}
+\sum_{\ell=1}^{j}\ \sum_{\vvr\in\mathbb S_\ell\topp i}
 \Phi_{\vvr}^{(n)}
+\sum_{\ell=j+1}^{d}\ \sum_{\vvr\in\mathbb S_\ell\topp{i-1}}
 \Phi_{\vvr}^{(n)}.
\end{equation}
\end{proposition}
We recall that $i,j$ depend on $\vv c$ and $u$ and these $\Phi$ functions also depend on $\vvx,\vvc,u, v$. 
This decomposition is to be compared with the decomposition of $\Psi_{u,v}(\vvx,\vvc)$ in \eqref{eq:path decomp}. 
Again, when $j=0$ the first double summation on the right-hand side is understood as zero, and when $i=0$ the second double summation is understood as zero.

In the remainder of this section, we shall first compute the limiting finite-dimensional Laplace transforms under extra restrictions on the parameters in addition to those in Proposition \ref{prop:openWASEP path}, and then apply a standard sandwich argument to remove these restrictions. This result is the next proposition.

\begin{proposition}\label{prop:sandwich}
Assume $d\in\N$, $0<x_1<\cdots<x_d=1$, $c_1,\dots,c_d>0$, and $\summ j1{d-1}c_j<1$. Assume that for all $v>-c_d/2$, and $u\in\R$ such that $(u,v)\in\calD\cap\R^2$, $u+v+2\notin\Z_{\le0}$, and $u-s_k/2\notin\Z_{\le 0},k=1,\dots,d$, we have
\[
\limn \esp\exp\pp{-\summ j1dc_jh_n\topp{u,v}(x_j)} = \esp\exp\pp{-\summ j1d c_j \pp{\frac1{\sqrt 2}\B_{x_j}+\zeta\topp{u,v}(x_j)}}.
\]
Then the same convergence holds for all $u\in\R, v>-c_d/2$. 
\end{proposition}
\begin{remark}\label{rem:uv}
We have imposed various assumptions on $u,v$ in different results. For the integral representations involving auxiliary measures in Theorem \ref{thm:AC}, $(u,v)\in\calD$ is needed and, for the numerator representation, we also impose
\equh\label{eq:c}
c_j>0, j\in\llbracket d\rrbracket, \quad \summ j1{d-1}c_j<1,\quad  -c_d<2\rmRe(v), \qmand \rmRe(u)-\frac{s_j}2\notin\Z_{\le0}, j\in\llbracket d\rrbracket.
\eque Here, for the pre-limit model to have an Askey--Wilson representation, we impose $u+v+2\notin\Z_{\le0}$ and $u-v\notin\Z$ if $u\le 0,v\le 0$; note that the latter is guaranteed for $(u,v)\in\calD\cap\R^2$, but not the former. 
For the dominated convergence arguments below, we further impose $u,v\notin\Z_{\le0}$  and \eqref{eq:c} so that we can apply the uniform estimates in Lemmas \ref{lem:marginal} and \ref{lem:transitional}. All these  restrictions on $u,v$ except $v>-c_d/2$  are removed by Proposition \ref{prop:sandwich}.
\end{remark}
\begin{proof}[Proof of Proposition \ref{prop:sandwich}]
For fixed $\vvc$, the collection of $(u,v)$ satisfying the parameter restrictions under which the convergence is assumed in Proposition \ref{prop:sandwich} is dense in $\{(u,v):v>-c_d/2\}$. Thus, we choose sequences $(u^\pm_m)_{m\in\N}, (v^\pm_m)_{m\in\N}$ such that $(u^-_m,v^+_m), (u^+_m,v^-_m)$ satisfy the assumption, $u^-_m<u<u^+_m$ and $v^-_m<v<v^+_m$ for each $m\in\N$, and $u^-_m\uparrow u, u^+_m\downarrow u, v^-_m\uparrow v, v^+_m\downarrow v$ as $m\to\infty$.
By stochastic monotonicity of the stationary measure of the open ASEP \citep[Lemma~4.12]{wang25asymmetric}, for each $m$ fixed,  the three height functions can be coupled so that $h_n\topp{u^-_m,v_m^+}(x)\le h_n\topp{u,v}(x)\le h_n\topp{u_m^+,v_m^-}(x)$ for all $x\in[0,1]$ almost surely. 
Since  $c_k>0$, the corresponding finite-dimensional Laplace transforms are in the reversed order. That is, 
\[
\esp\exp\pp{-\summ  j1dc_jh_n\topp{u_m^+,v_m^-}(x_j)}
\le
\esp\exp\pp{-\summ j1d c_jh_n\topp{u,v}(x_j)}
\le 
\esp \exp\pp{-\summ j1d c_jh\topp{u_m^-,v_m^+}_n(x_j)}.
\]  Taking first $n\to\infty$ and then $m\to\infty$ proves the claim by continuity of the limiting Laplace transform in $(u,v)$, which follows from Proposition \ref{prop:analytic RN}. 
\end{proof}

\subsection{Asymptotics of Askey--Wilson signed measures}\label{sec:AWasym}
We now establish asymptotic formulas for the densities and atomic masses of the rescaled Askey--Wilson marginals and transition kernels, relating them to the auxiliary measures introduced in Section \ref{sec:CDH}. We write $f(q)\sim g(q)$ as $q\uparrow1$ if $\lim_{q\uparrow 1}f(q)/g(q) = 1$. Equivalently, when viewing $q = q_n$ as a function of $n$, we write $f_n\sim g_n$ as $n\to\infty$ if $\limn f_n(q)/g_n(q) = 1$.

\begin{lemma}\label{lem:q1} 
For every fixed $z\in\C\setminus\Z_{\le0}$, we have
\begin{align}
    (q^z;q)_{\infty} & \sim \frac{(q;q)_{\infty}(1-q)^{1-z}}{\Gamma(z)},\nonumber\\
\label{e:quneg}
    (-q^z;q)_{\infty} & \sim \frac{(q^{2};q^{2})_{\infty}}{(q;q)_{\infty}}2^{1-z},
\end{align}
as $q\uparrow 1$. The asymptotic equivalence \eqref{e:quneg} also holds for
$z\in\Z_{\le0}$.
\end{lemma}
\begin{proof}
We recall the identity for $z\in\C\setminus\Z_{\le0}$,
\equh\label{eq:q Gamma}
(q^z;q)_\infty = \frac{(q;q)_\infty (1-q)^{1-z}}{\Gamma_q(z)}, 
\eque
 where $\Gamma_q$ is the $q$-Gamma function.
Using $\lim_{q\to 1}\Gamma_q(z)= \Gamma(z)$, we have
\begin{align}
\frac{(q^z;q)_\infty}{(q;q)_\infty}& = {(1-q)^{1-z}}\frac{\Gamma_q(1)}{\Gamma_q(z)} \sim (1-q)^{1-z}\frac{1}{\Gamma(z)},\nonumber \\
\frac{(q^{2z};q^2)_\infty}{(q^{2};q^2)_\infty} &= {(1-q^2)^{1-z}}\frac{\Gamma_{q^2}(1)}{\Gamma_{q^2}(z)} \sim 2^{1-z}(1-q)^{1-z}\frac{1}{\Gamma(z)},\nonumber \\
(-q^z;q)_\infty & = \frac{(q^{2z};q^2)_\infty}{(q^z;q)_\infty}  = \frac{(q^2;q^2)_\infty}{(q;q)_\infty}\frac{(1-q^2)^{1-z}}{(1-q)^{1-z}}\frac{\Gamma_{q}(z)}{\Gamma_{q^2}(z)}\sim \frac{(q^2;q^2)_\infty}{(q;q)_\infty}2^{1-z}.\nonumber
\end{align}
If $z=-m$ with $m\in\mathbb Z_{\ge0}$, then
\[
(-q^{-m};q)_\infty = \frac{(q^2;q^2)_\infty}{(q;q)_\infty} \prod_{r=0}^{m}(1+q^{-r}) \sim \frac{(q^2;q^2)_\infty}{(q;q)_\infty}2^{m+1},
\]
which proves \eqref{e:quneg} for $z\in\mathbb Z_{\le0}$.
\end{proof}

From now on, we write
\[
q=q_n=e^{-2/\sqrt n},
\]
 and set $y_n=1-y/(2n)$ and $z_n=1-z/(2n)$. For $x\in\R$, $y>0$, and $z\geq0$, take $n$ sufficiently large that $0<y<4n$ and $0\leq z<4n$. 

\begin{lemma}\label{lem:q2}
Fix $x\in\R$, $y>0$, and $z\geq0$. Then, as $n\to\infty$,
\begin{align}
\abs{(q^xe^{\i\theta_{y_n}};q)_\infty}^2
&\sim
\frac{(q;q)_\infty^2(-\log q)^{2-2x}}
{\displaystyle\abs{\Gamma\pp{x+\frac12\i\sqrt y}}^2},
\label{e:qutheta}\\
\abs{(-q^xe^{\i\theta_{y_n}};q)_\infty}^2
&\sim
\frac{(q^2;q^2)_\infty^2}{(q;q)_\infty^2}2^{2-2x},
\label{e:quthetaneg}\\
\abs{(e^{2\i\theta_{y_n}};q)_\infty}^2
&\sim
\frac{(q;q)_\infty^2(-\log q)^2}
{\abs{\Gamma\pp{\i\sqrt y}}^2},
\label{e:2theta}\\
\abs{(q^xe^{\i(\theta_{y_n}\pm\theta_{z_n})};q)_\infty}^2
&\sim
\frac{(q;q)_\infty^2(-\log q)^{2-2x}}
{\displaystyle\abs{\Gamma\pp{x+
\frac12\i\pp{\sqrt y\pm\sqrt z}}}^2}.
\label{e:plustheta}
\end{align}
In \eqref{e:plustheta}, the relation is understood to hold when both sides take the plus or the minus sign, and with the minus
sign when $y=z$, it is assumed in addition that
$x\notin\Z_{\leq0}$.
\end{lemma}

\begin{proof}
By the $q$-Gamma identity \eqref{eq:q Gamma} and recalling $q^{\i\sqrt n\theta_{y_n}/2} =e^{-\i\theta_{y_n}}$, we obtain 
\begin{align}
\abs{(q^xe^{\i\theta_{y_n}};q)_\infty}^2
&=\frac{(q;q)_\infty^2(1-q)^{2-2x}} {\displaystyle\abs{\Gamma_q\pp{x+\frac12\i\sqrt n\,\theta_{y_n}}}^2}, \label{eq:qgamma-exact-single-angle}\\
\abs{(e^{2\i\theta_{y_n}};q)_\infty}^2 &=\frac{(q;q)_\infty^2(1-q)^2} {\abs{\Gamma_q\pp{\i\sqrt n\,\theta_{y_n}}}^2},\label{eq:qgamma-exact-double-angle}
\\
\abs{(q^xe^{\i(\theta_{y_n}\pm\theta_{z_n})};q)_\infty}^2 &=\frac{(q;q)_\infty^2(1-q)^{2-2x}}
{\displaystyle\abs{\Gamma_q\pp{x+\frac12\i\sqrt n\,\pp{\theta_{y_n}\pm\theta_{z_n}}}}^2}.
\label{eq:qgamma-exact-mixed-angle}
\end{align}

For fixed $y,z$, we have
$\limn\sqrt n\theta_{y_n} =\sqrt y, \limn \sqrt n\theta_{z_n}=\sqrt z$.
The convergence $\Gamma_q\to\Gamma$ is locally uniform on compact sets avoiding the poles of $\Gamma$.  Since $1-q\sim-\log q$, the exact identities \eqref{eq:qgamma-exact-single-angle},
\eqref{eq:qgamma-exact-double-angle}, and
\eqref{eq:qgamma-exact-mixed-angle} give
\eqref{e:qutheta}, \eqref{e:2theta}, and \eqref{e:plustheta}.
When $y=z$ and the minus sign is used, the additional assumption
ensures that the limiting Gamma factor is finite.

For the factor with a negative parameter, direct multiplication
gives
\[
\frac{\abs{(-q^xe^{\i\theta_{y_n}};q)_\infty}^2}
{(-q^x;q)_\infty^2}
=\prod_{k=0}^\infty
\pp{1-\frac{yq^{x+k}}{n(1+q^{x+k})^2}}\to 1,
\]
as $n\to\infty$. Indeed, this follows from the observation that 
\[
\sif k0 \frac{yq^{x+k}}{n(1+q^{x+k})^2} \le \frac{y}{n}\sum_{k=0}^\infty q^{x+k}
=\frac{yq^x}{n(1-q)}\to 0 \mbox{ as } n\to\infty. 
\]
Now, \eqref{e:quthetaneg} follows from \eqref{e:quneg}.
\end{proof}

\begin{lemma}\label{lem:uniform-q-product-bounds}
~
\begin{enumerate}[(i)]
\item For each fixed $a\in\mathbb R\setminus\mathbb Z_{\leq0}$, $\epsilon\in(0,1)$, 
\begin{equation}\label{eq:uniform-qgamma-upper}
\sup_{q\in[1-\epsilon,1],r\in\R} |\Gamma_q(a+\mathrm ir)|<\infty.
\end{equation}
For each fixed $a>0$, $\epsilon\in(0,1)$, there exists $C>0$ such that
\begin{equation}\label{eq:uniform-qgamma-inverse}
\sup_{q\in[1-\epsilon,1]} |\Gamma_q(a+\mathrm ir)|^{-2}
 \leq C e^{C|r|}, \mfa r\in\R.
\end{equation}
For each $\epsilon\in(0,1)$, there exists a constant $C>0$ such that
\equh\label{eq:BKW}
\sup_{q\in[1-\epsilon,1]}|\Gamma_q(\i r)|^{-2} \le  Cr^2e^{\pi |r|}, \mfa r\ne 0.
\eque

\item For $q=e^{-\kappa}\in(0,1)$ and $x\in\R$, define
\[
R_{q,x}(\theta):=\frac{(-q^x;q)_\infty^2}{|(-q^xe^{\mathrm i\theta};q)_\infty|^2},\quad 0\leq\theta<\pi.
\]
Then $R_{q,x}(\theta)\geq1$. For each fixed $x\in\R$ and $\epsilon\in(0,1)$, there exists $C>0$ such that
\begin{equation}\label{eq:R bound}
R_{q,x}(\theta)\leq e^{C\theta^2/\kappa},\quad \mfa 0\leq\theta\leq\frac{\pi}{2}, q\in[1-\epsilon,1).
\end{equation}
For every fixed $y\geq0$, with $q=q_n$ and $y_n=1-y/(2n)$, we have $R_{q_n,x}(\theta_{y_n})\to1$ as $n\to\infty$.\end{enumerate}
\end{lemma}
\begin{proof}
Recall that
\[
 \Gamma_q(z)=(1-q)^{1-z}
 \frac{(q;q)_\infty}{(q^z;q)_\infty}.
\]
For real $a\notin\mathbb Z_{\leq0}$ and $r$, the factorwise
inequality $|1-q^{a+k+\mathrm ir}|\geq|1-q^{a+k}|$ gives
\[
 |\Gamma_q(a+\mathrm ir)|\leq|\Gamma_q(a)|.
\]
The pointwise convergence of $\Gamma_q$ to $\Gamma$ proves
\eqref{eq:uniform-qgamma-upper}.

Write $q = e^{-\kappa}$. For $a>0$, we have
\[
 \frac{|\Gamma_q(a)|^2}{|\Gamma_q(a+\mathrm ir)|^2}
 =\prod_{k=0}^\infty
 \left(1+\frac{\sin^2(\kappa r/2)}
 {\sinh^2(\kappa(a+k)/2)}\right)\leq\prod_{k=0}^\infty
 \left(1+\frac{r^2}{(a+k)^2}\right)
 =\frac{|\Gamma(a)|^2}{|\Gamma(a+\mathrm ir)|^2}.
\]
Here we used $|\sin x|\leq|x|$, $\sinh x\geq x$ for $x\geq0$, and the classical Gamma product identity. Since $|\Gamma(a+\i r)|^{-2}\sim (2\pi)\inv |r|^{1-2a}e^{\pi|r|}$ as $|r|\to\infty$ for all $a>0$, the claimed \eqref{eq:uniform-qgamma-inverse} now follows.

For $a=0$, the product comparison above with $a=1$ gives
\[
 |\Gamma_q(1+\mathrm ir)|^{-2}
 \leq\prod_{k=1}^\infty\left(1+\frac{r^2}{k^2}\right)
 =\frac{\sinh(\pi r)}{\pi r}\leq e^{\pi r},
 \quad r>0.
\]
Using the recurrence for $\Gamma_q$ and $\kappa r=2\theta$,
we obtain
\[
 |\Gamma_q(\mathrm ir)|^{-2}  =\frac{4\sin^2\theta}{(1-q)^2}  |\Gamma_q(1+\mathrm ir)|^{-2} \leq C\kappa^{-2}\sin^2\theta\,e^{\pi r} \le Cr^2e^{\pi r},
\]
where in the first inequality we used $\kappa/(1-q)$ over $[1-\epsilon,1]$ is uniformly bounded, and in the last step we used $|\sin\theta|\le \theta = \kappa r/2$.

We next prove (ii). Direct factorization gives
\begin{equation}\label{eq:uniform-negative-product-identity}
R_{q,x}(\theta)=\prod_{k=0}^\infty\left(1-\frac{\sin^2(\theta/2)}{\cosh^2(\kappa(x+k)/2)}\right)^{-1}.
\end{equation}
Each factor is at least $1$. If $0\leq\theta\leq\pi/2$, the quantity subtracted in each factor is at most $1/2$. Thus $-\log(1-z)\leq2z$ for $z\in[0,1/2]$ gives
\[
0\leq\log R_{q,x}(\theta)\leq2\sin^2(\theta/2)\sum_{k=0}^\infty\frac1{\cosh^2(\kappa(x+k)/2)}.
\]
Moreover,
\[
\sum_{k=0}^\infty\frac1{\cosh^2(\kappa(x+k)/2)}
=\sum_{k=0}^\infty\frac{4q^{x+k}}{(1+q^{x+k})^2}
\leq\frac{4q^x}{1-q}\leq\frac{C}{\kappa},
\]
where $C$ depends on $x$. Using $\sin(\theta/2)\leq\theta/2$ proves  \eqref{eq:R bound}.

Finally, $\sin^2(\theta_{y_n}/2)=y/(4n)$ and $\kappa_n=2/\sqrt n$. For fixed $y$, we have $\theta_{y_n}\leq\pi/2$ for all sufficiently large $n$, and the preceding estimates yield
\[
0\leq\log R_{q_n,x}(\theta_{y_n})\leq\frac{Cy}{n\kappa_n}=O(n^{-1/2})\to0.
\]
This proves the stated limit.
\end{proof}

We next collect a few results concerning the limits of $\wt\pi\topp n$ and $\wt P\topp n$.  Recall that we let $\rho(q^x) = (q^x+q^{-x})/2$, and set $\wt\rho_n(q^x) = 2n(1-\rho(q^x)) \to -4x^2$. In particular, we set 
\begin{align*}
u\topp n_{t,j} &:= \wt\rho_n(q^{u+t/2+j}) \to  u_{t,j} = -4\pp{u+\frac t2+j}^2,\\
v\topp n_{t,k} & :=\wt\rho_n(q^{v-t/2+k}) \to  v_{t,k} = -4\pp{v-\frac t2+k}^2,
\end{align*} 
as $n\to\infty$, for $j\in \llbracket 0,j_{u,t}\rrbracket $ and $k\in \llbracket 0,k_{v,t}\rrbracket$.
The limits are exactly the candidate atomic locations introduced in Section \ref{sec:CDH}.

\begin{lemma}\label{lem:marginal}
Let $(u,v)\in\mathcal D\cap\mathbb R^2$ satisfy
$u+v+2\notin\mathbb Z_{\leq0}$, and fix
$t\in\mathbb T_{u,v}\cup\{0\}$. 
For each fixed $y>0$, 
\begin{align}\label{eq:pi pointwise}
\limn n^{u+v}\wt\pi_t^{(n),c}(y)
&=
\Gamma(u+v+2)\mathfrak p_t^{u,v,c}(y),\\
\limn n^{u+v}\wt \pi_t^{(n),d}\pp{u\topp n_{t,j}} &= \Gamma(u+v+2)\mathfrak p_t^{u,v,d}(u_{t,j}),\quad j=0,\dots,j_{u,t},\label{eq:pi pointwise u}\\
\limn n^{u+v}\wt \pi_t^{(n),d}\pp{v\topp n_{t,k}}& = \Gamma(u+v+2)\mathfrak p_t^{u,v,d}(v_{t,k}),\quad k=0,\dots,k_{v,t}.\label{eq:pi pointwise v}
\end{align}
Moreover, with
\[
 J_n(y):=\left(1-\frac{y}{4n}\right)^{-1/2}, 
\]
there exists $C>0$ depending only on
$u,v,t$, such that
\begin{equation}\label{eq:marginal-uniform-bound}
|\wt\pi_t^{(n),c}(y)|
 \leq C n^{-(u+v)}\sqrt y\,e^{C\sqrt y}J_n(y),
 \quad  \mfa y\in(0,2n],
\end{equation}
for all $n$ large enough. 
\end{lemma}

\begin{proof}
For $y\in[0,4n]$, the identity
$\sqrt y=2\sqrt n\sin(\theta_{y_n}/2)$
gives
\begin{equation}\label{eq:marginal-angle-comparison}
\sqrt y\leq\sqrt n\,\theta_{y_n}
\leq\frac{\pi}{2}\sqrt y.
\end{equation}
We now rewrite the continuous marginal density. By formula of the Askey--Wilson signed measures, after the change of variables (recall \eqref{eq:wt pi}) we have
\begin{align}\label{eq:wt pi^c}
\wt\pi_t^{(n),c}(y)
& =\frac{1}
{\displaystyle 4\pi n
 \sqrt{y/n}\sqrt{1-y/(4n)}}
\\
 & \quad\times
\frac{(q,-q^{1+v-t},q^{u+v},-q^{1+v},-q^{1+u},q^2,-q^{1+u+t};q)_\infty}{(q^{u+v+2};q)_\infty}
\nonumber
\\
& \quad\times
\frac{\abs{(e^{2\i\theta_{y_n}};q)_\infty}^2}
{\displaystyle\abs{
 (q^{v-t/2}e^{\i\theta_{y_n}},
  -q^{1-t/2}e^{-\i\theta_{y_n}},
  q^{u+t/2}e^{\i\theta_{y_n}},
  -q^{1+t/2}e^{\i\theta_{y_n}};q)_\infty
}^2}.\label{eq:3rd}
\end{align}
The factor
$(q^{u+v};q)_\infty$ makes this density identically zero when
$u+v\in\{0,-1\}$. Assume henceforth that
$u+v\notin\{0,-1\}$. 

The expression in \eqref{eq:3rd} equals (recalling \eqref{eq:qgamma-exact-single-angle} and \eqref{eq:qgamma-exact-double-angle})
\begin{multline*}
\frac{(1-q)^{2(u+v)-2}}{(q;q)_\infty^2}\times \frac{R_{q,1-t/2}(\theta_{y_n})
R_{q,1+ t/2}(\theta_{y_n})}{\abs{\Gamma_q\pp{\i\sqrt n\,\theta_{y_n}}}^2}
\\\times
\frac{\displaystyle
\abs{\Gamma_q\pp{v-\frac t2+\frac12\i\sqrt n\,\theta_{y_n},u+\frac t2+\frac12\i\sqrt n\,\theta_{y_n}}}^2}
{(-q^{1- t/2};q)_\infty^2
(-q^{1+t/2};q)_\infty^2}.
\end{multline*}
Here and below, we write $\Gamma_q(a,b) = \Gamma_q(a)\Gamma_q(b)$. 
Notice also
\[
\frac{(q,q^{u+v},q^2;q)_\infty}
{(q^{u+v+2};q)_\infty}
=(q;q)_\infty^2(1-q)
\frac{\Gamma_q(u+v+2)}{\Gamma_q(u+v)}.
\]
Combining these identities gives the following exact formula
for $0<y<4n$:
\begin{align}\label{eq:bounded1}
\wt\pi_t^{(n),c}(y)
&=\frac{\pp{\sqrt n(1-q)}^{2(u+v)-1}}
{\displaystyle 4\pi}
\frac{\Gamma_q(u+v+2)}{\Gamma_q(u+v)}
\frac{\displaystyle
(-q^{1+v-t},-q^{1+v},-q^{1+u},-q^{1+u+t};q)_\infty}
{\displaystyle
(-q^{1-t/2};q)_\infty^2
(-q^{1+ t/2};q)_\infty^2}
\\
&\quad\times
\abs{\Gamma_q\pp{v-\frac t2+\frac12\i\sqrt n\,\theta_{y_n},
u+\frac t2+\frac12\i\sqrt n\,\theta_{y_n}}}^2\label{eq:bounded2}
\\
&\quad\times
\frac{n^{-(u+v)}}{\sqrt y\sqrt{1-y/(4n)}}\times \frac{R_{q,1-t/2}(\theta_{y_n})
R_{q,1+ t/2}(\theta_{y_n})}{\abs{\Gamma_q\pp{\i\sqrt n\,\theta_{y_n}}}^2}.\label{eq:increasing}
\end{align}
By Lemmas~\ref{lem:q1} and
\ref{lem:q2}, we have \eqref{eq:pi pointwise}.

Moreover, the expressions on the right-hand side of \eqref{eq:bounded1} and in \eqref{eq:bounded2} are uniformly bounded in $n$. 
Indeed, Lemma~\ref{lem:q1} gives
\[
\frac{(-q^{1+v-t},-q^{1+v},-q^{1+u},
        -q^{1+u+t};q)_\infty}
     {(-q^{1-t/2};q)_\infty^2
      (-q^{1+t/2};q)_\infty^2}
\longrightarrow 2^{-2(u+v)}.
\]
Notice also $\limn\sqrt n(1-q) = 2$, and $v-t/2,u+t/2\notin\Z_{\leq0}$. Thus \eqref{eq:uniform-qgamma-upper} yields
\[
\abs{\Gamma_q\pp{v-\frac t2+
      \frac12\i\sqrt n\,\theta_{y_n},u+\frac t2+
      \frac12\i\sqrt n\,\theta_{y_n}}}^2
\leq C,
\]
uniformly in $n$ and $y\in[0,4n]$.

Next, applying \eqref{eq:R bound} and \eqref{eq:BKW}  gives (here we need to restrict $y\in[0,2n]$)
\[
\frac{R_{q,1-t/2}(\theta_{y_n})
      R_{q,1+t/2}(\theta_{y_n})}
     {\abs{\Gamma_q\pp{\i\sqrt n\,\theta_{y_n}}}^2}
\leq
Cn\theta_{y_n}^2e^{C\sqrt n\,\theta_{y_n}} \le Cye^{C\sqrt y},
\]
where in the last step we recalled \eqref{eq:marginal-angle-comparison}.
In summary, we have established 
\[
\abs{\wt\pi_t^{(n),c}(y)}
\leq Cn^{-(u+v)}
\frac{y
      e^{C\sqrt y}}
     {\displaystyle\sqrt y\sqrt{1-y/(4n)}}
\leq Cn^{-(u+v)}
\sqrt y\,e^{C\sqrt y}J_n(y)
,
\quad\mfa y\in[0,2n].
\]
This proves \eqref{eq:marginal-uniform-bound}.

The calculation of \eqref{eq:pi pointwise u} and \eqref{eq:pi pointwise v} is straightforward; notice that we do not need to establish a bound as for $\wt\pi_t^{(n),c}$ here. The details are omitted. 
\end{proof}

\begin{lemma}\label{lem:transitional}
    Let $(u,v)\in\mathcal D\cap\mathbb R^2$ and let $s,t\in\mathbb T_{u,v}$ satisfy $s<t<2v$. Then
\begin{align}
\lim_{n\to\infty}
\wt P^{(n),c,c}_{s,t} \pp{x,y}& = \mathfrak p_{s,t}^{u,v,c,c} \pp{x,y},\nonumber
\\
\lim_{n\to\infty}
\wt P^{(n),d,c}_{s,t} \pp{u^{(n)}_{s,j},y}& = \mathfrak p_{s,t}^{u,v,d,c} \pp{u_{s,j},y}, \quad j\in\llbracket0,j_{u,s}\rrbracket,\nonumber
\\
\lim_{n\to\infty}
\wt P^{(n),d,d}_{s,t} \pp{u^{(n)}_{s,j},u^{(n)}_{t,j'}} &= \mathfrak p_{s,t}^{u,v,d,d} \pp{u_{s,j},u_{t,j'}}, \quad j\in\llbracket 0, j_{u,s}\rrbracket ,\quad j'\in \llbracket j,  j_{u,t}\rrbracket.\label{eq:wt P d,d conv}
\end{align}
Moreover, there exists $C>0$ depending only on $u,v,s,t$, such that
\begin{align}
\left|\wt P_{s,t}^{(n),c,c}(x,y)\right|
&\leq C\sqrt y\,e^{C(\sqrt x+\sqrt y)}J_n(y),
&&0<x,y\leq2n, \label{eq:transition-cc-uniform}\\
\left|\wt P_{s,t}^{(n),d,c}
       \bigl(u_{s,j}^{(n)},y\bigr)\right|
&\leq C\sqrt y\,e^{C\sqrt y}J_n(y),
&&0<y\leq2n,\quad
j\in\llbracket0,j_{u,s}\rrbracket,
\label{eq:transition-dc-uniform}
\end{align}
for all $n$ large enough.
\end{lemma}

\begin{remark}\label{rem:DCT}
For the dominated convergence arguments below, we establish pointwise convergence and uniform upper bounds separately. Dominated convergence was also needed in two closely related recent works involving $q$-Pochhammer asymptotics as $q\uparrow1$, and our treatment is simpler here. In particular, we avoid using the quantitative error estimates obtained from the $q$-Pochhammer asymptotics in \citet[Proposition~2.3, Lemmas~8.3 and 8.5]{corwin24stationary}, and we also do not require the exponential decay bounds for the $q$-Gamma function established in \citet[Lemmas~5.4 and~5.6]{bryc25limits}. Our inverse $q$-Gamma bound \eqref{eq:BKW}, however, is of the same form as \cite[Lemma~5.7]{bryc25limits}.

The elementary product estimates in Lemma~\ref{lem:uniform-q-product-bounds} yield the bounds in Lemmas~\ref{lem:marginal} and~\ref{lem:transitional} on the region where every continuous coordinate belongs to $(0,2n]$. The key here is to restrict every coordinate $y_k$ to this interval (more precisely, to $(-\infty,2n]$ but $(-\infty,0)$ is easy to control too) in the integral representation of the finite-dimensional Laplace transform. When any coordinate $y_k$ exceeds $2n$, the Laplace weight is bounded in absolute value by $Ce^{-an}$ for some $a>0$. As shown below, the integral of the absolute value of the remaining factors grows at most subexponentially in $n$, so the corresponding contribution vanishes as $n\to\infty$.
\end{remark}

\begin{proof}
The pointwise convergence results  follow from straightforward calculations. We only prove the uniform control.  
For $0<y\leq2n$, we have
$\theta_{y_n}\leq\pi/2$. Thus \eqref{eq:BKW}, \eqref{eq:R bound}, and \eqref{eq:marginal-angle-comparison}
give, for every fixed $x\in\R$,
\begin{equation}\label{eq:transition-angular-y}
\frac{R_{q,x}(\theta_{y_n})}
     {|\Gamma_q(\i\sqrt n\,\theta_{y_n})|^2}
\leq Cn\theta_{y_n}^2e^{C\sqrt n\,\theta_{y_n}}
\leq Cy e^{C\sqrt y}.
\end{equation}
For $\wt P_{s,t}^{(n),c,c}$, similarly as for $\wt\pi_t^{(n),c}$ we have
\begin{align}
\wt P_{s,t}^{(n),c,c}(x,y)
&=
\frac{J_n(y)}
 {4\pi\sqrt n(1-q)\sqrt y\,\Gamma_q(t-s)}
\frac{(-q^{v+1-t};q)_\infty
      (-q^{1-s/2};q)_\infty^2}
     {(-q^{v+1-s};q)_\infty
      (-q^{1-t/2};q)_\infty^2}
\label{eq:wt P^c,c}\\
&\quad\times
\left|\Gamma_q\left(
 v-\frac t2+\frac{\i\sqrt n\,\theta_{y_n}}2
,
 \frac{t-s}{2}
 +\frac{\i\sqrt n}{2}(\theta_{x_n}\pm\theta_{y_n})
 \right)
\right|^2
\label{eq:num Gamma}\\
&\quad\times
\frac1{\displaystyle\left|\Gamma_q\left(
 v-\frac s2+\frac{\i\sqrt n\,\theta_{x_n}}2
 \right)\right|^2}
\frac{R_{q,1-t/2}(\theta_{y_n})}
 {R_{q,1-s/2}(\theta_{x_n})
  |\Gamma_q(\i\sqrt n\,\theta_{y_n})|^2}.
\label{eq:transition-cc-exact-gamma}
\end{align}
Here, we write $\Gamma_q(a,b\pm c) =\Gamma_q(a)\Gamma_q(b+c)\Gamma_q(b-c)$. 
By Lemma~\ref{lem:q1}, the quotient of negative
$q$-Pochhammer products tends to $1$. Also
$\sqrt n(1-q)\to2$, and $\Gamma_q(t-s)$ is bounded
away from zero. The Gamma factors in \eqref{eq:num Gamma} are uniformly bounded thanks to \eqref{eq:uniform-qgamma-upper}. Since $v-s/2>0$,
\eqref{eq:uniform-qgamma-inverse} gives
\[
\left|\Gamma_q\left(
v-\frac s2+\frac{\i\sqrt n\,\theta_{x_n}}2
\right)\right|^{-2}
\leq Ce^{C\sqrt n\,\theta_{x_n}}
\leq Ce^{C\sqrt x}.
\]
Finally, $R_{q,1-s/2}(\theta_{x_n})\geq1$.
Applying \eqref{eq:transition-angular-y} to
\eqref{eq:transition-cc-exact-gamma} proves
\eqref{eq:transition-cc-uniform}.

Now fix $j\in\llbracket0,j_{u,s}\rrbracket$. If $u+v+j\in\Z_{\leq0}$, then $(q^{u+v+j};q)_\infty=0$, so the discrete-to-continuous density vanishes. Otherwise, we have
\begin{align}
\wt P_{s,t}^{(n),d,c}
       \bigl(u_{s,j}^{(n)},y\bigr)
&=
\frac{J_n(y)}
 {4\pi\sqrt n(1-q)\sqrt y\,
  \Gamma_q(u+v+j,v-u-j-s,t-s)}
\nonumber\\
&\quad\times
\frac{(-q^{v+1-t},-q^{1+u+j},
             -q^{1-u-j-s};q)_\infty}
     {(-q^{v+1-s};q)_\infty
      (-q^{1-t/2};q)_\infty^2}
\label{eq:wt P^d,c}\\
&\quad\times
\left|
\Gamma_q\left(v-\frac t2+
              \frac{\i\sqrt n\,\theta_{y_n}}2,u+j+\frac t2+
              \frac{\i\sqrt n\,\theta_{y_n}}2,\frac t2-u-j-s+
              \frac{\i\sqrt n\,\theta_{y_n}}2\right)
\right|^2\nonumber\\
&
\quad \times \frac{R_{q,1-t/2}(\theta_{y_n})}
     {|\Gamma_q(\i\sqrt n\,\theta_{y_n})|^2}.
\label{eq:transition-dc-exact-gamma}
\end{align}
One checks readily again that the terms to be controlled are $J_n(y)/(\sqrt n(1-q))$ and the  last factor in \eqref{eq:transition-dc-exact-gamma}. Applying \eqref{eq:transition-angular-y} to \eqref{eq:transition-dc-exact-gamma} proves \eqref{eq:transition-dc-uniform}. 

It remains to prove \eqref{eq:wt P d,d conv}.
Fix $j\in\llbracket0,j_{u,s}\rrbracket$ and $j'\in\llbracket j,j_{u,t}\rrbracket$. The four Askey--Wilson parameters of the transition kernel starting from $u_{s,j}^{(n)}$ are
\[
\pp{q^{v-t/2},-q^{1-t/2},q^{u+j+t/2},q^{t/2-u-j-s}}.
\]
The target $u_{t,j'}^{(n)}$ corresponds to the $(j'-j)$-th atom associated with the third parameter. Applying \eqref{eq:AW atoms 0}--\eqref{eq:AW atoms j} and the $q$-Gamma identity gives
\begin{align}
\wt P_{s,t}^{(n),d,d}\pp{u_{s,j}\topp n,u_{t,j'}\topp n}
&= \frac{\Gamma_q(v-u-j-t,-2u-2j-s)}{\Gamma_q(v-u-j-s,-2u-2j-t)}
\frac{(-q^{v+1-t},-q^{1-u-j-s};q)_\infty}{(-q^{v+1-s},-q^{1-u-j-t};q)_\infty}\nonumber\\
&\quad\times \frac{(q^{t+2u+2j},q^{u+v+j},-q^{1+u+j},q^{t-s};q)_{j'-j}}{(q,q^{1+t+j+u-v},-q^{t+u+j},q^{1+2u+2j+s};q)_{j'-j}}
\frac{1-q^{t+2u+2j'}}{1-q^{t+2u+2j}}(-q^{s-v})^{j'-j}.
\label{eq:wt P^d,d}
\end{align}
By Lemma~\ref{lem:q1}, the quotient of negative infinite $q$-Pochhammer products tends to $1$. Using also
\[
\lim_{q\uparrow1}\frac{(q^a;q)_m}{(1-q)^m}=[a]_m,\quad \lim_{q\uparrow1}(-q^a;q)_m=2^m,\quad a\in\R,\quad m\in\Z_{\ge0},
\]
we obtain
\begin{align*}
\limn\wt P_{s,t}^{(n),d,d}\pp{u_{s,j}\topp n,u_{t,j'}\topp n}
&= \frac{\Gamma(v-u-j-t,-2u-2j-s)}{\Gamma(v-u-j-s,-2u-2j-t)}\\
&\quad\times \frac{[t+2u+2j,u+v+j,t-s]_{j'-j}}{[1,1+t+j+u-v,1+2u+2j+s]_{j'-j}}
\frac{t+2u+2j'}{t+2u+2j}(-1)^{j'-j}\\
&= \mathfrak p_{s,t}^{u,v,d,d}(u_{s,j},u_{t,j'}).
\end{align*}
For $j'\in\llbracket0,j_{u,t}\rrbracket$ with $j'<j$, both the prelimit and limiting masses are zero.\end{proof}

\subsection{Convergence of the denominator}
\label{sec:denomconv}
Recall \[
Z_n  =\int_\R\pp{1-\frac y{4n}}^n\wt\pi_0\topp n(\d y).
\]
\begin{lemma}\label{lem:Zn}
Let $(u,v)\in\mathcal D\cap\mathbb R^2$ satisfy
$u+v+2\notin\mathbb Z_{\leq0}$. Then
\[
\lim_{n\to\infty}n^{u+v}Z_n
=\Gamma(u+v+2)\int_{\mathbb R}e^{-y/4}
  \mathfrak p_0^{u,v}(\d y)
=\Gamma(u+v+2)K_{u,v}.
\]
\end{lemma}

\begin{proof}
Since $B_n=D_n=-q$, the measure $\wt\pi_0^{(n)}$ has no
atoms larger than $4n$. The positive atoms are caused by $C_n$ and possibly $A_n$.  Under the rescaling $y=2n(1-z)$, its atoms are precisely the candidate locations $u_{0,j}^{(n)}$ and $v_{0,k}^{(n)}$, all of which are negative. Consequently,
\begin{align*}
n^{u+v}Z_n
={}&\int_0^{2n}
 \left(1-\frac{y}{4n}\right)^n
 n^{u+v}\wt\pi_0^{(n),c}(y)\,\d y\\
&+\int_{2n}^{4n}
 \left(1-\frac{y}{4n}\right)^n
 n^{u+v}\wt\pi_0^{(n),c}(y)\,\d y +\int_{(-\infty,0)}
 \left(1-\frac{y}{4n}\right)^n
 n^{u+v}\wt\pi_0^{(n)}(\d y).
\end{align*}

We first look at the integral over $[0,2n]$. Assume $u+v\notin\{0,-1\}$ (otherwise the density is zero).
On $0<y\leq2n$, Lemma~\ref{lem:marginal}
and $J_n(y)\leq\sqrt2$ imply
\[
\left|
\left(1-\frac{y}{4n}\right)^n
n^{u+v}\wt\pi_0^{(n),c}(y)
\right|
\leq C\sqrt y\,
   \exp\left(C\sqrt y-\frac y4\right).
\]
The dominated convergence theorem then implies that 
\begin{equation}\label{eq:denom-near-limit}
\lim_{n\to\infty}\int_0^{2n}
 \left(1-\frac{y}{4n}\right)^n
 n^{u+v}\wt\pi_0^{(n),c}(y)\,\d y
=
\Gamma(u+v+2)\int_0^\infty
 e^{-y/4}\mathfrak p_0^{u,v,c}(y)\,\d y.
\end{equation}

We next control $y\in(2n,4n)$. Write
$\theta=\theta_{y_n}\in(\pi/2,\pi)$. The product
representation \eqref{eq:uniform-negative-product-identity}
gives
\begin{align*}
\log R_{q,1}(\theta)&= \log \pp{\prodd m1\infty\frac{(1+q^m)^2}{|1+q^me^{\i \theta}|^2}}
\leq
\log \pp{\prodd m1\infty\frac{(1+q^m)^2}{|1-q^m|^2}}
\\
&=
2\sum_{m=1}^{\infty}
 \log\frac{1+e^{-2m/\sqrt n}}{1-e^{-2m/\sqrt n}}
\leq
\sqrt n\int_0^\infty
 \log\frac{1+e^{-r}}{1-e^{-r}}\,\d r
\leq C\sqrt n.
\end{align*}
The integral is finite. Assume again $u+v\notin\{0,-1\}$, and the factors in
\eqref{eq:bounded1} and \eqref{eq:bounded2} are uniformly bounded, by
Lemmas \ref{lem:q1} and \ref{lem:uniform-q-product-bounds}. Using
\eqref{eq:BKW}, $\theta\leq\pi$, and the preceding bound
on $R_{q,1}$ in \eqref{eq:increasing}, we obtain
\[
\left|n^{u+v}\wt\pi_0^{(n),c}(y)\right|
\leq C\sqrt n\,e^{C\sqrt n}J_n(y),
\quad 2n<y<4n.
\]
Since $(1-y/(4n))^n\le e^{-y/4}\le e^{-n/2}$ on this interval,  and
$\int_{2n}^{4n}J_n(y)\,\d y=O(n)$, it follows that
\begin{equation}\label{eq:denom-far-limit}
\left|\int_{2n}^{4n}
 \left(1-\frac{y}{4n}\right)^n
 n^{u+v}\wt\pi_0^{(n),c}(y)\,\d y\right|
\leq Cn^{3/2}e^{C\sqrt n}e^{-n/2}\to 0.
\end{equation}

It remains to consider the atoms. Note that the total number of atoms is finite and hence pointwise convergence suffices. By Lemma \ref{lem:marginal} and the fact that $u_{0,j}^{(n)}\to u_{0,j}$ and $v_{0,k}^{(n)}\to v_{0,k}$, and the fact that only a finite number of atoms are involved, we have 
\begin{equation}\label{eq:denom-atomic-limit}
\limn\int_{(-\infty,0)}
 \left(1-\frac{y}{4n}\right)^n
 n^{u+v}\wt\pi_0^{(n)}(\d y)
=
\Gamma(u+v+2)\int_{(-\infty,0)}
 e^{-y/4}\mathfrak p_0^{u,v}(\d y).
\end{equation}
Combining \eqref{eq:denom-near-limit},
\eqref{eq:denom-far-limit}, and
\eqref{eq:denom-atomic-limit} proves the lemma.
\end{proof}

\subsection{Convergence of the numerator}\label{sec:numconv}
 Define
\equh\label{eq:G}
G(\vv x,\vv c):=\frac14\sum_{k=1}^d s_k^2(x_k-x_{k-1}).
\eque
Note that 
\[
e^{G(\vvx,\vvc)} = \esp\exp\pp{-\summ j1d \frac{c_j}{\sqrt 2}\B_{x_j}}.
\]
\begin{proposition}\label{prop:Phin}
Let $(u,v)\in\mathcal D\cap\mathbb R^2$ satisfy
$u+v+2\notin\mathbb Z_{\leq0}$. Fix $d\in\mathbb Z_{>0}$,
$0=x_0<x_1<\cdots<x_d=1$, and assume $\vvc$ as before with $c_j>0, j=1,\dots,d$, and $\summ j1{d-1}c_j<1$. Assume 
\[
-s_d<\min\{0,2v\},
\]
$-s_k\in\mathbb T_{u,v}$ for $k\in\llbracket d\rrbracket$.
Then
\[
\lim_{n\to\infty}n^{u+v}\Phi^{(n)}(\vvx,\vvc)
=
\Gamma(u+v+2)e^{G(\vvx,\vvc)}
\Psi_{u,v}(\vvx,\vvc).
\]
\end{proposition}
\begin{proof}
We recall the path decomposition of the numerator in \eqref{eq:prelimit path decomp} here for convenience:
\[
\Phi^{(n)}
=
\Phi^{(n),-}+\Phi_{\emptyset}^{(n)}
+\sum_{\ell=1}^{j}\ \sum_{\vvr\in\mathbb S_\ell\topp i}
 \Phi_{\vvr}^{(n)}
+\sum_{\ell=j+1}^{d}\ \sum_{\vvr\in\mathbb S_\ell\topp{i-1}}
 \Phi_{\vvr}^{(n)}.
\]
The claimed result shall follow from the following:   for each admissible $\vvr = (r_1,\dots,r_\ell), \ell=1,\dots,d$, or $\vvr = \emptyset$, 
\begin{equation}\label{eq:numerator one path}
\lim_{n\to\infty}n^{u+v}\Phi_{\vvr}^{(n)}
=
\Gamma(u+v+2)e^{G(\vvx,\vvc)}H_{\vvr}(u),
\end{equation}
and
\[
\lim_{n\to\infty}n^{u+v}\Phi^{(n),-}=0.
\]
\bigskip

\noindent(i) We first  prove \eqref{eq:numerator one path} for each admissible $\vvr = (r_1,\dots,r_\ell), \ell=1,\dots,d$.
The case $\ell=d$ follows from a straightforward calculation (note that no integral is involved in this case) and the details are omitted.

Assume $\ell=1,\dots,d-1$ from now on. 
Recall that \[
W_n(\vvy)=\prod_{k=1}^d\pp{\frac{\cosh(s_k/\sqrt n)+1-y_k/(2n)}{2}}^{n_k-n_{k-1}},
\]
and $\Phi_\vvr\topp n = \int_{\Lambda_\vvr^{(d,n)}}W_n(\vv y)\wt\pi\topp n_{-s_1,\dots,-s_d}(\d\vvy)$
with
\[
\Lambda_\vvr^{(d,n)} = \ccbb{u_{-s_1,r_1}^{(n)}}\times\cdots
\times\ccbb{u_{-s_\ell,r_\ell}^{(n)}}
\times(0,4n)^{d-\ell}.
\]
It is straightforward to check that 
\[
\limn W_n(\vvy) = 
e^{G(\vvx,\vvc)}
\exp\left(
-\frac14\sum_{k=1}^d(x_k-x_{k-1})y_k
\right)=:W(\vvy), \vvy\in(0,\infty)^d.
\]
Then, to show \eqref{eq:numerator one path} it is equivalent to show
\equh\label{eq:one path'}
\limn n^{u+v}\int_{\Lambda_\vvr^{(d,n)}} W_n(\vvy)\wt\pi_{-s_1,\dots,-s_d}\topp n(\d\vvy) = \Gamma(u+v+2)\int_{\Lambda_\vvr\topp d}W(\vvy)\mu_{-s_1,\dots,-s_d}^{u,v}(\d\vvy).
\eque
Note that with $\vvr = (r_1,\dots,r_\ell)$ for some $\ell=1,\dots,d$, $\vvy\in\Lambda_\vvr^{(d,n)}$ implies that 
\[
y_k = u_{-s_k,r_k}\topp n \mbox{ for } k=1,\dots,\ell \qmand y_{k'}\in(0,4n) \mbox{ for } k'=\ell+1,\dots,d.
\]
To show \eqref{eq:one path'},  we shall apply the dominated convergence theorem to $(d-\ell)$-multiple integrals.  
Fix $\vvr = (r_1,\dots,r_\ell)$ for $\ell=1,\dots,d-1$. We first remark that the integral on the left-hand side can be expressed as a $(d-\ell)$-multiple integral over $(0,4n)$. Indeed, writing
$\vvy_{i:j} = (y_i,\dots,y_j)$, $\d\vvy = \d y_i\cdots\d y_j$, and introducing
\begin{align*}
\wt W_{\vvr,n}(\vvy_{\ell+1:d})& := W_n\pp{u_{-s_1,r_1}\topp n,\dots,u_{-s_\ell,r_\ell}\topp n, \vvy_{\ell+1:d}},\\
\wt W_\vvr(\vvy_{\ell+1:d}) &:=W\pp{u_{-s_1,r_1},\dots,u_{-s_\ell,r_\ell},\vvy_{\ell+1:d}},
\end{align*}
we have  \begin{multline}\label{eq:DCT}
n^{u+v}\int_{\Lambda_\vvr^{(d,n)}} W_n(\vvy)\wt\pi_{-s_1,\dots,-s_d}\topp n(\d\vvy) \\
= n^{u+v}\Pi_\vvr\topp n \int_{(0,4n)^{d-\ell}}\wt W_{\vvr,n}\pp{\vvy_{\ell+1:d}}\wt\pi^{(n),*}_{-s_\ell,\dots,-s_d}\pp{u\topp n_{-s_\ell,r_\ell},y_{\ell+1:d}}\d\vvy_{\ell+1:d},
\end{multline}
where 
\[
\Pi_\vvr\topp n :=\wt\pi_{-s_1}^{(n),d}\pp{u\topp n_{-s_1,r_1}}\prodd j1{\ell-1} \wt P_{-s_j,-s_{j+1}}^{(n),d,d}\pp{u\topp n_{-s_j,r_j},u\topp n_{-s_{j+1},r_{j+1}}},
\]
and
\equh\label{eq:pi}
\wt\pi^{(n),*}_{-s_\ell,\dots,-s_d}\pp{u\topp n_{-s_\ell,r_\ell},\vvy_{\ell+1:d}} = \wt P_{-s_\ell,-s_{\ell+1}}^{(n),d,c}\pp{u\topp n_{-s_\ell,r_\ell},y_{\ell+1}}\prodd j{\ell+1}{d-1}\wt P^{(n),c,c}_{-s_j,-s_{j+1}}(y_j,y_{j+1}).
\eque
One readily checks that
\equh\label{eq:weight conv}
\limn n^{u+v}\Pi_\vvr\topp n = \Gamma(u+v+2)\mathfrak p_{-s_1}^{u,v,d}(u_{-s_1,r_1})\prodd j1{\ell-1}\mathfrak p_{-s_j,-s_{j+1}}^{u,v,d,d}(u_{-s_j,r_j},u_{-s_{j+1},r_{j+1}}),
\eque
\[
\limn \wt W_{\vvr,n}\pp{\vvy_{\ell+1:d}} = \wt W_\vvr\pp{\vvy_{\ell+1:d}}.
\]
and
\begin{align*}
\limn\wt\pi^{(n),*}_{-s_\ell,\dots,-s_d}\pp{u\topp n_{-s_\ell,r_\ell},\vvy_{\ell+1:d}} &= \mathfrak p_{-s_\ell,-s_{\ell+1}}^{u,v,d,c}\pp{u_{-s_\ell,r_\ell},y_{\ell+1}}\prodd j{\ell+1}{d-1}\mathfrak p_{-s_j,-s_{j+1}}^{u,v,c,c}(y_j,y_{j+1})\\
&=:
\mu_{-s_\ell,\dots,-s_d}^{u,v,*}(u_{-s_\ell,r_\ell},\vv y_{\ell+1:d}). 
\end{align*}
The pointwise convergence of \eqref{eq:DCT} and hence \eqref{eq:one path'} now follows. 

It remains to show that the integrand is bounded by an integrable function. 
Again fix $\vvr$ as before. 
Set 
\[
\delta:=\min_{ k=\ell+1,\dots,d}(x_k-x_{k-1})>0.
\]
For the $(d-\ell)$-multiple integral on the right-hand side, we first decompose it into two with respectively domains 
\[
(0,2n)^{d-\ell} \qmand \Omega_n\topp{d-\ell}:=(0,4n)^{d-\ell}\setminus(0,2n)^{d-\ell}.
\] 
For
\[
\int_{(0,2n)^{d-\ell}}\wt W_{\vvr,n}\pp{\vvy_{\ell+1:d}}\wt\pi^{(n),*}_{-s_\ell,\dots,-s_d}\pp{u\topp n_{-s_\ell,r_\ell},\vvy_{\ell+1:d}}\d\vvy_{\ell+1:d},
\]
one readily checks that for all $\vvy_{\ell+1:d}\in(0,2n)^{d-\ell}$, we have 
\[
\wt W_{\vvr,n}(\vvy_{\ell+1:d})
\leq C\exp\left(
-\frac{\delta}{8}\sum_{k=\ell+1}^d y_k
\right).
\]
Moreover, by \eqref{eq:transition-cc-uniform} and
\eqref{eq:transition-dc-uniform}, we have
\[
\abs{\wt\pi^{(n),*}_{-s_\ell,\dots,-s_d}\spp{u\topp n_{-s_\ell,r_\ell},\vvy_{\ell+1:d}}}\le C\prodd j{\ell+1}{d} \sqrt{y_{j}}e^{C\sqrt{y_{j}}}J_n(y_j),
\]
Since $J_n(y_j)\leq\sqrt2$ on $(0,2n]$, we have
\begin{multline*}
\abs{\wt W_{\vvr,n}\pp{\vvy_{\ell+1:d}}\wt\pi^{(n),*}_{-s_\ell,\dots,-s_d}\pp{u\topp n_{-s_\ell,r_\ell},\vvy_{\ell+1:d}}\inddd{\vvy_{\ell+1:d}\in(0,2n)^{d-\ell}} }\\
\le C\prod_{j=\ell+1}^d
\sqrt{y_j}
\exp\left(C\sqrt{y_j}-\frac{\delta}{8}y_j\right),
\end{multline*}
and the upper bound is an integrable function over $(0,4n)^{d-\ell}$. 
This function is integrable on $(0,\infty)^{d-\ell}$. The dominated convergence theorem now yields
\begin{multline}\label{eq:numerator-one-path-low}
\limn \int_{(0,2n)^{d-\ell}}\wt W_{\vvr,n}\pp{\vvy_{\ell+1:d}}\wt\pi^{(n),*}_{-s_\ell,\dots,-s_d}\pp{u\topp n_{-s_\ell,r_\ell},\vvy_{\ell+1:d}}\d\vvy_{\ell+1:d}\\
= \int_{(0,\infty)^{d-\ell}}\wt W_\vvr(\vvy_{\ell+1:d})\mu_{-s_{\ell},\dots,-s_d}^{u,v,*}(u_{-s_\ell,r_\ell},\d\vvy_{\ell+1:d}).
\end{multline}

Next, for $\vvy_{\ell+1:d}\in \Omega_n\topp{d-\ell}:=(0,4n)^{d-\ell}\setminus(0,2n)^{d-\ell}$, assume $k\in\{\ell+1,\dots,d\}$ is such that $y_k\ge 2n$. Then, 
\equh\label{eq:away}
\frac{\cosh(s_k/\sqrt n)+1-y_k/(2n)}{2}
\leq \frac12+\frac Cn.
\eque
One can show that for some $a>0$, 
\[
\abs{\wt W_{\vvr,n}(\vvy_{\ell+1:d})}\leq Ce^{-an}.
\]
This time, to bound $\wt\pi^{(n),*}_{-s_\ell,\dots,-s_d}(u_{-s_\ell,r_\ell}\topp n,\vv y_{\ell+1:d})$ in \eqref{eq:pi}, we need the following facts: $\wt P_{-s_\ell,-s_{\ell+1}}^{(n)}(u\topp n_{-s_\ell,r_\ell},\cdot)$ is a signed measure that integrates to one,  the density function of the continuous part (over $(0,4n)$) has the same sign, there are a finite number of atoms, the mass on each atom converges (and hence is of order O(1)), and therefore the integral of the continuous part $\wt P_{-s_\ell,-s_{\ell+1}}^{(n),d,c}(u_{-s_\ell,r_\ell}\topp n,\cdot)$ is also of order $O(1)$. Then, we can simply bound the integral of the function $|\wt\pi^{(n),*}_{-s_\ell,\dots,-s_d}(u_{-s_\ell,r_\ell}\topp n,\vv y_{\ell+1:d})|$ by a constant. 
That is,
\begin{multline*}
\int_{(0,4n)^{d-\ell}\setminus (0,2n)^{d-\ell}}\abs{\wt W_{\vvr,n}\pp{\vvy_{\ell+1:d}}\wt\pi^{(n),*}_{-s_\ell,\dots,-s_d}\pp{u\topp n_{-s_\ell,r_\ell},\vvy_{\ell+1:d}}}\d\vvy_{\ell+1:d}\\
\le Ce^{-a n}\int_{(0,4n)^{d-\ell}\setminus (0,2n)^{d-\ell}}\abs{\wt\pi^{(n),*}_{-s_\ell,\dots,-s_d}\pp{u\topp n_{-s_\ell,r_\ell},\vvy_{\ell+1:d}}}\d\vvy_{\ell+1:d}\le Ce^{-an} \to 0,
\end{multline*}
as $n\to\infty$. Combining this with \eqref{eq:weight conv}, \eqref{eq:numerator-one-path-low}, we have proved \eqref{eq:one path'}. \bigskip

\noindent(ii) The case
\[
\lim_{n\to\infty}n^{u+v}\Phi_{\emptyset}^{(n)}
=
\Gamma(u+v+2)e^{G(\vvx,\vvc)}H_{\emptyset}(u),
\]
follows by a similar argument. There are no atoms in the path, and hence the corresponding convergence of integral formula concerns $d$-multiple integrals. Again we start by decomposing the domain into $(0,2n)^d$ and $(0,4n)^d\setminus(0,2n)^d$. We omit details. \bigskip

\noindent(iii) It remains to show 
\equh\label{eq:negative atoms}
n^{u+v}\Phi^{(n),-} \equiv n^{u+v} \int_{\R^d\setminus(-\infty,4n)^d}W_n(\vvy)\wt\pi_{-s_1,\dots,-s_d}\topp n(\d\vvy)\to 0
\eque
as $n\to\infty$. This case can be further decomposed into the sum of integrals such that each integral is over a path that starts at an atom  corresponding to $\mathsf n\topp n_{\wt s_{n,1},j}$ (before the change of variables) for some $j$ (recall \eqref{eq:mathfrak d}). 
We write
\[
\wt{\mathsf n}_{-s_j,k}\topp n := 2n(1-\mathsf n\topp n_{\wt s_{n,j},k})>4n, \quad k=0,\dots,\ell_{-s_j}. 
\]
The control \eqref{eq:away} applies again with $y_k = \wt{\mathsf n}_{-s_j,k}\topp n$ and this time we have
\[
W_n\pp{\wt{\mathsf n}_{-s_1,r_1}\topp n,\dots,\wt {\mathsf n}_{-s_\ell,r_\ell}\topp n,\vvy_{\ell+1:d}}\le C e^{-an}.
\]
Now for the measure $\wt\pi_{-s_1,\dots,-s_d}\topp n$, the marginal measure (necessarily on an atom $\wt{\mathsf n}_{-s_1,j}$) is of order $n^{-2}$ (when the corresponding weight is non-zero; otherwise there is nothing to prove), and one can check that the integration of the transitional laws can be bounded by a constant in a similar way as before.  Therefore, \eqref{eq:negative atoms} follows. We omit the details. 
\end{proof}

\begin{proof}[Proof of Theorem~\ref{thm:main}]
First assume that $(u,v),\vvx,\vvc$ satisfy the assumptions of Proposition~\ref{prop:Phin}. By Theorem~\ref{thm:AC}, $K_{u,v}=\wt K_{u,v}>0$, and hence $\Gamma(u+v+2)K_{u,v}\ne0$. Combining \eqref{eq:varphi}, Lemma~\ref{lem:Zn}, and Proposition~\ref{prop:Phin}, we obtain
\begin{align*}
\limn \esp\exp\pp{-\summ j1d c_jh_n\topp{u,v}(x_j)}
&=e^{G(\vvx,\vvc)}\frac{\Psi_{u,v}(\vvx,\vvc)}{K_{u,v}}=e^{G(\vvx,\vvc)}\frac{\wt\Psi_{u,v}(\vvx,\vvc)}{\wt K_{u,v}}\\
&=\esp\exp\pp{-\summ j1d c_j\pp{\frac1{\sqrt2}\B_{x_j}+\zeta_{x_j}\topp{u,v}}},
\end{align*}
where the second equality follows from Theorem~\ref{thm:AC}, and the last equality follows from \eqref{eq:Laplace fdd} (recall the formula for $G$ in \eqref{eq:G} and the independence of $\B$ and $\zeta\topp{u,v}$).

Proposition~\ref{prop:sandwich} now extends this convergence to every $(u,v)\in\R^2$, provided that $c_1,\dots,c_d>0$, $\sum_{k=1}^{d-1}c_k<1$, and $c_d>-2v$. For each fixed $(u,v)$, the admissible set of $\vvc$ includes a nonempty open subset of $\R^d$. It is known that the convergence of finite-dimensional Laplace transforms for coefficients $\vvc$ from a nonempty open subset of $\R^d$ implies the convergence of the corresponding finite-dimensional distributions \citep[Appendix~A]{bryc19limit}.
\end{proof}

\begin{remark}\label{rem:proof for fan}
We essentially have a (simplified) proof for the convergence in the fan region. In addition to $-c_d<2v$, in the fan region one could impose further
\equh\label{eq:fan cond}
-2u<-s_1<\cdots<-s_d<2v.
\eque
This condition can always be achieved when $u>0$ and $ u+v>0$, and in this case the auxiliary measures for the numerator do not involve any atoms near 1 as $C_nq^{-s_j/2}<1$  (there might be a few near $-1$, which would be much easier to deal with). Under \eqref{eq:fan cond}, all the analysis concerning atoms associated to $C_n$ are no longer needed (since there are none), and the analysis of the remaining terms in the decomposition remains the same. This leads to the proof of convergence of the Laplace transform for $u>0,u+v>0$. For the case $u\le 0, u+v>0$ the proof follows by a particle-hole duality argument. The calculations involved this way are significantly simpler than the ones in \citet{corwin24stationary} at a few places; recall Remark \ref{rem:CK}. 
\end{remark}

\appendix

\section{Some identities in the fan region}\label{sec:identity fan}
We collect two results that can be read from the literature.
\begin{lemma}
For $(u,v)\in\mathcal D\cap\R^2$ with $u+v>0$, we have
\equh\label{eq:K frak K}
K_{u,v} =\mathfrak K_{u,v}.
\eque
\end{lemma}
\begin{proof}
For $u,v\in\R$ with $u+v>0$, we define
\[
C_{u,v}:=\int_\R\int_\R e^{-2vx-2uy}p_1(x,y)\d x\d y,
\]
where
\[
p_t(x,y) = \frac2\pi\int_0^\infty e^{-tw^2/4}K_{\i w}(e^{-x})K_{\i w}(e^{-y}) \frac{\d w}{|\Gamma(\i w)|^2}
\]
is the Yakubovich heat kernel, defined for $t>0$ and $x,y\in\R$. Here $K_\nu$ denotes the modified Bessel function of the second kind. The integral defining $C_{u,v}$ is finite whenever $u,v\in\R$ and $u+v>0$. Applying \cite[Eq.~(1.14), corrected arXiv version]{bryc22markov} with $\mathsf a=2v$, and $\mathsf c=2u$, and $\tau=1$, yields, for $u+v>0$,
\equh\label{eq:BK}
\mathfrak K_{u,v} = \frac{2^{1-2u-2v}}{\Gamma(u+v)} C_{u,v}.
\eque

The convention for the Yakubovich heat kernel in
\citep{bryc23markov} differs from the one used above. If
$p_t^{\mathrm{BKWW}}$ denotes the kernel used in that paper, then $p_{t/4}^{\mathrm{BKWW}}(-x,-y)=p_t(x,y)$.
Consequently, with $\mathsf a=2v$ and $\mathsf c=2u$, the constant
$C_{\mathsf a,\mathsf c}^{(1/4)}$ in the notation of
\citep{bryc23markov} equals $C_{u,v}$. 
With $\mathsf a=2v$, $\mathsf c=2u$, and $\tau=1/4$, our normalization
is related to that of \citep{bryc23markov} by $K_{u,v}=\mathfrak K_{2v,2u}^{(1/4)}/\Gamma(u+v+2)$.
Therefore, \cite[Proposition~5.1]{bryc23markov} and the preceding
conversion of the Yakubovich heat kernel imply
\[
K_{u,v} = \frac{(2u+2v)(2u+2v+2)} {2^{2u+2v+1}\Gamma(u+v+2)}C_{u,v}.
\]
Using $\Gamma(u+v+2)=(u+v)(u+v+1)\Gamma(u+v)$, for $u+v>0$, this simplifies to
\equh\label{eq:BKWW}
K_{u,v} = \frac{1}{2^{2u+2v-1}\Gamma(u+v)}C_{u,v}.
\eque
Comparing \eqref{eq:BK} and \eqref{eq:BKWW} proves \eqref{eq:K frak K} for the above parameters. The derivation of \eqref{eq:BK} uses the Hartman--Watson density, whereas \cite[Proposition~5.1]{bryc23markov} obtains \eqref{eq:BKWW} by comparing explicit normalizing constants. 
\end{proof}
\begin{lemma}
For $u,v\in\R$, such that $u>s_1/2, v>0$, we have
\[
\Psi_{u,v}(\vvx,\vvc)=\wt\Psi_{u,v}(\vvx,\vvc)
\]
\end{lemma}
\begin{proof}
Our convention on the auxiliary measures and the one taken by \citet{corwin24stationary} is up to a time reversal, that we first explain. 
Set $s_{d+1}:=0$ and let $\widetilde{\mathfrak p}_t^{u,v}$ and $\widetilde{\mathfrak p}_{s,t}^{u,v}$ denote, respectively, the marginal measures and transition kernels therein, extended to $t<2u$ and $s<t<2u$ as summarized in \citet[Appendix~C]{bryc23markov}, with $(\mathsf a,\mathsf c)=(2v,2u)$. For $0\le t\le s_1$, the marginal measures satisfy $\widetilde{\mathfrak p}_t^{u,v} =\Gamma(u+v+2)\mathfrak p_{-t}^{u,v}$. The joint measures satisfy the following identity on $\R^d$ in the variables $(y_1,\dots,y_d)$:
\equh\label{eq:time reversal}
{\mathfrak p}_{-s_1}^{u,v}(\d y_1)
\prodd k1{d-1}
{\mathfrak p}^{u,v}_{-s_k,-s_{k+1}}(y_k,\d y_{k+1})
=
\frac{1}{\Gamma(u+v+2)}
\int_{\R}
\widetilde{\mathfrak p}_0^{u,v}(\d y_{d+1})
\prod_{k=d,d-1,\dots,1}
\widetilde{\mathfrak p}^{u,v}_{s_{k+1},s_k}
(y_{k+1},\d y_k).
\eque
Since $u-s_1/2>0$ and $v>0$, the marginal measures and transition kernels at all the relevant times have only continuous parts.
For $0\le a<b\le s_1$, a direct comparison of the extended continuous densities in \citet[Appendix~C]{bryc23markov} gives the following equality of measures on $\R^2$ in the variables $(x,y)$, extending the equation (6.15) from \citet{corwin24stationary} beyond the time range considered there:
\equh\label{eq:one step reversal}
\mathfrak p_{-b}^{u,v}(\d y)\, \mathfrak p_{-b,-a}^{u,v}(y,\d x) = \mathfrak p_{-a}^{u,v}(\d x)\,\widetilde{\mathfrak p}_{a,b}^{u,v}(x,\d y).
\eque
Applying \eqref{eq:one step reversal} successively with $(a,b)=(s_{k+1},s_k)$, $k\in\llbracket d\rrbracket$, yields the following equality of measures on $\R^{d+1}$ in the variables $(y_1,\dots,y_{d+1})$:
\[
\mathfrak p_{-s_1}^{u,v}(\d y_1)\prod_{k=1}^{d}\mathfrak p_{-s_k,-s_{k+1}}^{u,v}(y_k,\d y_{k+1})=\mathfrak p_0^{u,v}(\d y_{d+1})\prod_{k=d,d-1,\dots,1}\widetilde{\mathfrak p}_{s_{k+1},s_k}^{u,v} (y_{k+1},\d y_k).
\] 
Using $\mathfrak p_0^{u,v}= \widetilde{\mathfrak p}_0^{u,v}/\Gamma(u+v+2)$ and integrating out $y_{d+1}$ proves
\eqref{eq:time reversal}.

Now we can prove the desired identity. 
Using \eqref{eq:time reversal}, we apply
\citet[Theorem~1.3]{bryc23markov} to express $\Psi_{u,v}(\vvx,\vvc)/K_{u,v}$ as a Laplace transform of increments of the Markov process defined there. We take $\mathsf a =2v,\mathsf c=2u,\tau=1/4$, and $t_k=x_k$. 
The marginal measure normalization above yields
\begin{align*}
\mathfrak K_{2v,2u}^{(1/4)} &= \int_{\R}e^{-y/4}\widetilde{\mathfrak p}_0^{u,v}(\d y) = \Gamma(u+v+2) \int_{\R}e^{-y/4}\mathfrak p_0^{u,v}(\d y) = \Gamma(u+v+2)K_{u,v}.
\end{align*}
With this normalization, we obtain
\[
\frac{\Psi_{u,v}(\vvx,\vvc)}{K_{u,v}} = \mathbb E\exp\left(-\sum_{k=1}^d c_k \bigl(Y^{\mathrm{BKWW}}_0-Y^{\mathrm{BKWW}}_{x_k/4}\bigr) \right) ,
\]
where $Y^{\mathrm{BKWW}}$ denotes the Markov process introduced in \citet[Theorem~1.1]{bryc23markov} again with parameters $\mathsf a=2v,\mathsf c=2u$, and $\tau=1/4$. Let
$Y^{\mathrm{BK}}$ denote the Markov process defined in
\citet[(1.8)--(1.12)]{bryc22markov}, with parameters
$\mathsf a=2v,\mathsf c=2u$, and $\tau=1$. 
The change of time scale and sign described immediately before
\citet[(1.7)]{bryc22markov} yields the equality in distribution
\[
\left(Y^{\mathrm{BKWW}}_0-Y^{\mathrm{BKWW}}_{x/4}\right)_{x\in[0,1]}
\stackrel{\mathrm d}=
\left(Y^{\mathrm{BK}}_x-Y^{\mathrm{BK}}_0\right)_{x\in[0,1]}.
\]
The reference Brownian motion in \citet{bryc22markov} is the standard Brownian motion scaled by $1/2$. Thus, with $(\mathsf a,\mathsf c,\tau)=(2v,2u,1)$, the Radon--Nikodym derivative in \citet[(1.6)]{bryc22markov} is exactly \eqref{eq:zeta-RN}. By \citet[Theorem~1.2]{bryc22markov},
\[
\left(Y^{\mathrm{BK}}_x-Y^{\mathrm{BK}}_0\right)_{x\in[0,1]}
\stackrel{\mathrm d}=
\left(\zeta_x\topp{u,v}\right)_{x\in[0,1]}.
\]
Consequently,
\[
\frac{\Psi_{u,v}(\vvx,\vvc)}{K_{u,v}} = \mathbb E \exp\left(-\sum_{k=1}^d c_k\zeta_{x_k}\topp{u,v}\right)  = \frac{\widetilde\Psi_{u,v}(\vvx,\vvc)}{\widetilde K_{u,v}}.
\]
Finally, \eqref{eq:K frak K} implies $K_{u,v}=\mathfrak K_{u,v}=\widetilde K_{u,v}>0$. Therefore, we conclude that $\Psi_{u,v}(\vvx,\vvc)=\widetilde\Psi_{u,v}(\vvx,\vvc)$.
\end{proof}

\bibliographystyle{apalike}
  \bibliography{references,references18}
\end{document}